\documentclass[letterpaper,reqno,10pt,twoside]{amsart}
\usepackage{amsmath,amsthm,amsfonts,amssymb,euscript,mathrsfs,graphics,color,latexsym,marginnote}
\usepackage[dvips]{graphicx}
\usepackage[margin=1in]{geometry}
\usepackage[toc,page]{appendix}
\usepackage{relsize}
\usepackage[shortlabels]{enumitem}
\usepackage[all]{xy}

\usepackage[dvipsnames]{xcolor}

\usepackage{hyperref}
\hypersetup{colorlinks=true, pdfstartview=FitV,linkcolor=blue!70!black,citecolor=red!70!black, urlcolor=green!60!black}
\definecolor{labelkey}{rgb}{0.6,0,0}

\makeatletter
\renewcommand \theequation {%
\ifnum \c@section>\z@ \@arabic\c@section.%
\fi\@arabic\c@equation} \@addtoreset{equation}{section}
\makeatother

\makeatletter
\DeclareFontFamily{OMX}{MnSymbolE}{}
\DeclareSymbolFont{MnLargeSymbols}{OMX}{MnSymbolE}{m}{n}
\SetSymbolFont{MnLargeSymbols}{bold}{OMX}{MnSymbolE}{b}{n}
\DeclareFontShape{OMX}{MnSymbolE}{m}{n}{
    <-6>  MnSymbolE5
   <6-7>  MnSymbolE6
   <7-8>  MnSymbolE7
   <8-9>  MnSymbolE8
   <9-10> MnSymbolE9
  <10-12> MnSymbolE10
  <12->   MnSymbolE12
}{}
\DeclareFontShape{OMX}{MnSymbolE}{b}{n}{
    <-6>  MnSymbolE-Bold5
   <6-7>  MnSymbolE-Bold6
   <7-8>  MnSymbolE-Bold7
   <8-9>  MnSymbolE-Bold8
   <9-10> MnSymbolE-Bold9
  <10-12> MnSymbolE-Bold10
  <12->   MnSymbolE-Bold12
}{}

\let\llangle\@undefined
\let\rrangle\@undefined
\DeclareMathDelimiter{\llangle}{\mathopen}%
    {MnLargeSymbols}{'164}{MnLargeSymbols}{'164}
\DeclareMathDelimiter{\rrangle}{\mathclose}%
    {MnLargeSymbols}{'171}{MnLargeSymbols}{'171}
\makeatother

\newtheorem{theorem}{Theorem}[section]
\newtheorem{corollary}{Corollary}[theorem]
\newtheorem{lemma}[theorem]{Lemma}
\newtheorem{proposition}[theorem]{Proposition}

\theoremstyle{definition}

\theoremstyle{remark}
\newtheorem{remark}{Remark}[section]

\DeclareMathOperator*{\esssup}{ess\,sup}

\providecommand{\abs}[1]{\left\vert#1\right\vert}
\providecommand{\babs}[1]{\big\vert#1\big\vert}

\providecommand{\nm}[1]{\left\Vert#1\right\Vert}
\providecommand{\bnm}[1]{\big\Vert#1\big\Vert}

\providecommand{\nnm}[1]{\left\vert\kern-0.25ex\left\vert\kern-0.25ex\left\vert#1\right\vert\kern-0.25ex\right\vert\kern-0.25ex\right\vert}
\providecommand{\bnnm}[1]{\big\vert\kern-0.25ex\big\vert\kern-0.25ex\big\vert#1\big\vert\kern-0.25ex\big\vert\kern-0.25ex\big\vert}

\providecommand{\br}[1]{\left\langle #1 \right\rangle}
\providecommand{\bbr}[1]{\big\langle #1 \big\rangle}
\providecommand{\brx}[1]{\left\langle #1 \right\rangle_{x}}
\providecommand{\bbrx}[1]{\big\langle #1 \big\rangle_{x}}
\providecommand{\brv}[1]{\left\langle #1 \right\rangle_v}

\providecommand{\brb}[2]{\left\langle #1 \right\rangle_{#2}}
\providecommand{\bbrb}[2]{\big\langle #1 \big\rangle_{#2}}

\providecommand{\xnm}[1]{\left\Vert#1\right\Vert_{X}}
\providecommand{\xnnm}[1]{\left\vert\kern-0.25ex\left\vert\kern-0.25ex\left\vert#1\right\vert\kern-0.25ex\right\vert\kern-0.25ex\right\vert_{X}}
\providecommand{\bnm}[1]{\left\Vert#1\right\Vert_{\text{BV}}}

\providecommand{\pnm}[2]{\left\Vert#1\right\Vert_{L^{#2}}}
\providecommand{\tnm}[1]{\left\Vert#1\right\Vert_{L^{2}}}
\providecommand{\btnm}[1]{\big\Vert#1\big\Vert_{L^{2}}}
\providecommand{\lnm}[1]{\left\Vert#1\right\Vert_{L^{\infty}}}
\providecommand{\lnmm}[1]{\left\Vert#1\right\Vert_{L^{\infty}_{\varrho,\vartheta}}}
\providecommand{\um}[1]{\left\Vert#1\right\Vert_{L^2_{\nu}}}

\providecommand{\pnms}[3]{\left\vert#1\right\vert_{L^{#2}_{#3}}}
\providecommand{\tnms}[2]{\left\vert#1\right\vert_{L^{2}_{#2}}}

\providecommand{\lnmms}[2]{\left\vert#1\right\vert_{L^{\infty}_{#2,\varrho,\vartheta}}}

\providecommand{\jnm}[1]{\left\Vert#1\right\Vert_{L^{\N}}}
\providecommand{\knm}[1]{\left\Vert#1\right\Vert_{L^{\frac{\N}{\N-1}}}}
\providecommand{\jnms}[2]{\left\vert#1\right\vert_{L^{\frac{2\N}{3}}_{#2}}}
\providecommand{\knms}[2]{\left\vert#1\right\vert_{L^{\frac{2\N}{2\N-3}}_{#2}}}

\providecommand{\dbr}[1]{\left\llangle #1 \right\rrangle}
\providecommand{\dbbr}[1]{\big\llangle #1 \big\rrangle}
\providecommand{\dbrx}[1]{\left\llangle #1 \right\rrangle_{tx}}
\providecommand{\dbbrx}[1]{\big\llangle #1 \big\rrangle_{tx}}
\providecommand{\dbrv}[1]{\left\llangle #1 \right\rrangle_{tv}}

\providecommand{\dbrb}[2]{\left\llangle #1 \right\rrangle_{#2}}
\providecommand{\dbbrb}[2]{\big\llangle #1 \big\rrangle_{#2}}

\providecommand{\pnnm}[2]{\left\vert\kern-0.25ex\left\vert\kern-0.25ex\left\vert#1\right\vert\kern-0.25ex\right\vert\kern-0.25ex\right\vert_{L^{#2}}}
\providecommand{\tnnm}[1]{\left\vert\kern-0.25ex\left\vert\kern-0.25ex\left\vert#1\right\vert\kern-0.25ex\right\vert\kern-0.25ex\right\vert_{L^{2}}}
\providecommand{\lnnm}[1]{\left\vert\kern-0.25ex\left\vert\kern-0.25ex\left\vert#1\right\vert\kern-0.25ex\right\vert\kern-0.25ex\right\vert_{L^{\infty}}}
\providecommand{\lnnmm}[1]{\left\vert\kern-0.25ex\left\vert\kern-0.25ex\left\vert#1\right\vert\kern-0.25ex\right\vert\kern-0.25ex\right\vert_{L^{\infty}_{\varrho,\vartheta}}}
\providecommand{\unm}[1]{\left\vert\kern-0.25ex\left\vert\kern-0.25ex\left\vert#1\right\vert\kern-0.25ex\right\vert\kern-0.25ex\right\vert_{L^2_{\nu}}}

\providecommand{\pnnms}[3]{\left\Vert#1\right\Vert_{L^{#2}_{#3}}}
\providecommand{\tnnms}[2]{\left\Vert#1\right\Vert_{L^{2}_{#2}}}

\providecommand{\lnnmms}[2]{\left\Vert#1\right\Vert_{L^{\infty}_{#2,\varrho,\vartheta}}}

\providecommand{\jnnm}[1]{\left\vert\kern-0.25ex\left\vert\kern-0.25ex\left\vert#1\right\vert\kern-0.25ex\right\vert\kern-0.25ex\right\vert_{L^{\N}}}
\providecommand{\knnm}[1]{\left\vert\kern-0.25ex\left\vert\kern-0.25ex\left\vert#1\right\vert\kern-0.25ex\right\vert\kern-0.25ex\right\vert_{L^{\frac{\N}{\N-1}}}}

\def\ud{\mathrm{d}}
\def\dt{\partial_t}
\def\p{\partial}
\def\ls{\lesssim}
\def\gs{\gtrsim}

\def\rt{\rightarrow}
\def\r{\mathbb{R}}
\def\no{\nonumber}
\def\ue{\mathrm{e}}

\def\ds{\displaystyle}

\def\R{\mathbb{R}}

\def\S{\mathbb{S}}
\def\rp{\R_+}

\def\e{\varepsilon}
\def\d{\delta}
\def\a{\mathscr{A}}
\def\ab{\overline{\a}}
\def\b{\mathscr{B}}
\def\bbb{\overline{\b}}
\def\k{\kappa}
\def\qq{Q^{\ast}}

\def\fs{\mathfrak{F}}
\def\fss{\mathfrak{f}}
\def\f{f}
\def\fb{f^B}

\def\pk{\mathbf{P}}
\def\ik{\mathbf{I}}
\def\lc{\mathcal{L}}
\def\li{\lc^{-1}}
\def\nk{\mathcal{N}}
\def\nnk{\nk^{\perp}}

\def\h{h}
\def\z{z}
\def\ss{S}

\def\bb{\mathbf{b}}
\def\P{p}
\def\cc{c}

\def\re{R}
\def\bre{\pk[\re]}
\def\ire{(\ik-\pk)[\re]}

\def\rq{\rho}
\def\uq{u}
\def\tq{T}

\def\vx{x}
\def\vv{v}
\def\vn{n}
\def\nx{\nabla_x}

\def\dx{\Delta_x}

\def\m{\mu}
\def\mh{\m^{\frac{1}{2}}}
\def\mhh{\m^{-\frac{1}{2}}}

\def\id{\mathbf{1}}
\def\od{\mathbf{0}}
\def\oo{o(1)}
\def\oot{o_T}

\def\vo{\omega}
\def\vvv{\mathfrak{v}}
\def\vxx{\mathfrak{x}}
\def\vuu{\mathfrak{u}}

\def\vr{\mathbf{r}}
\def\pl{L}
\def\vt{\varsigma}
\def\mn{\mathfrak{n}}
\def\k{\kappa}
\def\va{v_{\eta}}
\def\vb{v_{\varphi}}
\def\vc{v_{\psi}}

\def\blf{\Phi}
\def\blfi{\blf_{\infty}}
\def\blff{\overline{\Phi}}

\def\ch{\overline{\chi}}

\def\vh{w}
\def\tvh{\widetilde{w}}

\def\bv{\br{\vv}^{\vth}\ue^{\varrho\frac{\abs{\vv}^2}{2}}}
\def\vth{\vartheta}
\def\vrh{\varrho}
\def\reg{\re_w}

\def\test{g}

\def\ssa{\ss_1}

\def\ssc{\ss_2}
\def\ssd{\ss_3}

\def\ssf{\ss_4}
\def\ssg{\ss_5}
\def\ssh{\ss_6}
\def\ssx{\ss_{2a}}
\def\ssy{\ss_{2b}}
\def\ssz{\ss_{2c}}

\def\N{r}
\def\NN{s}
\def\ass{\abs{\fss_b}_{W^{3,\infty}W^{1,\infty}_{\gamma_-,\vrh,\vth}}}
\def\ase{\nm{\fss_i}_{W^{1,\infty}L^{\infty}_{\vrh,\vth}}+\nm{\fss_b}_{W^{1,\infty}W^{3,\infty}W^{1,\infty}_{\gamma_-,\vrh,\vth}}}

\def\g{\mathcal{G}}
\def\ggg{\mathfrak{G}}
\def\hhh{\overline{h}}
\def\hh{\mathfrak{h}}

\def\ga{\overline{\gamma}}
\def\tm{\mathfrak{t}}
\def\xm{\mathfrak{x}}
\def\mf{M_f}
\def\tz{\mathfrak{T}}

\begin{document}

\title{Diffusive Limit of the Boltzmann Equation in Bounded Domains}

\author[Z. Ouyang]{Zhimeng Ouyang}
\address[Z. Ouyang]{
   \newline\indent Department of Mathematics, University of Chicago}
\email{ouyangzm9386@uchicago.edu}
\thanks{Z. Ouyang is supported by NSF Grant DMS-2202824.}

\author[L. Wu]{Lei Wu}
\address[L. Wu]{
   \newline\indent Department of Mathematics, Lehigh University}
\email{lew218@lehigh.edu}
\thanks{L. Wu is supported by NSF Grant DMS-2104775.}

\date{}



\maketitle

\begin{abstract}
The rigorous justification of the hydrodynamic limits of kinetic equations in bounded domains has been actively investigated in recent years. In spite of the progress for the diffuse-reflection boundary case \cite{Esposito.Guo.Kim.Marra2015, Jang.Kim2021, AA023}, the more challenging in-flow boundary case, in which the leading-order boundary layer effect is non-negligible, still remains open.

In this work, we consider the stationary and evolutionary Boltzmann equation with the in-flow boundary in general (convex or non-convex) bounded domains, and demonstrate their incompressible Navier-Stokes-Fourier (INSF) limits in $L^2$. 

Our method relies on a novel and surprising gain of $\e^{\frac{1}{2}}$ in the kernel estimate, which is rooted from a key cancellation of delicately chosen test functions and conservation laws. We also introduce the boundary layer with grazing-set cutoff and investigate its BV regularity estimates to control the source terms of the remainder equation with the help of Hardy's inequality.
\end{abstract}

\pagestyle{myheadings} \thispagestyle{plain} \markboth{Z. OUYANG AND L. WU}{DIFFUSIVE LIMIT OF BOLTZMANN EQUATION IN BOUNDED DOMAINS}

\setcounter{tocdepth}{2}
\makeatletter
\def\l@subsection{\@tocline{2}{0pt}{2.5pc}{5pc}{}}
\makeatother
\tableofcontents


\section{Introduction}

Hydrodynamic limits in non-convex domains (or the so-called exterior domains) play a significant role in the science and engineering problems: e.g. gas dynamics around airplane wings or high-rise buildings, water dynamics near ships or bridge pier, etc. However, due to the intrinsic singularity of kinetic equations in non-convex domains \cite{Kim2014, Esposito.Guo.Marra2018}, the rigorous derivation of fluid systems from kinetic equations remains large open so far.

In this work, we will study the asymptotic behaviors of both the stationary and evolutionary Boltzmann equation in general (convex or non-convex) smooth bounded domains in the presence of boundary layer corrections. 
Due to the complexity of the problem, we will put most of the notation in the Appendix for convenience of the reader.

\subsection{Stationary Problem}

We consider the stationary Boltzmann equation in a three-dimensional smooth bounded domain $\Omega\ni\vx=(x_1,x_2,x_3)$ and velocity domain $\R^3\ni\vv=(v_1,v_2,v_3)$. The stationary density function $\fs(\vx,\vv)$ satisfies
\begin{align}\label{large system}
\left\{
\begin{array}{l}
\vv\cdot\nx \fs =\e^{-1}Q\left[\fs ,\fs \right]\ \ \text{in}\ \
\Omega\times\R^3,\\\rule{0ex}{1.5em} \fs (\vx_0,\vv)=\fs_b(\vx_0,\vv) \ \ \text{for}\ \ \vx_0\in\p\Omega\ \ \text{and}\ \ \vv\cdot\vn<0,
\end{array}
\right.
\end{align}
where $\vn(\vx_0)$ is the unit outward normal vector at $\vx_0\in\p\Omega$, and the Knudsen number $0<\e\ll1$ characterizes the average distance a particle might travel between two collisions.

Here $Q$ is the hard-sphere collision operator
\begin{align}
Q[F,G]:=&\int_{\r^3}\int_{\S^2}q(\vo,\abs{\vuu-\vv})\Big[F(\vuu_{\ast})G(\vv_{\ast})-F(\vuu)G(\vv)\Big]\ud{\vo}\ud{\vuu},
\end{align}
with $\vo\in\r^3$ a unit vector,  $\vuu_{\ast}:=\vuu+\vo\big((\vv-\vuu)\cdot\vo\big)$, $\vv_{\ast}:=\vv-\vo\big((\vv-\vuu)\cdot\vo\big)$, and the hard-sphere collision kernel $q(\vo,\abs{\vuu-\vv}):=q_0\abs{\vo\cdot(\vv-\vuu)}$ for a positive constant $q_0$.

We intend to study the asymptotic limit of $\fs(\vx,\vv)$ as $\e\rt0$.

\subsubsection{Setup and Assumptions}

Assume the in-flow boundary data
\begin{align}
    \fs_b(\vx_0,\vv):=\m(\vv)+\e\mh(\vv)\fss_b(\vx_0,\vv),
\end{align}
where $\m$ denotes the global Maxwellian
\begin{align}
    \m(\vv):=\frac{1}{(2\pi)^{\frac{3}{2}}}\ue^{-\frac{\abs{\vv}^2}{2}},
\end{align}
and $\fss_b(\vx_0,\vv)$ is a small perturbation satisfying
\begin{align}\label{assumption:stationary}
    \ass=\oo.
\end{align}


\subsubsection{Normal Chart near Boundary}\label{sec:geometric-setup}

We follow the approach in \cite{BB002} to define the geometric quantities. 

For smooth manifold $\p\Omega$, there exists an orthogonal curvilinear coordinates system $(\iota_1,\iota_2)$ such that the coordinate lines coincide with the principal directions at any $\vx_0\in\p\Omega$ (at least locally).
Assume $\p\Omega$ is parameterized by $\vr=\vr(\iota_1,\iota_2)$. Let the vector length be $\pl_i=\abs{\p_{\iota_i}\vr}$ and unit vector $\vt_i=\pl_i^{-1}\p_{\iota_i}\vr$.

Consider the corresponding new coordinate system $(\iota_1,\iota_2,\mn)$, where $\mn$ denotes the normal distance to boundary surface $\p\Omega$, i.e.
\begin{align}
\vx=\vr-\mn\vn.
\end{align}
Define the orthogonal velocity substitution for $\vvv:=(\va,\vb,\vc)$ as
\begin{align}
-\vv\cdot\vn:=\va,\qquad
-\vv\cdot\vt_1:=\vb,\qquad
-\vv\cdot\vt_2:=\vc.
\end{align}
Finally, we define the scaled variable $\eta=\e^{-1}\mn$, which implies $\frac{\p}{\p\mn}=\frac{1}{\e}\frac{\p}{\p\eta}$. Denote $\vxx:=(\eta,\iota_1,\iota_2)$.

\subsubsection{Asymptotic Expansion}\label{sec:asymptotic}

We seek a solution to \eqref{large system} in the form 
\begin{align}\label{expand}
    \fs (x,v)=\m+\f+\fb+\e\mh\re
    =\m+\mh\Big(\e\f_1+\e^2\f_2\Big)+\mh\Big(\e\fb_1\Big)+\e\mh\re,
\end{align}
where the interior solution is
\begin{align}\label{expand 1}
\f(x,v):=\mh(v)\Big(\e\f_1(x,v)+\e^2\f_2(x,v)\Big),
\end{align}
and the boundary layer is
\begin{align}\label{expand 2}
\fb(\vxx,\vvv):= \mh(\vvv)\Big(\e\fb_1(\vxx,\vvv)\Big).
\end{align}
Here $\f$ and $\fb$ are defined in Section \ref{sec:expansion-s} and $\re(x,v)$ is the remainder satisfying
\begin{align}\label{remainder}
\left\{
\begin{array}{l}\displaystyle
\vv\cdot\nabla_x \re+\e^{-1}\lc[\re]=\ss\ \ \text{in}\ \ \Omega\times\R^3,\\\rule{0ex}{1.5em}
\re(\vx_0,\vv)=\h(\vx_0,\vv)\ \ \text{for}\
\ \vv\cdot\vn<0\ \ \text{and}\ \ \vx_0\in\p\Omega,
\end{array}
\right.
\end{align}
where $\h$ and $\ss$ are defined in \eqref{d:h}--\eqref{d:ssh}.

Let $\pk[\re]$ be the projection of $\re$ onto the null space $\nk$ of $\lc$ defined in \eqref{linearized}. Then we split
\begin{align}\label{splitting}
    \re=\bre+\ire:=\mh(\vv)\bigg(\P(x)+\vv\cdot\bb(x)+\frac{\abs{\vv}^2-5}{2}\cc(x)\bigg)+\ire.
\end{align}
Denote the working space
\begin{align}\label{working}
    \xnm{\re}:=&\;\e^{-\frac{1}{2}}\tnm{\bre}+\e^{-1}\um{\ire}+\pnm{\re}{6}+\e^{-\frac{1}{2}}\tnms{\re}{\gamma_+}+\pnms{\m^{\frac{1}{4}}\re}{4}{\gamma_+}\\
    &\;+\e^{\frac{1}{2}}\lnmm{\re}+\e^{\frac{1}{2}}\lnmms{\re}{\gamma}.\no
\end{align}

\subsubsection{Main Result}

\begin{theorem}[Stationary Problem]\label{main theorem}
Assume that $\Omega$ is a bounded $C^3$ domain and \eqref{assumption:stationary} holds. 
there exists $\e_0>0$ such that for any $\e\in(0,\e_0)$, there exists a nonnegative solution $\fs(\vx,\vv)$ to the equation \eqref{large system} represented by \eqref{expand} satisfying
\begin{align}\label{main'}
    \xnm{\re}\ls\oot,
\end{align}
where the $X$ norm is defined in \eqref{working}. Such a solution is unique among all solutions satisfying \eqref{main'}. This further yields 
\begin{align}\label{main}
    \tnm{\mhh\fs-\mh-\e\mh\bigg(\rq_1+\vv\cdot\uq_1+\frac{\abs{\vv}^2-3}{2}\tq_1\bigg)}\ls\oot\e^{\frac{3}{2}},
\end{align}
where $\big(\rq_1(x),\uq_1(x),\tq_1(x)\big)$ satisfies the steady Navier-Stokes-Fourier system
\begin{align}\label{fluid}
\left\{
\begin{array}{l}
\nx P_1=0,\\\rule{0ex}{1.5em}
\uq_1\cdot\nx\uq_1 -\gamma_1\dx\uq_1 +\nx \mathfrak{p}_1 =0,\\\rule{0ex}{1.5em}
\nx\cdot\uq_1 =0,\\\rule{0ex}{1.5em}
\uq_1 \cdot\nx\tq_1 -\gamma_2\dx\tq_1 =0,
\end{array}
\right.
\end{align}
for $P_1:=\rq_1+\tq_1$, constants $\gamma_1>0$ and $\gamma_2>0$. The boundary condition
\begin{align}
    \Big(\rq_1(\vx_0),\uq_1(\vx_0),\tq_1(\vx_0)\Big)=\Big(\rq^B(\vx_0),\uq^B(\vx_0),\tq^B(\vx_0)\Big)
\end{align}
is given by
\begin{align}\label{Milne infinity}
    \blfi(\vx_0,\vv)=\blfi(\iota_1,\iota_2,\vvv):=\mh\bigg(\rq^B(\iota_1,\iota_2)+\vv\cdot\uq^B(\iota_1,\iota_2)+\frac{\abs{\vv}^2-3}{2}\tq^B(\iota_1,\iota_2)\bigg)\in\nk,
\end{align}
solved from the Milne problem for $\blf(\vxx,\vvv)$:
\begin{align}\label{Milne problem}
    \left\{
    \begin{array}{l}
    \va\dfrac{\p\blf}{\p\eta}+\lc[\blf]=0,\\\rule{0ex}{1.5em}
    \blf(0,\iota_1,\iota_2,\vvv)=\fss_b(\iota_1,\iota_2,\vvv)\ \ \text{for}\ \ \va>0,\\\rule{0ex}{2.0em}
    \ds\int_{\r^3}\va\mh(\vvv)\blf(0,\iota_1,\iota_2,\vvv)\ud \vvv=\mf,\\\rule{0ex}{2.0em}
    \ds\lim_{\eta\rt\infty}\blf(\eta,\iota_1,\iota_2,\vvv)=\blfi(\iota_1,\iota_2,\vvv).
    \end{array}
    \right.
\end{align}
Here $\mf(\iota_1,\iota_2)$ is chosen such that $\rho^B+T^B=\text{constant}$ and $\int_{\p\Omega}\big(\uq^B\cdot n\big)\ud x=0$.
\end{theorem}

\subsection{Evolutionary Problem}

We consider the evolutionary Boltzmann equation in a three-dimensional smooth bounded domain $\Omega\ni\vx=(x_1,x_2,x_3)$ and velocity domain $\R^3\ni\vv=(v_1,v_2,v_3)$ with time $t\in\rp$. The evolutionary density function $\fs (t,\vx,\vv)$ satisfies
\begin{align}\label{large system=}
\left\{
\begin{array}{l}
\e\dt\fs +\vv\cdot\nx \fs =\e^{-1}Q\left[\fs ,\fs \right]\ \ \text{in}\ \
\rp\times\Omega\times\R^3,\\\rule{0ex}{1.5em} \fs (0,\vx,\vv)=\fs_i(\vx,\vv) \ \ \text{in}\ \
\Omega\times\R^3,\\\rule{0ex}{1.5em} \fs (t,\vx_0,\vv)=\fs_b(t,\vx_0,\vv) \ \ \text{for}\ \ t\in\rp,\ \ \vx_0\in\p\Omega\ \ \text{and}\ \ \vv\cdot\vn(\vx_0)<0.
\end{array}
\right.
\end{align}
We intend to study the asymptotic limit of $\fs(t,\vx,\vv)$ as $\e\rt0$.

\subsubsection{Setup and Assumptions}

Assume the in-flow boundary data
\begin{align}
    \fs_b(t,\vx_0,\vv):=\m(\vv)+\e\mh(\vv)\fss_b(t,\vx_0,\vv),
\end{align}
where $\fss_b(t,\vx_0,\vv)$ is a small perturbation satisfying
\begin{align}\label{assumption:evolutionary1}
    \nm{\fss_b}_{W^{1,\infty}W^{3,\infty}W^{1,\infty}_{\gamma_-,\vrh,\vth}}=\oo.
\end{align}
Assume the initial data
\begin{align}\label{itt 10}
    \fs_i(\vx,\vv):=\m(\vv)+\e\mh(\vv)\fss_i(\vx,\vv)=\m(\vv)+\sum_{k=1}^4\e^k\mh(\vv)\fss_i^{[k]}(\vx,\vv)
\end{align}
where for some $\big(\rq^I(\vx),\uq^I(\vx),\tq^I(\vx)\big)$ satisfying $\rq^I+\tq^I=\text{constant}$, $\nx\cdot\uq^I=0$ and $\nx\times\big(\uq^I\cdot\nx\uq^I-\gamma_1\dx\uq^I\big)=\od$:
\begin{align}
    \fss_i^{[1]}(\vx,\vv)&:=\mh(\vv)\bigg(\rq^I(\vx)+\vv\cdot\uq^I(\vx)+\frac{\abs{\vv}^2-3}{2}\tq^I(\vx)\bigg),\label{assumption:evolutionary4}\\
    \fss_i^{[2]}(\vx,\vv)&:=\lc^{-1}\Big[-\vv\cdot\nx\fss_i^{[1]}+\Gamma\big[\fss_i^{[1]},\fss_i^{[1]}\big]\Big],\\
    \fss_i^{[3]}(\vx,\vv)&:=\lc^{-1}\Big[-\vv\cdot\nx\overline{\fss_i}+2\Gamma\big[\fss_i^{[1]},\fss_i^{[2]}\big]\Big],
\end{align}
and $\fss_i^{[4]}(\vx,\vv)$ can be an arbitrary function.
We assume that $\fss_i(\vx,\vv)$ is a small perturbation term satisfying
\begin{align}\label{assumption:evolutionary2}
    \nm{\fss_i}_{W^{1,\infty}L^{\infty}_{\vrh,\vth}}\leq\sum_{k=1}^4\nm{\fss_i^{[k]}}_{W^{1,\infty}L^{\infty}_{\vrh,\vth}}=\oo.
\end{align}

\begin{remark}\label{rmk:initial}
    Solving $\dt\fs$ from \eqref{large system=}, our definition of $\fss_i$ and \eqref{assumption:evolutionary2} guarantee that 
    \begin{align}
        \lnmm{\e^{-1}\mhh\dt\fs\big|_{t=0}}=\oo\e,\qquad \lnmm{\dt\f_1\big|_{t=0}}=0,
    \end{align}
    which will play a significant role in the remainder estimates. As a matter of fact, our proof still holds with a slightly weaker assumption: for $\f_1,\f_2$ introduced in \eqref{expand}
    \begin{align}
        \lnmm{\e^{-1}\mhh\dt\fs\big|_{t=0}-\dt\f_1\big|_{t=0}-\e\dt\f_2\big|_{t=0}}=\oo\e^{\frac{1}{2}}.
    \end{align}
    Similar to the analysis in \cite{Esposito.Guo.Kim.Marra2015}, this requirement is very sharp to guarantee that the initial time derivative of the remainder $\lnmm{\dt\re\big|_{t=0}}\ls\oot\e^{\frac{1}{2}}$.
\end{remark}

In addition, assume that $\fss_b$ and $\fss_i$ satisfy the compatibility condition at $t=0$, $\vx_0\in\p\Omega$ and $\vv\cdot n<0$: 
\begin{align}\label{assumption:evolutionary3}
    \fss_b(0,\vx_0,\vv)=\fss_i(\vx_0,\vv)=0,\qquad \dt\fss_b(0,\vx_0,\vv)=0.
\end{align}

\begin{remark}
    The compatibility condition \eqref{assumption:evolutionary3} guarantees that there will be no boundary layer at $t=0$.
\end{remark}

\subsubsection{Asymptotic Expansions}\label{sec:asymptotic=}

We seek a solution to \eqref{large system=} in the form 
\begin{align}\label{expand=}
    \fs(t,x,v)=\m+\f+\fb+\e\mh\re=\m+\mh\Big(\e\f_1+\e^2\f_2\Big)+\mh\Big(\e\fb_1\Big)+\e\mh\re,
\end{align}
where the interior solution is
\begin{align}\label{expand 1=}
\f(t,x,v):= \mh(v)\Big(\e\f_1(t,x,v)+\e^2\f_2(t,x,v)\Big),
\end{align}
and the boundary layer is
\begin{align}\label{expand 2=}
\fb(t,\vxx,\vvv):= \mh(\vvv)\Big(\e\fb_1(t,\vxx,\vvv)\Big).
\end{align}
Here $\f$ and $\fb$ are defined in Section \ref{sec:expansion-e} and  $\re(t,x,v)$ is the remainder satisfying
\begin{align}\label{remainder=}
\left\{
\begin{array}{l}\displaystyle
\e\dt\re+\vv\cdot\nabla_x \re+\e^{-1}\lc[\re]=\ss\ \ \text{in}\ \ \rp\times\Omega\times\R^3,\\\rule{0ex}{1.5em}
\re(0,\vx,\vv)=\z(\vx,\vv)\ \ \text{in}\ \ \Omega\times\R^3,\\\rule{0ex}{1.5em}
\re(t,\vx_0,\vv)=\h(t,\vx_0,\vv)\ \ \text{for}\
\ \vv\cdot\vn<0\ \ \text{and}\ \ \vx_0\in\p\Omega,
\end{array}
\right.
\end{align}
where $\z$, $\h$ and $\ss$ are defined in \eqref{d:z=}--\eqref{d:ssh=}.

Similar to \eqref{splitting}, we split
\begin{align}\label{splitting=}
    \re=\bre+\ire=\mh(\vv)\left(\P(t,x)+\vv\cdot\bb(t,x)+\frac{\abs{\vv}^2-5}{2}\cc(t,x)\right)+\ire.
\end{align}
Denote the working space
\begin{align}\label{working=}
    \xnnm{\re}:=
    &\nnm{\re}_{L^{\infty}_tL^2_{x,v}}+\e^{-\frac{1}{2}}\tnnms{\re}{\ga_+}+\e^{-\frac{1}{2}}\tnnm{\bre}+\e^{-1}\unm{\ire}\\
    &\nnm{\dt\re}_{L^{\infty}_tL^2_{x,v}}+\e^{-\frac{1}{2}}\tnnms{\dt\re}{\ga_+}+\e^{-\frac{1}{2}}\tnnm{\dt\bre}+\e^{-1}\unm{\dt\ire}\no\\
    &+\e^{-\frac{1}{2}}\nm{\re}_{L^{\infty}_tL^2_{\gamma_+}}+\nm{\m^{\frac{1}{4}}\re}_{L^{\infty}_tL^4_{\gamma_+}}+\e^{-1}\nnm{\ire}_{L^{\infty}_tL^2_{\nu}}+\nnm{\re}_{L^{\infty}_tL^6_{x,v}}\no\\
    &
    +\e^{\frac{1}{2}}\lnnmm{\re}+\e^{\frac{1}{2}}\lnnmms{\re}{\ga_+}.\no
\end{align}

\subsubsection{Main Result}

\begin{theorem}[Evolutionary Problem]\label{main theorem=}
Assume that $\Omega$ is a bounded $C^3$ domain and \eqref{assumption:evolutionary1}\eqref{assumption:evolutionary2}\eqref{assumption:evolutionary3} hold. 
there exists $\e_0>0$ such that for any $\e\in(0,\e_0)$ and any prescribed constant $\tz>0$, there exists a nonnegative solution $\fs(t,\vx,\vv)$ defined on $(t,x,v)\in[0,\tz]\times\Omega\times\r^3$ to the equation \eqref{large system=} represented by \eqref{expand=} satisfying
\begin{align}\label{main'=}
    \xnnm{\re}\ls\oot,
\end{align}
where the $X$ norm is defined in \eqref{working=}. Such a solution is unique among all solutions satisfying \eqref{main'=}. This further yields 
\begin{align}\label{main=}
    \tnnm{\mhh\fs-\mh-\e\mh\bigg(\rq_1+\vv\cdot\uq_1+\frac{\abs{\vv}^2-3}{2}\tq_1\bigg)}\ls\oot\e^{\frac{3}{2}},
\end{align}
where $\Big(\rq_1(t,x),\uq_1(t,x),\tq_1(t,x)\Big)$ satisfies the unsteady Navier-Stokes-Fourier system
\begin{align}\label{fluid=}
\left\{
\begin{array}{l}
\nx P_1=0,\\\rule{0ex}{1.5em}
\dt\uq_1+\uq_1\cdot\nx\uq_1 -\gamma_1\dx\uq_1 +\nx \mathfrak{p}_1 =0,\\\rule{0ex}{1.5em}
\nx\cdot\uq_1 =0,\\\rule{0ex}{1.5em}
\dt\tq_1+\uq_1 \cdot\nx\tq_1 -\gamma_2\dx\tq_1 =0.
\end{array}
\right.
\end{align}
The initial condition
\begin{align}
    \Big(\rq_1(0,\vx),\uq_1(0,\vx),\tq_1(0,\vx)\Big)=\Big(\rq^I(\vx),\uq^I(\vx),\tq^I(\vx)\Big)
\end{align}
is given by \eqref{assumption:evolutionary4}, and
the boundary condition
\begin{align}
    \Big(\rq_1(t,\vx_0),\uq_1(t,\vx_0),\tq_1(t,\vx_0)\Big)=\Big(\rq^B(t,\vx_0),\uq^B(t,\vx_0),\tq^B(t,\vx_0)\Big)
\end{align}
is given by \eqref{Milne infinity}
solved from the Milne problem \eqref{Milne problem} for $\blf(t,\vxx,\vvv)$.
\end{theorem}


\begin{remark}
    If we strengthen the boundary data assumption \eqref{assumption:evolutionary1} such that for some constant $K_0>0$
    \begin{align}
    \nm{\ue^{K_0t}\fss_b}_{W^{1,\infty}W^{3,\infty}W^{1,\infty}_{\gamma_-,\vrh,\vth}}=\oo,
    \end{align}
    then by a similar argument, Theorem \ref{main theorem=} still holds with $\tz=\infty$ and improved \eqref{main=} that for some constant $K\in(0,K_0)$
    \begin{align}
    \tnnm{\ue^{Kt}\left\{\mhh\fs-\mh-\e\mh\bigg(\rq_1+\vv\cdot\uq_1+\frac{\abs{\vv}^2-3}{2}\tq_1\bigg)\right\}}\ls\oot\e^{\frac{3}{2}}.
\end{align}
\end{remark}


\bigskip
\section{Background and Methodology}

\subsection{Literature Review}

The hydrodynamic limit of the Boltzmann equation has attracted a lot of attention since the pioneering work by Hilbert \cite{Hilbert1916, Hilbert1953}. This is a key ingredient to tackle his famous sixth problem \cite{Hilbert1900} on the axiomization of physics. Most of the important fluid models can be derived by asymptotic expansion with respect to the Knudsen number $\e$ (the so-called Hilbert expansion) at least formally. The rigorous justification of this asymptotic convergence has been studied in many different settings (domain, scaling, solution types) and it is almost impossible for us to provide an extensive literature review.

When the Mach number is $O(1)$ and Reynolds number is $O\left(\e^{-1}\right)$, the solution of the Boltzmann equation will converge to solutions of the compressible Euler equations. Such result was obtained by Caflisch \cite{Caflisch1980} and Lachowicz \cite{Lachowicz1987}, while Nishida \cite{Nishida1978}, Asano-Ukai \cite{Asano.Ukai1983} proved the similar results with a different approach. For the convergence in the presence of singularities for the Euler equations, we refer to Yu \cite{Yu2005} and Huang-Wang-Yang \cite{Huang.Wang.Yang2010, Huang.Wang.Yang2010(=)}. The bounded domain case with boundary effects was considered in Huang-Wang-Yang \cite{Huang.Wang.Yang2013}, Huang-Wang-Wang-Yang \cite{Huang.Wang.Wang.Yang2016} and Guo-Huang-Wang \cite{Guo.Huang.Wang2021}. The relativistic Euler limit has been studied in Speck-Strain \cite{Speck.Strain2011}. Due to the intrinsic difficulty of local Maxwellian and the possible singularity from the Euler equations, most of these results are local in time and the bounded domain case remains largely open.

On a longer time scale $\e^{-1}$, where diffusion effects become signiﬁcant, the problem can be settled only in the low Mach numbers regime (Mach number is $O(\e)$ or smaller and Reynolds number is $O(1)$). The solution of the Boltzmann equation will converge to a global Maxwellian plus an $O(\e)$ perturbation solving the incompressible Navier-Stokes equations. We refer to Bardos-Ukai \cite{Bardos.Ukai1991}, DeMasi-Esposito-Lebowitz \cite{Masi.Esposito.Lebowitz1989}, Guo \cite{Guo2006}, Guo-Jang \cite{Guo.Jang2010} for smooth solutions, and Bardos-Golse-Levermore \cite{Bardos.Golse.Levermore1991, Bardos.Golse.Levermore1993, Bardos.Golse.Levermore1998, Bardos.Golse.Levermore2000}, Lions-Masmoudi \cite{Lions.Masmoudi2001}, Jiang-Masmoudi \cite{Jiang.Masmoudi2016}, Masmoudi-Saint-Raymond \cite{Masmoudi.Saint-Raymond2003}, Golse-Saint-Raymond \cite{Golse.Saint-Raymond2004} for renormalized solutions. For more references and related topics and developments, we refer to Villani \cite{Villani2002}, Desvillettes-Villani \cite{Desvillettes.Villani2005, Desvillettes.Villani2001}, Carlen-Carvalho \cite{Carlen.Carvalho1994}, Arkeryd-Nouri \cite{Arkeryd.Nouri2000}, Arkeryd-Esposito-Marra-Nouri \cite{Arkeryd.Esposito.Marra.Nouri2010, Arkeryd.Esposito.Marra.Nouri2011}, Esposito-Lebowitz-Marra \cite{Esposito.Lebowitz.Marra1994, Esposito.Lebowitz.Marra1995}. We also refer to the review and survey by Saint-Raymond \cite{Saint-Raymond2009}, Golse \cite{Golse2005}, Esposito-Marra \cite{Esposito.Marra2020}, and the references therein.

When the Mach number is $O(\e)$ and the Reynolds number is $O\left(\e^{\alpha}\right)$ for $0<\alpha<1$, the solution of the Boltzmann equation will converge to a global Maxwellian plus an $O(\e)$ perturbation solving the incompressible Euler equations. Besides the overlapped references as above, we also refer to the recent development with boundary effects (including both the Knudsen layer and Prandlt layer) in Jang-Kim \cite{Jang.Kim2021}, Cao-Jang-Kim \cite{Cao.Jang.Kim2021}, Kim-La \cite{Kim.La2022}.

When the Mach number is $O(\e)$, Reynolds number is $O(1)$ but the density/temperature is $O(\e)$, as Sone \cite{Sone2002, Sone2007} predicted, a new type of mixed fluid system (the so-called ghost-effect equations) emerges as the hydrodynamic limit of the Boltzmann equation. We refer to the recent development \cite{AA023, AA024}.

In this paper, we will focus on the diffusive limit of Boltzmann equation in bounded domains, under both stationary and evolutionary settings. Our work is closely related to the recent development of $L^2-L^6-L^{\infty}$ framework and the kinetic boundary layer with geometric effects. We refer to Esposito-Guo-Kim-Marra \cite{Esposito.Guo.Kim.Marra2013, Esposito.Guo.Kim.Marra2015}, and Wu \cite{AA013}, Wu-Ouyang \cite{BB002, AA017, AA018} for the diffuse-reflection boundary. In particular, \cite{Esposito.Guo.Kim.Marra2013, Esposito.Guo.Kim.Marra2015} justify the $L^2$ convergence (for both convex and non-convex domains) relying on an improved $L^2-L^6-L^{\infty}$ framework without boundary layer expansion, and \cite{AA013, BB002, AA017, AA018} show the $L^{\infty}$ convergence for convex domains with the boundary layer expansion. As \cite{AA006, AA020} reveal, the delicately designed boundary layer with geometric correction cannot attain $W^{1,\infty}$ regularity in non-convex domains, and thus the $L^{\infty}$ convergence is far from reach. The case of specular-reflection and bounce-back boundary remain largely open (except on half-space of channel domains), mainly due to the lack of explicit kernel estimate to track the $\e$ dependence \cite{Guo2010, Briant.Guo2016}. 

Boundary effect and grazing singularity play a more significant role for the more challenging in-flow boundary, since this is the only case that the leading-order boundary layer does not vanish. Hence, even for $L^2$ convergence, we cannot see any viable option to fully avoid the boundary layer expansion. As far as we are aware of, the best result is the $L^{\infty}$ convergence for unit disk or unit ball domains \cite{AA004}, but there is no clear path to extend the techniques to cover general smooth convex domains, let alone the non-convex ones.

In addition, we include some recent papers on the diffusive limit of the Boltzmann equation and related models \cite{Briant.Guo2016, Briant.Merino-Aceituno.Mouhot2019, Jiang.Masmoudi2022, Jang2009}. We also list some recent developments along $L^2-L^p-L^{\infty}$ framework \cite{Guo2010, Guo.Jang2010, Guo.Jang.Jiang2010, Guo.Kim.Tonon.Trescases2013, Guo.Kim.Tonon.Trescases2016, Kim2011, Kim2014, Speck.Strain2011, AA003, AA016}. 

In this paper, we utilize a different approach and design a cutoff boundary layer to obtain the $L^2$ convergence in general smooth bounded (including convex and non-convex) domains.

\subsection{Difficulty}

In the following, we will utilize the stationary remainder equation \eqref{remainder} to illustrate the key ideas. Rooted from the basic energy estimate and the coercivity of $\bbr{\lc[\re],\re}$, it is well-known that we may bound $\unm{\ire}$:
\begin{align}\label{itt 01}
    \e^{-\frac{1}{2}}\tnms{\re}{\gamma_+}+\e^{-1}\tnm{\ire}\ls \oo\tnm{\bre}+\e^{-\frac{1}{2}}.
\end{align}
Then by testing \eqref{remainder} against functions $\nx\varphi\cdot\ab$ with $\varphi\sim\dx^{-1}\P,\ \dx^{-1}\cc$ and $\nx\psi:\bbb$ with $\psi\sim\dx^{-1}\bb$, we may in turn bound $\bre$ in terms of $\ire$:
\begin{align}\label{itt 02}
    \tnm{\bre}\ls \e^{-\frac{1}{2}}\tnms{\re}{\gamma_+}+\e^{-1}\tnm{\ire}+\e^{-\frac{1}{2}}.
\end{align}
Clearly, \eqref{itt 01} and \eqref{itt 02} lead to
\begin{align}\label{itt 05}
    \e^{-\frac{1}{2}}\tnms{\re}{\gamma_+}+\e^{-1}\tnm{\ire}+\tnm{\bre}\ls \e^{-\frac{1}{2}}.
\end{align}
Note that \eqref{itt 05} is insufficient to justify the desired $L^2$ convergence 
\begin{align}\label{itt 07}
    \lim_{\e\rt0}\tnnm{\re}=0.
\end{align}
For the diffuse-reflection boundary \cite{Esposito.Guo.Kim.Marra2015, AA017}, the general strategy to overcome this difficulty is to expand the interior solution and boundary layer to sufficiently higher order, i.e. we redefine
\begin{align}\label{expand'}
    \fs (x,v)=\m+\f+\fb+\e\mh\re
    =\m+\mh\Big(\e\f_1+\e^2\f_2+\e^3\f_3\Big)+\mh\Big(\e\fb_1+\e^2\fb_2\Big)+\e\mh\re.
\end{align}
The diffuse-reflection boundary condition implies that $\fb_1=0$ and thus there is no difficulties caused by the regularity of the boundary layer \cite{AA004}. The extra terms $\e^3\f_3\mh$ and $\e^2\fb_2\mh$ help improve the bounds of $\ss$ and $\h$ in \eqref{remainder}. Eventually, this leads to 
\begin{align}\label{itt 06}
    \e^{-\frac{1}{2}}\tnms{\re}{\gamma_+}+\e^{-1}\tnm{\ire}+\tnm{\bre}\ls\e^{\alpha},
\end{align}
with $\alpha>0$ and thus \eqref{itt 07} follows.

Unfortunately, this strategy does not work for the in-flow boundary, in which $\fb_1$ is not necessarily zero and $\frac{\p\fb_1}{\p\va}\notin L^{\infty}$ \cite{AA004}. Hence, we cannot expect to expand to $\fb_2$ which satisfies the Milne problem
\begin{align}\label{itt 08}
    \va\frac{\p\fb_2}{\p\eta}+\lc\left[\fb_2\right]\approx \dfrac{\p\fb_1}{\p\va}.
\end{align}
The loss of $\fb_1$ regularity makes it impossible \cite{AA004} to justify the well-posedness of \eqref{itt 08}, and thus \eqref{itt 07} is not attainable.

\subsection{Methodology -- Stationary Problem}

As the above analysis reveals, with the expansion \eqref{expand} for the in-flow boundary, the bottleneck of \eqref{itt 05} lies in the kernel bound \eqref{itt 02} and the source terms estimates for both \eqref{itt 01} and \eqref{itt 02}. In this paper, we will design several delicate test functions to obtain
\begin{align}\label{itt 09}
    \e^{-\frac{1}{2}}\tnms{\re}{\gamma_+}+\e^{-1}\tnm{\ire}\ls \oo\e^{-\frac{1}{2}}\tnm{\bre}+1,
\end{align}
and
\begin{align}\label{itt 04}
    \e^{-\frac{1}{2}}\tnm{\bre}\ls \e^{-\frac{1}{2}}\tnms{\re}{\gamma_+}+\e^{-1}\tnm{\ire}+1,
\end{align}
which introduce a crucial gain of half-order $\e$ compared with \eqref{itt 02}, and lead to \eqref{itt 07}.

Our key idea is a set of tricky combinations of weak formulations and conservation laws to eliminate the worst term $\e^{-1}\tnnm{\ire}$ on the RHS of \eqref{itt 02}. We will illustrate more precise statement of the argument in the following. Due to the presence of the nonlinear term $\Gamma[\re,\re]$ in \eqref{remainder}, we need to bound $L^2-L^6-L^{\infty}$ norms for both $\bre$ and $\ire$.\\

\paragraph{\underline{Energy Estimate}}
Testing \eqref{remainder} against $\e^{-1}\re$ and utilizing the coercivity and orthogonality yield
\begin{align}
    \e^{-1}\tnms{\re}{\gamma_+}^2+\e^{-2}\tnm{\ire}^2\ls \abs{\e^{-1}\bbr{\ss,\re}}.
\end{align}
Here, the most difficult term in $\ss$ is the normal velocity derivative of $\fb_1$:
\begin{align}
    \abs{\e^{-1}\br{\frac{\p\fb_1}{\p\va},\bre}}&\ls\e^{-1}\abs{\br{\fb_1,\bre}}\ls\e^{-1}\nm{\fb_1}_{L^2_xL^1_v}\nm{\bre}_{L^2_xL^{\infty}_v}\ls\oot\e^{-\frac{1}{2}}\tnm{\bre},
\end{align}
which relies on a key integration by parts with respect to $\va$ and taking the full advantage of the rescaling $\eta=\e^{-1}\mn$. 

Therefore, we arrive at
\begin{align}
    \e^{-\frac{1}{2}}\tnms{\re}{\gamma_+}+\e^{-1}\tnm{\ire}\ls\oot\xnm{\re}+\xnm{\re}^2+\oot,
\end{align}
and further with the help of interpolation
\begin{align}
    \pnm{\ire}{6}+\pnms{\m^{\frac{1}{4}}\re}{4}{\gamma_+}
    \ls\oot\xnm{\re}+\xnm{\re}^2+\oot.
\end{align}

\paragraph{\underline{Estimate of $\P$}}
Test \eqref{remainder} against  the smooth function $\psi=\mh\big(\vv\cdot\nx\varphi\big)$ with $\varphi\sim\dx^{-1}\P$ yields
\begin{align}
    \int_{\gamma}\re\psi(\vv\cdot\vn)-\bbr{\re,\vv\cdot\nx\psi}
    &=\bbr{\ss,\psi}.
\end{align}
By oddness and orthogonality, we eliminate the worst contribution $\e^{-1}\bbr{\lc[\re],\psi}=0$. Then a straightforward estimate for the source term $\bbr{\ss,\psi}$ reveals the $\e^{\frac{1}{2}}$ gain:
\begin{align}
    \e^{-\frac{1}{2}}\tnm{\P}+\pnm{\P}{6}\ls\oot\xnm{\re}+\xnm{\re}^2+\oot.
\end{align}

\paragraph{\underline{Estimate of $\cc$}}
For smooth function $\varphi\sim \dx^{-1}\cc$ with $\varphi\big|_{\p\Omega}=0$, we test \eqref{remainder} against  $\varphi\left(\abs{v}^2-5\right)\mh$ to obtain
\begin{align}\label{cc 01i}
-\bbrx{\nx\varphi,\varsigma}+\int_{\p\Omega}\varphi\varsigma\cdot n=&\br{\varphi\left(\abs{v}^2-5\right)\mh,\ss},
\end{align}
and against $\nx\varphi\cdot\a$ to obtain
\begin{align}\label{conservation law 4i}
    &-\k\bbrx{\dx\varphi,c}+\e^{-1}\bbr{\nx\varphi,\varsigma}\\
    =&\bbrb{\nx\varphi\cdot\a,h}{\gamma_-}-\bbrb{\nx\varphi\cdot\a,\re}{\gamma_+}+\br{v\cdot\nx\Big(\nx\varphi\cdot\a\Big),\ire}+\bbr{\nx\varphi\cdot\a,\ss},\no
\end{align}
where $\ds\varsigma:=\int_{\r^3}\mh v\abs{v}^2\ire$.

Therefore, adding $\e^{-1}\times$\eqref{cc 01i} and \eqref{conservation law 4i} exactly eliminates the troublesome terms $\e^{-1}\bbrx{\nx\varphi,\varsigma}$ and $\ds\e^{-1}\int_{\p\Omega}\varphi\varsigma\cdot n$ whose presence leads to the worst $\e^{-2}\unm{\ire}^2$ contribution. These new conservation laws will be the backbone of $\cc$ estimates. Hence, we obtain the $\e^{\frac{1}{2}}$ gain:
\begin{align}
    \e^{-\frac{1}{2}}\tnm{\cc}+\pnm{\cc}{6}\ls\oot\xnm{\re}+\xnm{\re}^2+\oot,
\end{align}
which relies on the careful estimates for the source terms $\e^{-1}\br{\varphi\left(\abs{v}^2-5\right)\mh,\ss}$ and $\bbr{\nx\varphi\cdot\a,\ss}$. Here we will take full advantage of the cutoff boundary layer (see the key bounds in Lemma \ref{s2-estimate}) and Hardy's inequality
\begin{align}
    \abs{\e^{-1}\br{\varphi\left(\abs{v}^2-5\right)\mh,\frac{\p\fb_1}{\p\va}}}&\ls\abs{\e^{-1}\br{\varphi\left(\abs{v}^2-5\right)\mh,\fb_1}}=\abs{\br{\frac{1}{\mn}\int_0^{\mn}\frac{\p\varphi}{\p\mn}\left(\abs{v}^2-5\right)\mh,\eta\fb_1}}\\
    &\ls\tnm{\frac{1}{\mn}\int_0^{\mn}\frac{\p\varphi}{\p\mn}\left(\abs{v}^2-5\right)\mh}\tnm{\eta\fb_1}\ls\oot\e^{\frac{1}{2}}\nm{\varphi}_{H^1}\ls\oot\e^{\frac{1}{2}}\tnm{\cc}.\no
\end{align}

\paragraph{\underline{Estimate of $\bb$}}
For smooth function $\psi\sim \dx^{-1}\bb$ with $\nx\cdot\psi=0$ and $\psi\big|_{\p\Omega}=0$ (solved from the Stokes problem), 
we test \eqref{remainder} against $\psi\cdot\vv\mh$ to obtain
\begin{align}\label{cc 03i}
-\bbrx{\nx\cdot\psi, \P}-\bbrx{\nx\psi,\varpi}+\int_{\p\Omega}\Big(\P\psi+\psi\cdot\varpi\Big)\cdot n &=\br{\psi\cdot v\mh,\ss},
\end{align}
and against $\nx\psi:\b$ to obtain
\begin{align}\label{conservation law 5i}
    &\;-\lambda\bbrx{\dx\psi,\bb}+\e^{-1}\bbr{\nx\psi,\varpi}\\
    =&\;\bbrb{\nx\psi\cdot\b,h}{\gamma_-}-\bbrb{\nx\psi\cdot\b,\re}{\gamma_+}+\br{v\cdot\nx\Big(\nx\psi:\b\Big),\ire}+\bbr{\nx\psi:\b,\ss},\no
\end{align}
where $\ds\varpi:=\int_{\r^3}\mh\big(v\otimes v\big)\ire$.

Therefore, adding $\e^{-1}\times$\eqref{cc 03i} and \eqref{conservation law 5i} exactly eliminates the worrisome terms $\e^{-1}\bbrx{\nx\psi,\varpi}$, $\e^{-1}\bbrx{\nx\cdot\psi, \P}$ and $\ds\e^{-1}\int_{\p\Omega}\Big(\P\psi+\psi\cdot\varpi\Big)\cdot n$ which yields the worst $\e^{-2}\unm{\ire}^2$ contribution. Then similar to $\cc$ estimates, after the careful estimates for the source terms $\e^{-1}\bbr{\psi\cdot v\mh,\ss}$ and $\bbr{\nx\psi:\b,\ss}$, we obtain $O(\e^{\frac{1}{2}})$ gain of RHS:
\begin{align}
    \e^{-\frac{1}{2}}\tnm{\bb}+\pnm{\bb}{6}\ls&\oot\xnm{\re}+\xnm{\re}^2+\oot,
\end{align}
These new conservation laws will sit at the center stage of $\bb$ estimates.\\

\paragraph{\underline{$L^{\infty}$ Estimate}}
We will employ the $L^2-L^6-L^{\infty}$ framework to obtain
\begin{align}
\e^{\frac{1}{2}}\lnmm{\re}+\e^{\frac{1}{2}}\lnmms{\re}{\gamma_+}
&\ls \pnm{\bre}{6}+\e^{-1}\um{\ire}+ \oot\xnm{\re}+\xnm{\re}^{2}+\oot\\
&\ls \oot\xnm{\re}+\xnm{\re}^{2}+\oot.\no
\end{align}
Note that the normal velocity derivative $\frac{\p\fb_1}{\p\eta}\approx\oot\e^{-1}$ from Lemma \ref{s2-estimate} and the grazing-set cutoff, which is under control from \eqref{ltt 13}.

\subsection{Methodology -- Evolutionary Problem}

Besides the difficulties and methodology mentioned in the stationary problem, we have an additional layer of obstacle in the evolutionary settings for the remainder equation \eqref{remainder=}. The $L^{\infty}$ estimate
\begin{align}\label{itt 12}
    \e^{\frac{1}{2}}\lnnmm{\re}+\e^{\frac{1}{2}}\lnnmms{\re}{\ga_+}
    \ls \nnm{\bre}_{L^{\infty}_tL^6_{x,v}}+\e^{-1}\nnm{\ire}_{L^{\infty}_tL^2_{x,v}}+ \oot\xnm{\re}+\xnm{\re}^{2}+\oot
\end{align}
calls for the control of $\nnm{\bre}_{L^{\infty}_tL^6_{x,v}}$, instead of the full $L^6$ bound $\pnnm{\bre}{6}$ from the energy and kernel estimates. Hence, we have to carefully estimate both the accumulative and instantaneous $L^2-L^6$ norms. \\

\paragraph{\underline{Accumulative Estimates}} Based on a delicate choice of test functions and the analogous cancellation with weak formulations and conservation laws, we obtain the energy estimate
\begin{align}
    \tnm{\re(t)}+\e^{-\frac{1}{2}}\tnnms{\re}{\ga_+}+\e^{-1}\tnnm{\ire}&\ls\oot\xnnm{\re}+\xnnm{\re}^2+\oot,\\
    \tnm{\dt\re(t)}
    +\e^{-\frac{1}{2}}\tnnms{\dt\re}{\ga_+}+\e^{-1}\tnnm{(\ik-\pk)[\dt\re]}
    &\ls\oot\xnnm{\re}+\xnnm{\re}^2+\oot,\label{itt 11}
\end{align}
and the kernel estimate
\begin{align}
    \e^{-\frac{1}{2}}\tnnm{\bre}&\ls\oot\xnnm{\re}+\xnnm{\re}^{2}+\oot,\\
    \e^{-\frac{1}{2}}\tnnm{\dt\bre}&\ls\oot\xnnm{\re}+\xnnm{\re}^{2}+\oot.
\end{align}
Notice that we need to bound both $\re$ and its time derivative $\dt\re$ for the convenience of instantaneous estimates. Notably, the $\dt\re$ estimates calls for $\tnm{\dt\re(0)}\ls \e^{\frac{1}{2}}$ which is the key reason that our argument only applies to the well-prepared initial data \eqref{itt 10}, and cannot include the discussion of the initial layer (as opposed to the case of transport equation \cite{Ouyang2023}). 

In the analysis of evolutionary conservation laws, we also need a careful bound of the time-derivative terms which provide a favorable sign and separate estimates of $\dt\nx\varphi$ and $\dt\nx\psi$ to close the proof.\\

\paragraph{\underline{Instantaneous Estimates}} We rewrite \eqref{remainder==} by moving $\dt\re$ to the RHS
\begin{align}
    \vv\cdot\nx\re+\e^{-1}\lc\left[\re\right]=\ss-\e\dt\re,
\end{align}
where we regard $\dt\re$ as a source term in the stationary remainder equation. Hence, by a similar estimates as the stationary case, we obtain the energy estimate
\begin{align}
    \e^{-\frac{1}{2}}\tnms{\re(t)}{\gamma_+}+\e^{-1}\tnm{\ire(t)}&\ls\oot\xnnm{\re}+\xnnm{\re}^2+\oot,\\
    \pnm{\ire(t)}{6}+\pnms{\m^{\frac{1}{4}}\re(t)}{4}{\gamma_+}
    &\ls\oot\xnnm{\re}+\xnnm{\re}^2+\oot,
\end{align}
and the kernel estimate
\begin{align}
    \pnm{\bre(t)}{6}\ls\oot\xnnm{\re}+\xnnm{\re}^{2}+\oot.
\end{align}
Notice that these estimates heavily rely on the accumulative $L^2$ bound of $\dt\re$ in \eqref{itt 11}. Then we may proceed to $L^{\infty}$ estimate \eqref{itt 12} and close the proof.


\bigskip
\section{Stationary Problem}

\subsection{Asymptotic Analysis}\label{sec:expansion-s}

\subsubsection{Interior Solution}

The derivation of the interior solution is classical. We refer to \cite{Sone2002, Sone2007, Golse2005} and the references therein. By inserting \eqref{expand 1} into \eqref{large system} and comparing the order of $\e$, we require that
\begin{align}
    0&=2\mhh\qq\big[\m,\mh\f_1\big],\\
    \vv\cdot\nx\f_1&=2\mhh\qq\big[\m,\mh\f_2\big]+\mhh\qq\big[\mh\f_1,\mh\f_1\big],
\end{align}
which are equivalent to
\begin{align}
    \lc\left[\f_1\right]&=0,\\
    \vv\cdot\nx\f_1+\lc\left[\f_2\right]&=\Gamma\left[\f_1,\f_1\right].
\end{align}
Considering the further expansion, we additionally require
\begin{align}\label{att 01}
    \vv\cdot\nx\f_2\perp\nk.
\end{align}
Hence, we conclude that
\begin{align}
    \f_1(\vx,\vv)=\mh(\vv)\bigg(\rq_1(x)+\vv\cdot\uq_1(x)+\frac{\abs{\vv}^2-3}{2}\tq_1(x)\bigg),
\end{align}
where $(\rq_1,\uq_1,\tq_1)$ satisfies the incompressible Navier-Stokes-Fourier system \eqref{fluid}.

Also, we have 
\begin{align}
\f_2(\vx,\vv)=&\;\mh(\vv)\left(\rq_2(x) +\vv\cdot\uq_2(x)+\frac{\abs{\vv}^2-3}{2}\tq_2(x)\right)\\
&\;+\mh(\vv)\left(\rq_1(\vv\cdot\uq_1)+\left(\rq_1\tq_1+\frac{\abs{\vv}^2-3}{2}\abs{\uq_1}^2\right)\right)+\li\Big[-\vv\cdot\nx\f_1+\Gamma[\f_1,\f_1]\Big]\no
\end{align}
where $(\rq_2,\uq_2,\tq_2)$ satisfies the fluid system
\begin{align} \label{fluid'}
\\
\left\{
\begin{array}{l}
\nx P_2=0,\\\rule{0ex}{1.5em}
\uq_1\cdot\nx\uq_2+(\rq_1\uq_1+\uq_2)\cdot\nx\uq_1-\gamma_1\dx\uq_2 +\nx\mathfrak{p}_2 =-\gamma_2\nx\cdot\dx\tq_1-\gamma_4\nx\cdot\Big(\tq_1\big(\nx\uq_1+(\nx\uq_1)^T\big)\Big),\\\rule{0ex}{1.5em}
\nx\cdot\uq_2 =-\uq_1\cdot\nx\rq_1,\\\rule{0ex}{1.5em}
\uq_1\cdot\nx\tq_2+(\rq_1\uq_1+\uq_2)\cdot\nx\tq_1-\uq_1\cdot\nx\mathfrak{p}_1=\gamma_1\Big(\nx\uq_1+(\nx\uq_1)^T\Big)^2+\dx\big(\gamma_2\tq_2+\gamma_5\tq_1^2\big),\no
\end{array}
\right.
\end{align}
for $P_2:=\rq_2 +\tq_2 +\rq_1 \tq_1$ and constants $\gamma_3,\gamma_4,\gamma_5$.

\subsubsection{Milne Problem}

The normal chart defined in Section \ref{sec:geometric-setup} was introduced in \cite{AA023, AA024}, and was designed to split the normal and tangential variables for the convenience of defining boundary layers. Under the substitution $(\vx,\vv)\rt(\vxx,\vvv)$, we have
\begin{align}
\vv\cdot\nx=&\frac{1}{\e}\va\dfrac{\p }{\p\eta}-\dfrac{1}{R_1-\e\eta}\bigg(\vb^2\dfrac{\p }{\p\va}-\va\vb\dfrac{\p}{\p\vb}\bigg)
-\dfrac{1}{R_2-\e\eta}\bigg(\vc^2\dfrac{\p }{\p\va}-\va\vc\dfrac{\p }{\p\vc}\bigg)\\
&+\frac{1}{\pl_1\pl_2}\Bigg(\frac{R_1\p_{\iota_1\iota_1}\vr\cdot\p_{\iota_2}\vr}{\pl_1(R_1-\e\eta)}\vb\vc
+\frac{R_2\p_{\iota_1\iota_2}\vr\cdot\p_{\iota_2}\vr}{\pl_2(R_2-\e\eta)}\vc^2\Bigg)\frac{\p}{\p\vb}\no\\
&+\frac{1}{\pl_1\pl_2}\Bigg(R_2\frac{\p_{\iota_2\iota_2}\vr\cdot\p_{\iota_1}\vr}{\pl_2(R_2-\e\eta)}\vb\vc
+\frac{R_1\p_{\iota_1\iota_2}\vr\cdot\p_{\iota_1}\vr}{\pl_1(R_1-\e\eta)}\vb^2\Bigg)\frac{\p}{\p\vc}\no\\
&+\bigg(\frac{R_1\vb}{\pl_1(R_1-\e\eta)}\frac{\p}{\p\iota_1}+\frac{R_2\vc}{\pl_2(R_2-\e\eta)}\frac{\p}{\p\iota_2}\bigg).\no
\end{align}

\paragraph{\underline{Well-Posedness and Regularity}}
Now we discuss the well-posedness and regularity of the Milne problem for $\g(\vxx,\vvv)$:
\begin{align}\label{Milne}
    \left\{
    \begin{array}{l}
    \va\dfrac{\p\g }{\p\eta}+\nu\g- K\left[\g\right]=0,\\\rule{0ex}{1.2em}
    \g(0,\iota_1,\iota_2,\vvv)=\hhh(\iota_1,\iota_2,\vvv)\ \ \text{for}\ \ \va>0,\\\rule{0ex}{1.5em}
    \ds\int_{\r^3}\va\mh(\iota_1,\iota_2,\vvv)\g(0,\iota_1,\iota_2,\vvv)\ud \vvv=\mf.
    \end{array}
    \right.
\end{align}

The following result is a generalization of \cite[Theorem 2.1]{AA023} and \cite{Bardos.Caflisch.Nicolaenko1986,BB002, AA004} when we consider nonzero mass flux $\mf$:
\begin{proposition}\label{boundary well-posedness}
Assume $\hhh\in W^{k,\infty}_{\iota_1,\iota_2}$ for some $k\in\mathbb{N}$. Given $\mf\in\r$. Then there exists 
\begin{align}
    \g_{\infty}(\vvv)=\mh\bigg(\rq^{\infty}+\vv\cdot\uq^{\infty}+\frac{\abs{\vv}^2-3}{2}\tq^{\infty}\bigg)\in\nk
\end{align}
and a unique solution $\g(\eta,\vvv)$ to \eqref{Milne} such that $\ggg:=\g-\g_{\infty}$ satisfies 
\begin{align}\label{Milne.}
        \left\{
        \begin{array}{l}
        \va\dfrac{\p\ggg }{\p\eta}+\nu \ggg-K\left[\ggg\right]=0,\\\rule{0ex}{1.2em}
        \ggg(0,\vvv)=\hhh(\vvv)-   \g_{\infty}(\vvv):=\hh(\vvv)\ \ \text{for}\ \ \va>0,\\\rule{0ex}{1.5em}
        \ds\int_{\r^3}\va\mh(\vvv)\ggg(0,\vvv)\ud \vvv=0,
        \end{array}
        \right.
\end{align}
and for some $K_0>0$ and any $0<\N\leq k$
\begin{align}
    \abs{\g_{\infty}-\va\mh\mf}+\lnmm{\ue^{K_0\eta}\ggg}\ls& \lnmms{\hhh}{\gamma_-}+\abs{\mf},\label{Milne 01}\\ \lnmm{\ue^{K_0\eta}\va\p_{\eta}\ggg}+\lnmm{\ue^{K_0\eta}\va\p_{\va}\ggg}\ls& \lnmms{\hhh}{\gamma_-}+\lnmms{\nabla_{\vvv}\hhh}{\gamma_-}+\abs{\mf},\label{Milne 02}\\
    \lnmm{\ue^{K_0\eta}\p_{\vb}\ggg}+\lnmm{\ue^{K_0\eta}\p_{\vc}\ggg}\ls& \lnmms{\hhh}{\gamma_-}+\lnmms{\nabla_{\vvv}\hhh}{\gamma_-}+\abs{\mf},\label{Milne 03}\\
    \lnmm{\ue^{K_0\eta}\p_{\iota_1}^{\N}\ggg}+\lnmm{\ue^{K_0\eta}\p_{\iota_2}^{\N}\ggg}\ls& \lnmms{\hhh}{\gamma_-}+\sum_{j=1}^{\N}\lnmms{\p_{\iota_1}^j\hhh}{\gamma_-}+\sum_{j=1}^{\N}\lnmms{\p_{\iota_2}^j\hhh}{\gamma_-}\label{Milne 04}\\
    &+\abs{\mf}+\sum_{j=1}^{\N}\abs{\p_{\iota_1}^j\mf}+\sum_{j=1}^{\N}\abs{\p_{\iota_2}^j\mf}.\no
\end{align}
\end{proposition}

\begin{remark}
    Denote $\uq^{\infty}=\left(\uq^{\infty}_n,\uq^{\infty}_{\iota_1},\uq^{\infty}_{\iota_2}\right)$,
    for the normal component and two tangential components. Direct integrating over $\r^3$ in \eqref{Milne.}, we obtain
    \begin{align}
        \mf=\int_{\r^3}\va\mh(\vvv)\g(0,\vvv)=\int_{\r^3}\va\mh(\vvv)\g_{\infty}(\vvv)=\uq^{\infty}_n.
    \end{align}
    Hence, $\g-\va\mh\mf$ satisfies \eqref{Milne.} with zero mass flux. Then we may directly apply the well-known results from \cite[Theorem 2.1]{AA023} and \cite{Bardos.Caflisch.Nicolaenko1986,BB002, AA004}.
\end{remark}

\paragraph{\underline{BV Estimates}}
For $\ggg(\vxx,\vvv)$, denote the semi-norm
\begin{align}
    \nm{\ggg}_{\widetilde{\text{BV}}}:=\sup\bigg\{\iint_{\eta,\vvv}\ggg(\nabla_{\eta,\vvv}\cdot\psi)\ud\eta\ud\vvv:\ \psi\in C^1_c\ \text{and}\ \nm{\psi}_{L^{\infty}}\leq 1\bigg\},
\end{align}
and thus the BV norm can be defined as
\begin{align}
    \bnm{\ggg}:=\pnm{\ggg}{1}+\nm{\ggg}_{\widetilde{\text{BV}}}.
\end{align}
It is classical that $W^{1,1}\hookrightarrow \text{BV}$. The following result comes from \cite[Theorem 2.16]{AA023}:
\begin{proposition}\label{boundary regularity}
We have
\begin{align}
    \bnm{\nu\ggg}\ls\lnmms{\hh}{\gamma_-}+\int_{\va>0}\abs{\p_{\va}\hh }\va\ud\vvv+\int_{\va>0}\abs{\hh}\ud\vvv+\abs{\mf}.
\end{align}
\end{proposition}

\paragraph{\underline{Mass Flux}}
The mass flux $\mf$ in \eqref{Milne.} has an interesting interaction with $\rq^{\infty}+\tq^{\infty}$ defined in $\g_{\infty}$:

\begin{proposition}\label{boundary flux}
    Given $\hhh$ as in Proposition \ref{boundary well-posedness}. For any given constant $P$, there exists $\mf(\iota_1,\iota_2)$ such that $\g_{\infty}$ in Proposition \ref{boundary well-posedness} satisfies
    \begin{align}
        \rq^{\infty}(\iota_1,\iota_2)+\tq^{\infty}(\iota_1,\iota_2)=P.
    \end{align}
\end{proposition}
\begin{proof}
We first solve a zero-mass flux problem for $\overline{\g}$:
\begin{align}\label{Milne=}
    \left\{
    \begin{array}{l}
    \va\dfrac{\p\overline{\g}}{\p\eta}+\nu\overline{\g}- K\left[\overline{\g}\right]=0,\\\rule{0ex}{1.2em}
    \overline{\g}(0,\iota_1,\iota_2,\vvv)=\hhh(\iota_1,\iota_2,\vvv)\ \ \text{for}\ \ \va>0,\\\rule{0ex}{1.5em}
    \ds\int_{\r^3}\va\mh(\iota_1,\iota_2,\vvv)\overline{\g}(0,\iota_1,\iota_2,\vvv)\ud \vvv=0.
    \end{array}
    \right.
\end{align}
Based on Proposition \ref{boundary well-posedness}, there exists a limit function
\begin{align}
    \overline{\g}_{\infty}(\vvv)=\mh\bigg(\overline{\rq}^{\infty}+\vv\cdot\overline{\uq}^{\infty}+\frac{\abs{\vv}^2-3}{2}\overline{\tq}^{\infty}\bigg)\in\nk.
\end{align}
If we already have $P=\overline{\rq}^{\infty}+\overline{\tq}^{\infty}$, then simply take $\mf=0$. Otherwise, we consider the following auxiliary problem for $\widetilde{\g}$
\begin{align}\label{Milne==}
    \left\{
    \begin{array}{l}
    \va\dfrac{\p\widetilde{\g}}{\p\eta}+\nu\widetilde{\g}- K[\widetilde{\g}]=0,\\\rule{0ex}{1.2em}
    \widetilde{\g}(0,\iota_1,\iota_2,\vvv)=0\ \ \text{for}\ \ \va>0,\\\rule{0ex}{1.5em}
    \ds\int_{\r^3}\va\mh(\iota_1,\iota_2,\vvv)\widetilde{\g}(0,\iota_1,\iota_2,\vvv)\ud \vvv=\widetilde{\mf}.
    \end{array}
    \right.
\end{align}
Based on Proposition \ref{boundary well-posedness}, there exists a limit function
\begin{align}
    \widetilde{\g}_{\infty}(\vvv)=\mh\left(\widetilde{\rq}^{\infty}+\vv\cdot\widetilde{\uq}^{\infty}+\frac{\abs{\vv}^2-3}{2}\widetilde{\tq}^{\infty}\right)\in\nk.
\end{align}
Without loss of generality, we may consider \eqref{Milne==} for fixed $(\iota_1,\iota_2)$. Multiplying $\widetilde{\g}$ on both sides of \eqref{Milne==} and integrating over $(\eta,\vvv)\in\rp\times\r^3$, we obtain
\begin{align}
    \int_{\r^3}\va\babs{\widetilde{\g}_{\infty}}^2-\int_{\r^3}\va\babs{\widetilde{\g}(0)}^2+\int_0^{\infty}\int_{\r^3}\widetilde{\g}\Big(\nu\widetilde{\g}- K[\widetilde{\g}]\Big)=0.
\end{align}
When $\mf\neq0$, direct computation reveals that $\widetilde{g}$ cannot be in the kernel $\nk$. Then based on the proof of Proposition \ref{boundary well-posedness}, we know
\begin{align}
    \int_0^{\infty}\int_{\r^3}\widetilde{\g}\Big(\nu\widetilde{\g}- K[\widetilde{\g}]\Big)>0
\end{align}
Clearly, from the boundary condition in \eqref{Milne==}, we know
\begin{align}
   -\int_{\r^3}\va\abs{\widetilde{\g}(0)}^2=-\int_{\va<0}\va\abs{\widetilde{\g}(0)}^2\geq0.
\end{align}
Hence, we must have
\begin{align}
    \int_{\r^3}\va\babs{\widetilde{\g}_{\infty}}^2=2\widetilde{\uq}^{\infty}_n\left(\widetilde{\rq}^{\infty}+\widetilde{\tq}^{\infty}\right)=2\widetilde{\mf}\left(\widetilde{\rq}^{\infty}+\widetilde{\tq}^{\infty}\right)<0.
\end{align}
Hence, when $\widetilde{\mf}\neq0$, we must also have $\widetilde{\rq}^{\infty}+\widetilde{\tq}^{\infty}\neq0$. Since \eqref{Milne==} is a linear equation, we know that $\widetilde{\mf}$ is proportional to $\widetilde{\rq}^{\infty}+\widetilde{\tq}^{\infty}$, i.e. there exists a nonzero constant $D$ such that
\begin{align}
    \widetilde{\mf}=D\left(\widetilde{\rq}^{\infty}+\widetilde{\tq}^{\infty}\right).
\end{align}
Then consider the sum of \eqref{Milne=} and $\dfrac{P-\overline{\rq}^{\infty}-\overline{\tq}^{\infty}}{\widetilde{\rq}^{\infty}+\widetilde{\tq}^{\infty}}\times$\eqref{Milne==}, we know that the limit function is 
\begin{align}
    \Big(\overline{\rq}^{\infty}+\overline{\tq}^{\infty}\Big)+\dfrac{P-\overline{\rq}^{\infty}-\overline{\tq}^{\infty}}{\widetilde{\rq}^{\infty}+\widetilde{\tq}^{\infty}}\times\left(\widetilde{\rq}^{\infty}+\widetilde{\tq}^{\infty}\right)=P
\end{align}
which satisfies the requirement. In this case, we know the mass flux
\begin{align}
    \mf=\dfrac{P-\overline{\rq}^{\infty}-\overline{\tq}^{\infty}}{\widetilde{\rq}^{\infty}+\widetilde{\tq}^{\infty}}\times\widetilde{\mf}=D\Big(P-\overline{\rq}^{\infty}-\overline{\tq}^{\infty}\Big).
\end{align}
\end{proof}

\begin{corollary}\label{boundary zero flux}
    Given $\hhh$ as in Proposition \ref{boundary well-posedness}. There exists constant $P$ and $\mf(\iota_1,\iota_2)$ such that $\g_{\infty}$ in Proposition \ref{boundary well-posedness} satisfies
    \begin{align}
        \rq^{\infty}(\iota_1,\iota_2)+\tq^{\infty}(\iota_1,\iota_2)=P,
    \end{align}
    and
    \begin{align}
        \int_{\p\Omega}\mf=\int_{\p\Omega}\uq^{\infty}_n=0.
    \end{align}
\end{corollary}
\begin{proof}
Based on the proof of Proposition \ref{boundary flux}, we have
\begin{align}
    \int_{\p\Omega}\mf=&\int_{\p\Omega}D\Big(P-\overline{\rq}^{\infty}-\overline{\tq}^{\infty}\Big)=0,
\end{align}
and thus we must take
\begin{align}
    P=\frac{1}{\abs{\p\Omega}}\int_{\p\Omega}\Big(\overline{\rq}^{\infty}+\overline{\tq}^{\infty}\Big).
\end{align}
Then the desired result follows from Proposition \ref{boundary flux}.
\end{proof}

\subsubsection{Boundary Layer}

Let $\blf$ be solution to the Milne problem \eqref{Milne problem} with $\mf$ solved from Corollary \ref{boundary zero flux} with $\hhh=\fss_b$.
Based on Proposition \ref{boundary well-posedness}, we know that there exists 
\begin{align}\label{extra 17}
    \blfi(\iota_1,\iota_2,\vvv):=\mh\left(\rq^B(\iota_1,\iota_2)+\vv\cdot\uq^B(\iota_1,\iota_2)+\frac{\abs{\vv}^2-3}{2}\tq^B(\iota_1,\iota_2)\right)\in\nk,
\end{align}
such that 
\begin{align}
    \int_{\r^3}\va\mh(\vvv)\blfi(\iota_1,\iota_2,\vvv)\ud \vvv=\mf,
\end{align}
and for some $K_0>0$, $\blff(\vxx,\vvv)=\blf(\vxx,\vvv)-\blfi(\iota_1,\iota_2,\vvv)$ satisfies
\begin{align}
    \abs{\blfi}+\lnmm{\ue^{K_0\eta}\blff}\ls\lnmms{\fss_b}{\gamma_-}.
\end{align}
Let $\chi(y)\in C^{\infty}(\r)$
and $\ch(y)=1-\chi(y)$ be smooth cut-off functions satisfying
\begin{align}
    \chi(y)=\left\{
    \begin{array}{ll}
    1&\ \ \text{if}\ \ \abs{y}\leq1,\\
    0&\ \ \text{if}\ \ \abs{y}\geq2,
    \end{array}
    \right.
\end{align}
We define a cutoff boundary layer $\fb_1$. Denote
\begin{align}\label{boundary layer}
    \fb_1(\vxx,\vvv)=\ch\left(\e^{-1}\va\right)\chi(\e\eta)\blff (\vxx,\vvv).
\end{align}
We may verify that $\fb_1$ satisfies
\begin{align}
    \va\dfrac{\p\fb_1 }{\p\eta}+\lc\left[\fb_1\right]=\va\ch(\e^{-1}\va)\frac{\p\chi(\e\eta)}{\p\eta}\blff+\chi(\e\eta)\bigg(\ch(\e^{-1}\va)Kg[\blff]-K\Big[\ch(\e^{-1}\va)\blff\Big]\bigg),
\end{align}
with
\begin{align}
    \fb_1(0,\iota_1,\iota_2,\vvv)=\ch\left(\e^{-1}\va\right)\Big(\fss_b(\iota_1,\iota_2,\vvv)-\blfi(\iota_1,\iota_2,\vvv)\Big)\ \ \text{for}\ \ \va>0.
\end{align}
Due to the cutoff $\ch$, $\fb_1$ may not preserve the mass-flux, i.e.
\begin{align}
    \int_{\r^3}\va\mh(\vvv)\fb_1(0,\iota_1,\iota_2,\vvv)\ud \vvv=&\int_{\r^3}\va\mh(\vvv)\ch\left(\e^{-1}\va\right)\blff(0,\iota_1,\iota_2,\vvv)\ud \vvv\\
    =&\int_{\r^3}\va\mh(\vvv)\chi\left(\e^{-1}\va\right)\blff(0,\iota_1,\iota_2,\vvv)\ud \vvv=\mf+O(\e).\no
\end{align}

\subsubsection{Matching Procedure}\label{sec:matching}

Based on Proposition \ref{boundary well-posedness}, we have 
\begin{align}
    \abs{P}=\abs{\rq^B+\tq^B}\ls\oot,\quad \abs{\mf}\ls\oot,
\end{align}
and for some $K_0>0$ and any $0<\N\leq 3$
\begin{align}\label{btt 01}
    \lnmm{\ue^{K_0\eta}\fb_1}+\lnmm{\ue^{K_0\eta}\frac{\p^\N\fb_1}{\p\iota_1^\N}}+\lnmm{\ue^{K_0\eta}\frac{\p^\N\fb_1}{\p\iota_2^\N}}&\ls\oot.
\end{align}
Considering the boundary condition in \eqref{large system} and the expansion \eqref{expand}, we require the matching condition for $\vx_0\in\p\Omega$ and $v\cdot\vn<0$:
\begin{align}
    \mh\left(\f_1+\fb_1\right)\Big|_{v\cdot n<0}&=\fss_b.
\end{align}
Hence, it suffices to define 
\begin{align}\label{btt 02}
    \rq_1\Big|_{\p\Omega}=\rq^B,\quad\uq_1\Big|_{\p\Omega}=\uq^B,\quad\tq_1\Big|_{\p\Omega}=\tq^B. 
\end{align}
Therefore, from \eqref{btt 01} and \eqref{btt 02}, we know
\begin{align}
    \abs{\rq_1}_{W^{3,\infty}}+\abs{\uq_1}_{W^{3,\infty}}+\abs{\tq_1}_{W^{3,\infty}}\ls\oot.
\end{align}
In particular, we know
\begin{align}
    \big(\rq_1+\tq_1\big)\Big|_{\p\Omega}=\rq^B+\tq^B=P,\quad \int_{\Omega}\big(\nx\cdot\uq_1\big)=\int_{\p\Omega}\big(\uq_1\cdot n\big)=\int_{\p\Omega}\mf=0.
\end{align}
By standard fluid theory \cite{Boyer.Fabrie2013, Cattabriga1961} for the steady Navier-Stokes equations \eqref{fluid}, we have for any $\NN\in[2,\infty)$
\begin{align}
    \nm{\rq_1}_{W^{3,\NN}}+\nm{\uq_1}_{W^{3,\NN}}+\nm{\tq_1}_{W^{3,\NN}}\ls\oot.
\end{align}
Then for $\f_2$, there is no corresponding boundary layer, and thus we may simply take \begin{align}\label{btt 03}
    \rq_2\Big|_{\p\Omega}=0,\quad\uq_2\Big|_{\p\Omega}=-\frac{1}{\abs{\p\Omega}}\int_{\Omega}\big(\uq_1\cdot\nx\rq_1\big),\quad\tq_2\Big|_{\p\Omega}=0. 
\end{align}
By standard fluid theory \cite{Boyer.Fabrie2013, Cattabriga1961} for the linear steady Navier-Stokes equations \eqref{fluid'}, we have for any $\NN\in[2,\infty)$
\begin{align}
    \nm{\rq_2}_{W^{2,\NN}}+\nm{\uq_2}_{W^{2,\NN}}+\nm{\tq_2}_{W^{2,\NN}}\ls\oot.
\end{align}

\begin{theorem}\label{thm:approximate-solution}
    Under the assumption \eqref{assumption:stationary}, there exists a unique solution $(\rq_1,\uq_1,\tq_1$ to the steady Navier-Stokes equations \eqref{fluid} satisfying for any $\NN\in[2,\infty)$
    \begin{align}
    \nm{\rq_1}_{W^{3,\NN}}+\nm{\uq_1}_{W^{2,\NN}}+\nm{\tq_1}_{W^{3,\NN}}\ls\oot.
    \end{align}
    Also, we can construct $\f_1$, $\f_2$ and $\fb_1$ such that 
    \begin{align}
    \nm{f_1}_{W^{3,\NN}L^{\infty}_{\varrho,\vartheta}}+\abs{f_1}_{W^{3-\frac{1}{\NN},\NN}L^{\infty}_{\varrho,\vartheta}}&\ls\oot,\\
    \nm{f_2}_{W^{2,\NN}L^{\infty}_{\varrho,\vartheta}}+\abs{f_2}_{W^{2-\frac{1}{\NN},\NN}L^{\infty}_{\varrho,\vartheta}}&\ls\oot,
    \end{align}
and for some $K_0>0$ and any $0<\N\leq 3$
\begin{align}
    \lnmm{\ue^{K_0\eta}\fb_1}+\lnmm{\ue^{K_0\eta}\frac{\p^\N\fb_1}{\p\iota_1^\N}}+\lnmm{\ue^{K_0\eta}\frac{\p^\N\fb_1}{\p\iota_2^\N}}&\ls\oot.
\end{align}
\end{theorem}


\subsection{Remainder Equation}

Inserting \eqref{expand} into \eqref{large system}, we have
\begin{align}
    \vv\cdot\nx\left(\m+\f+\fb+\e\mh\re\right)=\e^{-1}\qq\left[\m+\f+\fb+\e\mh\re,\m+\f+\fb+\e\mh\re\right]
\end{align}
or equivalently
\begin{align}
    \vv\cdot\nx\re-2\e^{-1}\mhh\qq\left[\m,\mh\re\right]=&-\e^{-1}\mhh\Big(\vv\cdot\nx\left(\f+\fb\right)\Big)+\mhh\qq\left[\mh\re,\mh\re\right]\\
    &+2\e^{-1}\mhh\qq\left[\f+\fb,\mh\re\right]+\e^{-2}\mhh\qq\left[\m+\f+\fb,\m+\f+\fb\right].\no
\end{align}
Also, we have the boundary condition
\begin{align}
    \left(\m+\f+\fb+\e\mh\re\right)\Big|_{\gamma_-}=\m+\e\mh\fss_b,
\end{align}
which is equivalent to 
\begin{align}
    \re\Big|_{\gamma_-}=\fss_b-\e^{-1}\mhh\big(\f+\fb\big).
\end{align}
Therefore, we need to consider the remainder equation \eqref{remainder}.
Here the boundary data is given by
\begin{align}\label{d:h}
    \h=\Big(-\e\f_2+\chi\big(\e^{-1}\va\big)\blff\Big)\Big|_{\gamma_-},
\end{align}
and
\begin{align}
    \ss:=\ssa+\ssc+\ssd+\ssf+\ssg+\ssh,
\end{align}
where
\begin{align}
    \ssa:=&-\e\vv\cdot\nx\f_2,\label{d:ssa}\\
    \ssc:=&\dfrac{1}{R_1-\e\eta}\bigg(\vb^2\dfrac{\p\fb_1}{\p\va}-\va\vb\dfrac{\p\fb_1}{\p\vb}\bigg)+\dfrac{1}{R_2-\e\eta}\bigg(\vc^2\dfrac{\p \fb_1}{\p\va}-\va\vc\dfrac{\p\fb_1}{\p\vc}\bigg)\label{d:ssc}\\
    &-\dfrac{1}{\pl_1\pl_2}\left(\dfrac{R_1\p_{\iota_1\iota_1}\vr\cdot\p_{\iota_2}\vr}{\pl_1(R_1-\e\eta)}\vb\vc
    +\dfrac{R_2\p_{\iota_1\iota_2}\vr\cdot\p_{\iota_2}\vr}{\pl_2(R_2-\e\eta)}\vc^2\right)\dfrac{\p\fb_1}{\p\vb}\no\\
    &-\dfrac{1}{\pl_1\pl_2}\left(\dfrac{R_2\p_{\iota_2\iota_2}\vr\cdot\p_{\iota_1}\vr}{\pl_2(R_2-\e\eta)}\vb\vc
    +\dfrac{R_1\p_{\iota_1\iota_2}\vr\cdot\p_{\iota_1}\vr}{\pl_1(R_1-\e\eta)}\vb^2\right)\dfrac{\p\fb_1}{\p\vc}\no\\
    &-\left(\dfrac{R_1\vb}{\pl_1(R_1-\e\eta)}\dfrac{\p\fb_1}{\p\iota_1}+\dfrac{R_2\vc}{\pl_2(R_2-\e\eta)}\dfrac{\p\fb_1}{\p\iota_2}\right)\no\\
    &+\e^{-1}\va\ch(\e^{-1}\va)\frac{\p\chi(\e\eta)}{\p\eta}\blff-\e^{-1}\Big(K\left[\blff\right]\chi(\e^{-1}\va)\chi(\e\eta)-K\Big[\blff\chi(\e^{-1}\va)\chi(\e\eta)\Big]\Big),\no\\
    \ssd:=&2\mhh\qq\left[\mh\f_1+\e\mh\f_2,\mh\re\right]=2\Gamma[\f_1+\e\f_2,\re],\label{d:ssd}\\
    \ssf:=&2\mhh\qq\left[\mh\fb_1,\mh\re\right]=2\Gamma\left[\fb_1,\re\right],\label{d:ssf}\\
    \ssg:=&\e\mhh\qq\left[\mh\f_2,\mh\big(2\f_1+\e\f_2\big)\right]+2\mhh\qq\left[\mh\big(2\f_1+2\e\f_2+\fb_1\big),\mh\fb_1\right]\label{d:ssg}\\
    =&\e\Gamma\left[\f_2,2\f_1+\e\f_2\right]+\Gamma\left[2\f_1+2\e\f_2+\fb_1,\fb_1\right],\no\\
    \ssh:=&\mhh\qq\left[\mh\re,\mh\re\right]=\Gamma[\re,\re].\label{d:ssh}
\end{align}
In particular, we may further split $\ssc$:
\begin{align}
    \ssx:=&\dfrac{1}{R_1-\e\eta}\bigg(\vb^2\dfrac{\p\fb_1}{\p\va}\bigg)
    +\dfrac{1}{R_2-\e\eta}\bigg(\vc^2\dfrac{\p \fb_1}{\p\va}\bigg),\\
    \ssy:=&-\dfrac{1}{R_1-\e\eta}\bigg(\va\vb\dfrac{\p\fb_1}{\p\vb}\bigg)
    -\dfrac{1}{R_2-\e\eta}\bigg(\va\vc\dfrac{\p\fb_1}{\p\vc}\bigg)\\
    &-\dfrac{1}{\pl_1\pl_2}\left(\dfrac{R_1\p_{\iota_1\iota_1}\vr\cdot\p_{\iota_2}\vr}{\pl_1(R_1-\e\eta)}\vb\vc
    +\dfrac{R_2\p_{\iota_1\iota_2}\vr\cdot\p_{\iota_2}\vr}{\pl_2(R_2-\e\eta)}\vc^2\right)\dfrac{\p\fb_1}{\p\vb}\no\\
    &-\dfrac{1}{\pl_1\pl_2}\left(\dfrac{R_2\p_{\iota_2\iota_2}\vr\cdot\p_{\iota_1}\vr}{\pl_2(R_2-\e\eta)}\vb\vc
    +\dfrac{R_1\p_{\iota_1\iota_2}\vr\cdot\p_{\iota_1}\vr}{\pl_1(R_1-\e\eta)}\vb^2\right)\dfrac{\p\fb_1}{\p\vc}\no\\
    &-\left(\dfrac{R_1\vb}{\pl_1(R_1-\e\eta)}\dfrac{\p\fb_1}{\p\iota_1}+\dfrac{R_2\vc}{\pl_2(R_2-\e\eta)}\dfrac{\p\fb_1}{\p\iota_2}\right)+\e^{-1}\va\ch(\e^{-1}\va)\frac{\p\chi(\e\eta)}{\p\eta}\blff,\no\\
    \ssz:=&-\e^{-1}\Big(K\left[\blff\right]\chi(\e^{-1}\va)\chi(\e\eta)-K\Big[\blff\chi(\e^{-1}\va)\chi(\e\eta)\Big]\Big).
\end{align}

\begin{lemma}[Green's Identity, Lemma 2.2 of \cite{Esposito.Guo.Kim.Marra2013}]\label{lem:green-identity}
Assume $f(\vx,\vv),\ g(\vx,\vv)\in L^2_{\nu}(\Omega\times\r^3)$ and
$\vv\cdot\nx f,\ \vv\cdot\nx g\in L^2(\Omega\times\r^3)$ with $f,\
g\in L^2_{\gamma}$. Then 
\begin{align}
&\bbr{\vv\cdot\nx f, g}+\bbr{\vv\cdot\nx
g,f}
=\int_{\gamma}fg\big(v\cdot n\big).
\end{align}
\end{lemma}

Using Lemma \ref{lem:green-identity}, we can derive the weak formulation of \eqref{remainder}. For any test function $\test(x,v)\in L^2_{\nu}(\Omega\times\r^3)$ with $\vv\cdot\nx\test\in L^2(\Omega\times\r^3)$ with $\test\in L^2_{\gamma}$, we have
\begin{align}\label{weak formulation}
    \int_{\gamma}\re\test\big(v\cdot n\big)-\bbr{v\cdot\nx \test, \re}+\e^{-1}\bbr{\lc[\re],\test}&=\bbr{\ss,\test}.
\end{align}

\subsubsection{Estimates of Boundary and Source Terms}\label{sec:source}

The estimates below follow from analogous argument as \cite[Section 3]{AA023} with $\alpha=1$, so we omit the details and only present the results. In the following, assume that $g$ is a given function and $2\leq\N\leq 6$.

\paragraph{\underline{Estimates of $\h$}}

\begin{lemma}\label{h-estimate}
Under the assumption \eqref{assumption:stationary}, for $h$ defined in \eqref{d:h}, we have
\begin{align}
    \tnms{h}{\gamma_-}&\ls\oot\e,\quad
    \jnms{h}{\gamma_-}\ls\oot\e^{\frac{3}{\N}},\quad
    \lnmms{h}{\gamma_-}\ls\oot,\quad \sup_{\iota_1,\iota_2}\int_{v\cdot n<0}\abs{h}\abs{v\cdot n}\ud v\ls\oot\e.
\end{align}
\end{lemma}

\begin{remark}\label{rmk:boundary}
We may directly compute that for $x_0\in\p\Omega$
\begin{align}
    \bb(x_0)\cdot n=&\int_{\r^3}\re(\vx_0)\mh(v\cdot n)\ud v=\int_{v\cdot n<0}h(\vx_0)\mh(v\cdot n)\ud v+\int_{v\cdot n>0}\re(\vx_0)\mh(v\cdot n)\ud v.
\end{align}
\end{remark}

\paragraph{\underline{Estimates of $\ssa$}}

\begin{lemma}\label{s1-estimate}
Under the assumption \eqref{assumption:stationary},  for $\ssa$ defined in \eqref{d:ssa}, we have
\begin{align}
    \btnm{\!\br{v}^2\!\ssa}\ls \oot\e,\qquad
    \pnm{\ssa}{\N}\ls\oot\e,\qquad
    \lnmm{\ssa}\ls\oot\e.
\end{align}
Also, we have the orthogonality property
\begin{align}
    \brv{\mh,\ssa}=0,\quad\brv{\mh\vv,\ssa}=\od,\quad \brv{\mh\abs{\vv}^2,\ssa}=0.
\end{align}
\end{lemma}

\paragraph{\underline{Estimates of $\ssc$}}

\begin{lemma}\label{s2-estimate}
Under the assumption \eqref{assumption:stationary},  for $\ssc$ defined in \eqref{d:ssc}, we have
\begin{align}
    \pnm{\ssc}{1}+\pnm{\eta\left(\ssy+\ssz\right)}{1}+\pnm{\eta^2\left(\ssy+\ssz\right)}{1}&\ls \oot \e,\label{ss3-estimate1}\\
    \tnm{\br{v}^2\ssc}+\tnm{\eta\left(\ssy+\ssz\right)}+\tnm{\eta^2\left(\ssy+\ssz\right)}&\ls \oot,\label{ss3-estimate2}\\
    \pnm{\ssc}{\N}+\pnm{\eta\left(\ssy+\ssz\right)}{\N}+\pnm{\eta^2\left(\ssy+\ssz\right)}{\N}&\ls\oot\e^{\frac{2}{\N}-1},\label{ss3-estimate3}\\
    \nm{\ssc}_{L^{\N}_{\iota_1\iota_2}L^1_{\mn}L^1_v}+\nm{\eta\left(\ssy+\ssz\right)}_{L^{\N}_{\iota_1\iota_2}L^1_{\mn}L^1_v}&\ls\oot\e,\label{ss3-estimate4},
\end{align}
and
\begin{align}
    \nm{\ssy+\ssz}_{L^{\N}_{x}L^1_v}+\nm{\eta\left(\ssy+\ssz\right)}_{L^{\N}_{x}L^1_v}&\ls\oot\e^{\frac{1}{\N}},\label{ss3-estimate5}\\
    \abs{\br{\ssx,g}}+\abs{\br{\eta\ssx,g}}+\abs{\br{\eta^2\ssx,g}}&\ls\knm{\br{v}^2\fb_1}\jnm{\nabla_v g}\ls\oot\e^{1-\frac{1}{\N}}\jnm{\nabla_v g}.\label{ss3-estimate6}
\end{align}
Also, we have
\begin{align}
    \lnmm{\ssc}&\ls\oot\e^{-1}.
\end{align}
\end{lemma}

\begin{remark}\label{rmk:s2}
    Notice that the BV estimate in Theorem \ref{boundary regularity} does not contain exponential decay in $\eta$, and thus we cannot directly bound $\eta\ssx$ and $\eta^2\ssx$. Instead, we should first integrate by parts with respect to $\va$ as in \eqref{ss3-estimate6} to study $\fb_1$ instead:
    \begin{align}
        \pnm{\fb_1}{\N}+\pnm{\eta\fb_1}{\N}+\pnm{\eta^2\fb_1}{\N}&\ls\oot\e^{\frac{2}{\N}-1},\label{ss3-estimate7}\\
        \nm{\fb_1}_{L^{\N}_{\iota_1\iota_2}L^1_{\mn}L^1_v}+\nm{\eta\fb_1}_{L^{\N}_{\iota_1\iota_2}L^1_{\mn}L^1_v}&\ls\oot\e,\label{ss3-estimate8},\\
        \nm{\fb_1}_{L^{\N}_{x}L^1_v}+\nm{\eta\fb_1}_{L^{\N}_{x}L^1_v}&\ls\oot\e^{\frac{1}{\N}}.
    \end{align}
\end{remark}

\paragraph{\underline{Estimates of $\ssd$}}

\begin{lemma}\label{s3-estimate}
Under the assumption \eqref{assumption:stationary},  for $\ssd$ defined in \eqref{d:ssd}, we have
\begin{align}
    \abs{\brv{\ssd,g}}\ls \oot\left(\int_{\r^3}\nu\abs{g}^2\right)^{\frac{1}{2}}\left(\int_{\r^3}\nu\abs{\re}^2\right)^{\frac{1}{2}},
\end{align}
and thus
\begin{align}
    \abs{\br{\ssd,g}}\ls \oot\um{g}\um{\re}\ls\oot\um{g}\left(\tnm{\bre}+\um{\ire}\right).
\end{align}
Also, we have
\begin{align}
    \tnm{\ssd}&\ls\oot\um{\re},\quad
    \lnmm{\nu^{-1}\ssd}\ls\oot\lnmm{\re}.
\end{align}
\end{lemma}

\paragraph{\underline{Estimates of $\ssf$}}

\begin{lemma}\label{s4-estimate}
Under the assumption \eqref{assumption:stationary},  for $\ssf$ defined in \eqref{d:ssf}, we have
\begin{align}
    \abs{\brv{\ssf,g}}\ls \left(\int_{\r^3}\nu\abs{g}^2\right)^{\frac{1}{2}}\left(\int_{\r^3}\nu\abs{\fb_1}^2\right)^{\frac{1}{2}}\left(\int_{\r^3}\nu\abs{\re}^2\right)^{\frac{1}{2}},
\end{align}
and thus
\begin{align}
    \abs{\br{\ssf,g}}\ls& \oot\um{g}\um{\re}\ls\oot\um{g}\left(\tnm{\bre}+\um{\ire}\right),\\
    \abs{\br{\ssf,g}}\ls& \oot\um{\fb_1}\lnmm{g}\um{\re}\ls\oot\e^{\frac{1}{2}}\lnmm{g}\left(\tnm{\bre}+\um{\ire}\right).
\end{align}
Also, we have
\begin{align}
    \tnm{\ssf}&\ls\oot\um{\re},\quad
    \lnmm{\nu^{-1}\ssf}\ls\oot\lnmm{\re}.
\end{align}
\end{lemma}

\paragraph{\underline{Estimates of $\ssg$}}

\begin{lemma}\label{s5-estimate}
Under the assumption \eqref{assumption:stationary},  for $\ssg$ defined in \eqref{d:ssg}, we have
\begin{align}
    \abs{\brv{\ssg,g}}\ls \oot\left(\int_{\r^3}\nu\abs{g}^2\right)^{\frac{1}{2}},
\end{align}
and thus
\begin{align}
    \abs{\br{\ssg,g}}\ls&\oot\e^{\frac{1}{2}}\um{g},\quad
    \abs{\br{\ssg,g}}\ls\oot\e\lnmm{g}.
\end{align}
Also, we have
\begin{align}
    \tnm{\ssg}&\ls\oot\e^{\frac{1}{2}},\quad
    \lnmm{\nu^{-1}\ssg}\ls\oot.
\end{align}
\end{lemma}

\paragraph{\underline{Estimates of $\ssh$}}

\begin{lemma}\label{s6-estimate}
Under the assumption \eqref{assumption:stationary},  for $\ssh$ defined in \eqref{d:ssh}, we have
\begin{align}
    \abs{\brv{\ssh,g}}\ls \left(\int_{\r^3}\nu\abs{g}^2\right)^{\frac{1}{2}}\left(\int_{\r^3}\nu\abs{\re}^2\right),
\end{align}
and thus
\begin{align}
    \abs{\br{\ssh,g}}\ls& \um{g}\Big(\um{\ire}\lnmm{\re}+\pnm{\bre}{3}\pnm{\bre}{6}\Big)
    \ls\um{g}\xnm{\re}^2.
\end{align}
Also, we have
\begin{align}
    \tnm{\ssh}&\ls\Big(\um{\ire}\lnmm{\re}+\pnm{\bre}{3}\pnm{\bre}{6}\Big)\ls\xnm{\re}^2,\\
    \lnmm{\nu^{-1}\ssh}&\ls\lnmm{\re}^2.
\end{align}
\end{lemma}

\subsubsection{Conservation Laws}

\paragraph{\underline{Classical Conservation Laws}}

\begin{lemma}\label{lem:conservation}
    Under the assumption \eqref{assumption:stationary}, we have the conservation laws
    \begin{align}
    \nx\cdot\bb=&\brv{\mh,\ss}=\brv{\mh,\ssc},\label{conservation law 1}\\
    \nx\P+\nx\cdot\varpi=&\brv{v\mh,\ss}=\brv{v\mh,\ssc},\label{conservation law 2}\\
    5\nx\cdot\bb+\nx\cdot\varsigma=&\brv{\abs{v}^2\mh,\ss}=\brv{\abs{v}^2\mh,\ssc},\label{conservation law 3}
    \end{align}
    where $\ds\varpi:=\int_{\r^3}\mh\big(v\otimes v\big)\ire$ and $\ds\varsigma:=\int_{\r^3}\mh v\abs{v}^2\ire$.
\end{lemma}
\begin{proof}
    We multiply test functions $\mh,v\mh,\abs{v}^2\mh$ on both sides of \eqref{remainder} and integrate over $v\in\r^3$. Using the orthogonality of $\lc$ and Lemma \ref{s1-estimate} (which comes from \eqref{att 01}), the results follow. 
\end{proof}

\paragraph{\underline{Conservation Law with Test Function $\nx\varphi\cdot\a$}}

\begin{lemma}\label{lem:conservation 1}
    Under the assumption \eqref{assumption:stationary}, for any smooth function $\varphi(x)$, we have
    \begin{align}\label{conservation law 4}
    &-\k\bbrx{\dx\varphi,c}+\e^{-1}\bbr{\nx\varphi,\varsigma}\\
    =&\bbrb{\nx\varphi\cdot\a,h}{\gamma_-}-\bbrb{\nx\varphi\cdot\a,\re}{\gamma_+}+\br{v\cdot\nx\Big(\nx\varphi\cdot\a\Big),\ire}+\bbr{\nx\varphi\cdot\a,\ss}.\no
    \end{align}
\end{lemma}
\begin{proof}
Taking test function $\test=\nx\varphi\cdot\a$ in \eqref{weak formulation}, we obtain
\begin{align}
    \int_{\gamma}\Big(\nx\varphi\cdot\a\Big)\re\big(v\cdot n\big)-\bbr{v\cdot\nx\Big(\nx\varphi\cdot\a\Big), \re}+\e^{-1}\bbr{\lc[\re],\nx\varphi\cdot\a}&=\bbr{\nx\varphi\cdot\a,\ss}.
\end{align}
Using the splitting \eqref{splitting}, oddness and orthogonality of $\a$, we deduce
\begin{align}
&-\k\bbrx{\dx\varphi,c}+\e^{-1}\bbr{\nx\varphi,\varsigma}\\
=&-\int_{\gamma}\Big(\nx\varphi\cdot\a\Big)\re (v\cdot n)+\br{v\cdot\nx\Big(\nx\varphi\cdot\a\Big),\ire}+\bbr{\nx\varphi\cdot\a,\ss}.\no
\end{align}
Notice that
\begin{align}\label{boundary decomposition}
    \re\id_{\gamma}=\re\id_{\gamma_+}+h\id_{\gamma_-},
\end{align}
we have \eqref{conservation law 4}.
\end{proof}

\paragraph{\underline{Conservation Law with Test Function $\nx\psi:\b$}}

\begin{lemma}\label{lem:conservation 2}
    Under the assumption \eqref{assumption:stationary}, for any smooth function $\psi(x)$ satisfying $\nx\cdot\psi=0$, we have
    \begin{align}\label{conservation law 5}
    &-\lambda\bbrx{\dx\psi,\bb}+\e^{-1}\bbr{\nx\psi,\varpi}\\
    =&\bbrb{\nx\psi\cdot\b,h}{\gamma_-}-\bbrb{\nx\psi\cdot\b,\re}{\gamma_+}+\br{v\cdot\nx\Big(\nx\psi:\b\Big),\ire}+\bbr{\nx\psi:\b,\ss}.\no
    \end{align}
    \end{lemma}
\begin{proof}
Taking test function $\test=\nx\psi:\b$ in \eqref{weak formulation}, we obtain
\begin{align}
    \int_{\gamma}\Big(\nx\psi:\b\Big)\re\big(v\cdot n\big)-\bbr{v\cdot\nx\Big(\nx\psi:\b\Big), \re}+\e^{-1}\bbr{\lc[\re],\nx\psi:\b}&=\bbr{\nx\psi:\b,\ss}.
\end{align}
Using the splitting \eqref{splitting}, oddness and orthogonality of $\b$, we deduce
\begin{align}
&-\br{v\cdot\nx\Big(\nx\psi:\b\Big),v\mh\cdot\bb}+\e^{-1}\bbr{\nx\psi,\varpi}\\
=&-\int_{\gamma}\Big(\nx\psi:\b\Big)\re (v\cdot n)+\br{v\cdot\nx\Big(\nx\psi:\b\Big),\ire}+\bbr{\nx\psi:\b,\ss}.\no
\end{align}
Here, we may further compute
\begin{align}
    \br{v\cdot\nx\Big(\nx\psi:\b\Big),v\mh\cdot\bb}=\bbr{\bbb\cdot\nx\big(\nx\psi:\b\big),\bb},
\end{align}
and use $\nx\cdot\psi=0$ to obtain
\begin{align}
  \int_{\r^3}\bbb\cdot\nx\big(\nx\psi:\b\big)=&\begin{pmatrix}
  \alpha\p_{11}\psi_1+(\gamma+\lambda)\p_{12}\psi_2+(\gamma+\lambda)\p_{13}\psi_3+\lambda\p_{22}\psi_1+\lambda\p_{33}\psi_1\\
  \alpha\p_{22}\psi_2+(\gamma+\lambda)\p_{12}\psi_1+(\gamma+\lambda)\p_{23}\psi_3+\lambda\p_{11}\psi_2+\lambda\p_{33}\psi_2\\
  \alpha\p_{33}\psi_3+(\gamma+\lambda)\p_{13}\psi_1+(\gamma+\lambda)\p_{23}\psi_2+\lambda\p_{11}\psi_3+\lambda\p_{22}\psi_3
  \end{pmatrix}\\
  =&\begin{pmatrix}
  (\alpha-\gamma-\lambda)\p_{11}\psi_1+\lambda\p_{22}\psi_1+\lambda\p_{33}\psi_1\\
  (\alpha-\gamma-\lambda)\p_{22}\psi_2+\lambda\p_{11}\psi_2+\lambda\p_{33}\psi_2\\
  (\alpha-\gamma-\lambda)\p_{33}\psi_3+\lambda\p_{11}\psi_3+\lambda\p_{22}\psi_3
  \end{pmatrix}=\lambda\begin{pmatrix}\Delta_x\psi_1\\\Delta_x\psi_2\\\Delta_x\psi_3\end{pmatrix}.\no
\end{align}
Here, we use the fact that $\b=\Upsilon\big(\abs{v}\big)\bbb$ for some function $\Upsilon$ that only depend on $\abs{v}$ (see \cite[Lemma 14]{Golse2014}). Then direct computation shows that $\alpha-\gamma-\lambda=\lambda$: for $i\neq j$
\begin{align}
    \alpha-\gamma-2\lambda
    =&\int_{\r^3}\Upsilon\big(\abs{v}\big)\Bigg(\left(v_i^2-\frac{1}{3}\abs{v}^2\right)^2-\left(v_i^2-\frac{1}{3}\abs{v}^2\right)\left(v_j^2-\frac{1}{3}\abs{v}^2\right)-2v_i^2v_j^2\Bigg)\m(v)\ud v\\
    =&\int_{\r^3}\Upsilon\big(\abs{v}\big)\Big(v_i^4-3v_i^2v_j^2\Big)\m(v)\ud v.\no
\end{align}
Then we use the spherical coordinates
\begin{align}
    v_i=\abs{v}\sin\theta\sin\varphi,\qquad v_j=\abs{v}\sin\theta\cos\varphi,
\end{align}
to estimate
\begin{align}
    \alpha-\gamma-2\lambda=\int_0^{\infty}\abs{v}^2\Upsilon\big(\abs{v}\big)\m\big(\abs{v}\big)\ud\abs{v}\int_0^{\pi}\sin^5\theta\ud\theta\int_{0}^{2\pi}\Big(\sin^4\varphi-3\sin^2\varphi\cos^2\varphi\Big)\ud\varphi=0.
\end{align}
Using \eqref{boundary decomposition}, we obtain \eqref{conservation law 5}.
\end{proof}

\paragraph{\underline{Conservation Law with Test Function $\nx\varphi\cdot\a+\e^{-1}\varphi\big(\abs{v}^2-5\big)\mh$}}

\begin{lemma}\label{lem:conservation 3}
    Under the assumption \eqref{assumption:stationary}, for any smooth function $\varphi(x)$ satisfying $\varphi\big|_{\p\Omega}=0$, we have 
    \begin{align}\label{conservation law 6}
    -\k\bbrx{\dx\varphi,c}
    =&\bbrb{\nx\varphi\cdot\a,h}{\gamma_-}-\bbrb{\nx\varphi\cdot\a,\re}{\gamma_+}+\br{v\cdot\nx\Big(\nx\varphi\cdot\a\Big),\ire}\\
    &+\e^{-1}\br{\varphi\left(\abs{v}^2-5\right)\mh,\ss}+\bbr{\nx\varphi\cdot\a,\ss}.\no
    \end{align}
\end{lemma}
\begin{proof}
From \eqref{conservation law 1} and \eqref{conservation law 3}, we have
\begin{align}\label{cc 06}
    \nx\cdot\varsigma=&\brv{\left(\abs{v}^2-5\right)\mh,\ss}.
\end{align}
Multiplying $\varphi(x)\in\r$ on both sides of \eqref{cc 06} and integrating over $x\in\Omega$, we obtain
\begin{align}\label{cc 01}
-\bbrx{\nx\varphi,\varsigma}+\int_{\p\Omega}\varphi\varsigma\cdot n=&\br{\varphi\left(\abs{v}^2-5\right)\mh,\ss}.
\end{align}
Hence, adding $\e^{-1}\times$\eqref{cc 01} and \eqref{conservation law 4} to eliminate $\e^{-1}\bbrx{\nx\varphi,\varsigma}$ yields
\begin{align}\label{cc 02}
    &-\k\bbrx{\dx\varphi,c}+\e^{-1}\int_{\p\Omega}\varphi\varsigma\cdot n\\
    =&\bbrb{\nx\varphi\cdot\a,h}{\gamma_-}-\bbrb{\nx\varphi\cdot\a,\re}{\gamma_+}+\br{v\cdot\nx\Big(\nx\varphi\cdot\a\Big),\ire}\no\\
    &+\e^{-1}\br{\varphi\left(\abs{v}^2-5\right)\mh,\ss}+\bbr{\nx\varphi\cdot\a,\ss}.\no
\end{align}
The assumption $\varphi\big|_{\p\Omega}=0$
completely eliminates the boundary term $\ds\e^{-1}\int_{\p\Omega}\varphi\varsigma\cdot n$ in \eqref{cc 02}. Hence, we have \eqref{conservation law 6}.
\end{proof}

\paragraph{\underline{Conservation Law with Test Function $\nx\psi:\b+\e^{-1}\psi\cdot v\mh$}}

\begin{lemma}\label{lem:conservation 4}
    Under the assumption \eqref{assumption:stationary}, for any smooth function $\psi(x)$ satisfying $\nx\cdot\psi=0$, $\psi\big|_{\p\Omega}=0$, we have 
    \begin{align}\label{conservation law 7}
    -\lambda\bbrx{\dx\psi,\bb}
    =&\bbrb{\nx\psi:\b,h}{\gamma_-}-\bbrb{\nx\psi:\b,\re}{\gamma_+}+\br{v\cdot\nx\Big(\nx\psi:\b\Big),\ire}\\
    &+\e^{-1}\br{\psi\cdot v\mh,\ss}+\bbr{\nx\psi:\b,\ss}.\no
    \end{align}
\end{lemma}
\begin{proof}
Multiplying $\psi(x)\in\r^3$ on both sides of \eqref{conservation law 2} and integrating over $x\in\Omega$, we obtain
\begin{align}\label{cc 03}
-\bbrx{\nx\cdot\psi, \P}-\bbrx{\nx\psi,\varpi}+\int_{\p\Omega}\Big(\P\psi+\psi\cdot\varpi\Big)\cdot n= &\br{\psi\cdot v\mh,\ss}.
\end{align}
Hence, adding $\e^{-1}\times$\eqref{cc 03} and \eqref{conservation law 5} to eliminate $\e^{-1}\bbrx{\nx\psi,\varpi}$ yields
\begin{align}\label{cc 04}
&-\lambda\bbr{\dx\psi,\bb}-\e^{-1}\bbrx{\nx\cdot\psi, \P}+\e^{-1}\int_{\p\Omega}\Big(\P\psi+\psi\cdot\varpi\Big)\cdot n\\
=&\bbrb{\nx\psi:\b,h}{\gamma_-}-\bbrb{\nx\psi:\b,\re}{\gamma_+}+\br{v\cdot\nx\Big(\nx\psi:\b\Big),\ire}\no\\
&+\e^{-1}\br{\psi\cdot v\mh,\ss}+\bbr{\nx\psi:\b,\ss}.\no
\end{align}
The assumptions $\nx\cdot\psi=0$ and $\psi\big|_{\p\Omega}=0$ 
eliminate $\e^{-1}\bbrx{\nx\cdot\psi, \P}$ and $\e^{-1}\ds\int_{\p\Omega}\Big(\P\psi+\psi\cdot\varpi\Big)\cdot n$ in \eqref{cc 04}. 
Hence, we have \eqref{conservation law 7}.
\end{proof}


\subsection{Energy Estimate}

\begin{proposition}\label{prop:energy}
    Under the assumption \eqref{assumption:stationary}, we have 
    \begin{align}\label{est:energy}
    \e^{-\frac{1}{2}}\tnms{\re}{\gamma_+}+\e^{-1}\tnm{\ire}\ls\oot\xnm{\re}+\xnm{\re}^2+\oot.
    \end{align}
\end{proposition}

\begin{proof}
It suffices to justify 
\begin{align}\label{eq:energy}
    \e^{-\frac{1}{2}}\tnms{\re}{\gamma_+}+\e^{-1}\tnm{\ire}\ls \oot\e^{-\frac{1}{2}}\tnm{\bre}+\oot\xnm{\re}+\xnm{\re}^2+\oot.
\end{align}

\paragraph{\underline{Weak Formulation}}
Taking test function $\test=\e^{-1}\re$ in \eqref{weak formulation}, we obtain
\begin{align}
    \frac{\e^{-1}}{2}\int_{\gamma}\re^2(v\cdot n)+\e^{-2}\bbr{\lc[\re],\re}=\e^{-1}\bbr{\ss,\re}.
\end{align}
Notice that
\begin{align}
    \int_{\gamma}\re^2(v\cdot n)=\tnms{\re}{\gamma_+}^2-\tnms{\re}{\gamma_-}^2=\tnms{\re}{\gamma_+}^2-\tnms{\h}{\gamma}^2,
\end{align}
and
\begin{align}
    \bbr{\lc[\re],\re}\gs\um{\ire}^2.
\end{align}
Then we know
\begin{align}
    \e^{-1}\tnms{\re}{\gamma_+}^2+\e^{-2}\um{\ire}^2\ls \abs{\e^{-1}\bbr{\ss,\re}}+\e^{-1}\tnms{\h}{\gamma}^2.
\end{align}
Using Lemma \ref{h-estimate}, we have \begin{align}\label{energy 01}
    \e^{-1}\tnms{\re}{\gamma_+}^2+\e^{-2}\tnm{\ire}^2\ls \abs{\e^{-1}\bbr{\ss,\re}}+\oot\e.
    \end{align}

\paragraph{\underline{Source Term Estimates}}
We split 
\begin{align}
    \e^{-1}\bbr{\ss,\re}=\e^{-1}\bbr{\ss,\bre}+\e^{-1}\bbr{\ss,\ire}.
\end{align}
We may directly bound using Lemma \ref{s1-estimate} -- Lemma \ref{s6-estimate}
\begin{align}
    \abs{\e^{-1}\bbr{\ss,\ire}}\ls& \e^{-1}\tnm{\ss}\tnm{\ire}\\
    \ls& \big(\oo+\oot\big)\e^{-2}\tnm{\ire}^2+\oot\xnm{\re}^2+\xnm{\re}^4+\oot.\no
\end{align}
Using orthogonality of $\Gamma$, we have
\begin{align}
    \e^{-1}\bbr{\ss,\bre}=\e^{-1}\bbr{\ssa+\ssc,\bre}.
\end{align}
From Lemma \ref{s1-estimate}, we know
\begin{align}
    \abs{\e^{-1}\bbr{\ssa,\bre}}\ls \e^{-1}\tnm{\ssa}\tnm{\bre}\ls\oot\tnm{\bre}^2+\oot.
\end{align}
Also, from Lemma \ref{s2-estimate} and Remark \ref{rmk:s2}, we have
\begin{align}
    \e^{-1}\bbr{\ssc,\bre}=&\e^{-1}\bbr{\ssx,\bre}+\e^{-1}\bbr{\ssy+\ssz,\bre}.
\end{align}
After integrating by parts with respect to $\va$ in $\ssx$ term, we obtain
\begin{align}
    \abs{\e^{-1}\bbr{\ssc,\bre}}\ls&\e^{-1}\Big(\nm{\fb_1}_{L^2_{x}L^1_v}+\nm{\ssy+\ssz}_{L^2_{x}L^1_v}\Big)\nm{\bre}_{L^2_{x}L^{\infty}_v}\\
    \ls&\oot\e^{-\frac{1}{2}}\tnm{\bre}\ls\oot\e^{-1}\tnm{\bre}^2+\oot.\no
\end{align}
In total, we have
\begin{align}\label{energy 02}
    \abs{\e^{-1}\bbr{\ss,\re}}\ls \oot\e^{-1}\tnm{\bre}^2+\big(\oo+\oot\big)\e^{-2}\tnm{\ire}^2+\oot\xnm{\re}^2+\xnm{\re}^4+\oot.
\end{align}

\paragraph{\underline{Synthesis}}
Inserting \eqref{energy 02} into \eqref{energy 01}, we have 
\begin{align}
    \e^{-1}\tnms{\re}{\gamma_+}^2+\e^{-2}\tnm{\ire}^2\ls \oot\e^{-1}\tnm{\bre}^2+\oot\xnm{\re}^2+\xnm{\re}^4+\oot.
    \end{align}
Then we have \eqref{eq:energy}.
\end{proof}

\begin{corollary}\label{lem:energy-extension} 
Under the assumption \eqref{assumption:stationary}, we have
    \begin{align}\label{energy-extension}
        \pnm{\ire}{6}+\pnms{\m^{\frac{1}{4}}\re}{4}{\gamma_+}
        \ls&\oot\xnm{\re}+\xnm{\re}^2+\oot.
    \end{align}
\end{corollary}
\begin{proof}
By interpolation and Proposition \ref{prop:energy}, we obtain
\begin{align}
    \pnm{\ire}{6}\ls&\tnm{\ire}^{\frac{1}{3}}\lnmm{\ire}^{\frac{2}{3}}\label{energy 03}\\
    \ls& \Big(\oot\xnm{\re}^{\frac{1}{3}}+\xnm{\re}^{\frac{1}{3}}+\oot\Big)\left(\e^{\frac{1}{2}}\lnmm{\re}\right)^{\frac{2}{3}}\ls\oot\xnm{\re}+\xnm{\re}^2+\oot,\no\\
    \pnms{\m^{\frac{1}{4}}\re}{4}{\gamma_+}\ls&\tnms{\re}{\gamma_+}^{\frac{1}{2}}\lnmms{\re}{\gamma_+}^{\frac{1}{2}}\label{energy 04}\\
    \ls& \Big(\oot\xnm{\re}^{\frac{1}{2}}+\xnm{\re}^{\frac{1}{2}}+\oot\Big)\Big(\e^{\frac{1}{2}}\lnmms{\re}{\gamma_+}\Big)^{\frac{1}{2}}
    \ls\oot\xnm{\re}+\xnm{\re}^2+\oot.\no
\end{align}
Then the desired result follows from \eqref{energy 03}\eqref{energy 04}.
\end{proof}


\subsection{Kernel Estimate}

\subsubsection{Estimate of $\P$}

\begin{proposition}\label{prop:p-bound}
Under the assumption \eqref{assumption:stationary}, we have 
\begin{align}\label{est:p-bound}
    \e^{-\frac{1}{2}}\tnm{\P}+\pnm{\P}{6}\ls\oot\xnm{\re}+\xnm{\re}^2+\oot.
\end{align}
\end{proposition}

\begin{proof}
It suffices to show for $2\leq\N\leq 6$
\begin{align}\label{eq:p-bound}
    \jnm{\P}\ls\jnms{\m^{\frac{1}{4}}\re}{\gamma_+}+\oot\e^{\frac{2}{\N}}.
\end{align}

\paragraph{\underline{Weak Formulation}}
Denote
\begin{align}\label{f:p-test}
\psi(x,v):=\mh(\vv)\Big(\vv\cdot\nx\varphi(x)\Big),
\end{align}
where $\varphi(x)$ is defined via solving the elliptic problem
\begin{align}
\left\{
\begin{array}{l}
-\dx\varphi=\ds \P\abs{\P}^{\N-2}\ \ \text{in}\ \ \Omega,\\\rule{0ex}{1.5em}
\varphi=0\ \ \text{on}\ \ \p\Omega.
\end{array}
\right.
\end{align}
Based on standard elliptic estimates \cite{Krylov2008}, there exists a solution $\varphi$ satisfying
\begin{align}
\nm{\psi}_{W^{1,\frac{\N}{\N-1}}L^{\infty}_{\vrh,\vth}}\ls\nm{\varphi}_{W^{2,\frac{\N}{\N-1}}}\ls\pnm{\P\abs{\P}^{\N-2}}{\frac{\N}{\N-1}}\ls\jnm{\P}^{\N-1}.
\end{align}
Based on Sobolev embedding and trace estimate, we have for $2\leq\N\leq6$
\begin{align}
    \tnm{\psi}+\knms{\psi}{\gamma}\ls\jnm{\P}^{\N-1}.
\end{align}
Taking test function $\test=\psi$ in \eqref{weak formulation}, we obtain
\begin{align}
    \int_{\gamma}\re\psi\ud\gamma-\bbr{\re,\vv\cdot\nx\psi}
    &=\bbr{\ss,\psi}.
\end{align}
From Lemma \ref{h-estimate}, we know
\begin{align}
    \abs{\int_{\gamma}\re\psi\ud\gamma}\leq&\abs{\int_{\gamma_+}\re\psi\ud\gamma}+\abs{\int_{\gamma_-}\h\psi\ud\gamma}
    \ls\jnms{\m^{\frac{1}{4}}\re}{\gamma_+}\knms{\psi}{\gamma_+}+\jnms{\h}{\gamma_-}\knms{\psi}{\gamma_-}\\
    \ls&\oo\knms{\psi}{\gamma}^{\frac{\N}{\N-1}}+ \jnms{\m^{\frac{1}{4}}\re}{\gamma_+}^{\N}+\jnms{\h}{\gamma_-}^{\N}
   \ls\oo\jnm{\P}^{\N}+\jnms{\m^{\frac{1}{4}}\re}{\gamma_+}^{\N}+\oot\e^3.\no
\end{align}
Due to oddness and orthogonality, we have
\begin{align}
\br{\mh\big(\vv\cdot\bb\big),\vv\cdot\nx\psi}=\bbr{\ire,\vv\cdot\nx\psi}=0.
\end{align}
Due to orthogonality of $\ab$, we know
\begin{align}
&\br{\mh\frac{\abs{\vv}^2-5}{2}\cc,\vv\cdot\nx\psi}
=\bbr{c,\mh\ab\cdot\nx\psi}=0.
\end{align}
Also, we have
\begin{align}
-\br{\mh\P ,\vv\cdot\nx\psi}
=&-\br{\P\m,\vv\cdot\nx\Big(\vv\cdot\nx\varphi\Big)}
=-\frac{1}{3}\int_{\Omega}p\big(\dx\varphi \big)\int_{\r^3}\m\abs{\vv}^2
=\jnm{\P}^{\N}.
\end{align}
Collecting the above, we have
\begin{align}\label{kernel-p01}
    \jnm{\P}^{\N}\ls&\jnms{\m^{\frac{1}{4}}\re}{\gamma_+}^{\N}+\oot\e^3+\abs{\bbr{\ss,\psi}}.
\end{align}

\paragraph{\underline{Source Term Estimates}}
Due to the orthogonality of $\Gamma$ and Lemma \ref{s1-estimate}, we know
\begin{align}
    \bbr{\ss,\psi}=\bbr{\ssc,\psi}.
\end{align}
Using Hardy's inequality and integrating by parts with respect to $\va$ in $\ssx$, based on Lemma \ref{s2-estimate} and Remark \ref{rmk:s2}, we have
\begin{align}\label{kernel-p02}
    &\abs{\bbr{\ssc,\psi}}\leq\abs{\br{\ssc,\psi\Big|_{\mn=0}}}+\abs{\br{\ssc,\int_0^{\mn}\p_{\mn}\psi}}
    =\abs{\br{\ssc,\psi\Big|_{\mn=0}}}+\abs{\e\br{\eta\ssc,\frac{1}{\mn}\int_0^{\mn}\p_{\mn}\psi}}\\
    \ls&\nm{\fb_1+\ssy+\ssz}_{L^2_{\iota_1\iota_2}L^1_{\mn}L^1_{\vv}}\tnms{\psi}{\gamma}+\e\pnm{\eta\big(\fb_1+\ssy+\ssz\big)}{\N}\pnm{\frac{1}{\mn}\int_0^{\mn}\p_{\mn}\psi}{\frac{\N}{\N-1}}\no\\
    \ls&\nm{\fb_1+\ssy+\ssz}_{L^2_{\iota_1\iota_2}L^1_{\mn}L^1_{\vv}}\tnms{\psi}{\gamma}+\e\pnm{\eta\big(\fb_1+\ssy+\ssz\big)}{\N}\pnm{\p_{\mn}\psi}{\frac{\N}{\N-1}}\no\\
    \ls&\oot\e\tnms{\psi}{\gamma}+\oot\e^{\frac{2}{\N}}\pnm{\p_{\mn}\psi}{\frac{\N}{\N-1}}\ls \oot\e^{\frac{2}{\N}}\jnm{\P}^{\N-1}\ls\oot\jnm{\P}^{\N}+\oot\e^2.\no
\end{align}
Inserting \eqref{kernel-p02} into \eqref{kernel-p01}, we have shown
\begin{align}
    \jnm{\P}^{\N}\ls&\jnms{\m^{\frac{1}{4}}\re}{\gamma_+}^{\N}+\oot\e^2.
\end{align}
Hence, we have \eqref{eq:p-bound}.
\end{proof}


\subsubsection{Estimate of $\cc$}

\begin{proposition}\label{prop:c-bound}
Under the assumption \eqref{assumption:stationary}, we have  
\begin{align}\label{est:c-bound}
    \e^{-\frac{1}{2}}\tnm{\cc}+\pnm{\cc}{6}\ls&\oot\xnm{\re}+\xnm{\re}^{2}+\oot.
\end{align}
\end{proposition}

\begin{proof}
It suffices to justify for $2\leq\N\leq6$
\begin{align}\label{eq:c-bound}
    \jnm{\cc}\ls\e^{\frac{1}{4\N}}\xnm{\re}^{\frac{1}{\N}}\jnm{\cc}+\jnms{\m^{\frac{1}{4}}\re}{\gamma_+}+\jnm{\ire}+\oot\e^{\frac{1}{2}}\xnm{\re}+\e^{\frac{1}{2}}\xnm{\re}^{2}+\oot\big(\e^{\frac{1}{2}}+\e^{\frac{2}{\N}}\big).
\end{align}

\paragraph{\underline{Weak Formulation}}
We consider the conservation law \eqref{conservation law 6}
where the smooth test function $\varphi(x)$ satisfies
\begin{align}\label{f:c-test}
\left\{
\begin{array}{l}
-\dx\varphi=\cc\abs{\cc}^{\N-2}\ \ \text{in}\ \ \Omega,\\\rule{0ex}{1.5em}
\varphi=0\ \ \text{on}\ \ \p\Omega.
\end{array}
\right.
\end{align}
Based on standard elliptic estimates \cite{Krylov2008}, there exists a solution $\varphi$ satisfying 
\begin{align}
    \nm{\varphi}_{W^{2,\frac{\N}{\N-1}}}\ls\pnm{\cc\abs{\cc}^{\N-2}}{\frac{\N}{\N-1}}\ls\jnm{\cc}^{\N-1}.
\end{align}
Based on Sobolev embedding and trace estimate, we have for $2\leq\N\leq6$
\begin{align}
    \nm{\varphi}_{H^1}+\knms{\nx\varphi}{\p\Omega}\ls\jnm{\cc}^{\N-1}.
\end{align}
Hence, from \eqref{conservation law 6}, we have
\begin{align}
    \k\jnm{\cc}^{\N}
    =&\bbrb{\nx\varphi\cdot\a,h}{\gamma_-}-\bbrb{\nx\varphi\cdot\a,\re}{\gamma_+}+\br{v\cdot\nx\Big(\nx\varphi\cdot\a\Big),\ire}\\
    &+\e^{-1}\br{\varphi\left(\abs{v}^2-5\right)\mh,\ss}+\bbr{\nx\varphi\cdot\a,\ss}.\no
\end{align}
From Lemma \ref{h-estimate}, we have
\begin{align}
    \abs{\bbrb{\nx\varphi\cdot\a,\h}{\gamma_-}}\ls&\knms{\nx\varphi\cdot\a}{\gamma_-}\jnms{\h}{\gamma_-}\ls\oot\jnm{\cc}^{\N}+\oot\e^3,\\
    \abs{\bbrb{\nx\varphi\cdot\a,\re}{\gamma_+}}\ls&\knms{\nx\varphi\cdot\a}{\gamma_+}\jnms{\re}{\gamma_+}\ls\oo\jnm{\cc}^{\N}+\jnms{\m^{\frac{1}{4}}\re}{\gamma_+}^{\N},
\end{align}
and
\begin{align}
    \abs{\br{v\cdot\nx\Big(\nx\varphi\cdot\a\Big),\ire}}\ls&\knm{v\cdot\nx\Big(\nx\varphi\cdot\a\Big)}\jnm{\ire}\\
    \ls&\oo\jnm{\cc}^{\N}+\jnm{\ire}^{\N}.\no
\end{align}
Collecting the above, we have shown that
\begin{align}\label{kernel-c01}
    \jnm{\cc}^{\N}\ls&\jnms{\m^{\frac{1}{4}}\re}{\gamma_+}^{\N}+\jnm{\ire}^{\N}+\oot\e^3+\abs{\e^{-1}\br{\varphi\big(\abs{v}^2-5\big)\mh,\ss}}+\abs{\bbr{\nx\varphi\cdot\a,\ss}}.
\end{align}

\paragraph{\underline{Source Term Estimates}}
Due to the orthogonality of $\Gamma$ and Lemma \ref{s1-estimate}, we have
\begin{align}
    \e^{-1}\br{\varphi\big(\abs{v}^2-5\big)\mh,\ss}=\e^{-1}\br{\varphi\big(\abs{v}^2-5\big)\mh,\ssc}.
\end{align}
Similar to \eqref{kernel-p02}, based on Lemma \ref{s2-estimate}, Remark \ref{rmk:s2} and Hardy's inequality, we have
\begin{align}
    &\abs{\e^{-1}\br{\varphi\big(\abs{v}^2-5\big)\mh,\ssc}}\ls\e^{-1}\abs{\br{\ssc,\int_0^{\mn}\p_{\mn}\varphi}}\\
    \ls&\abs{\br{\eta\ssc,\frac{1}{\mn}\int_0^{\mn}\p_{\mn}\varphi}}
    \ls\nm{\eta\big(\fb_1+\ssy+\ssz\big)}_{L^2_xL^1_{\vv}}\tnm{\frac{1}{\mn}\int_0^{\mn}\p_{\mn}\varphi}\no\\
    \ls&\nm{\eta\big(\fb_1+\ssy+\ssz\big)}_{L^2_xL^1_{\vv}}\tnm{\p_{\mn}\varphi}\ls\oot\jnm{\cc}^{\N}+\oot\e^{\frac{\N}{2}}.\no
\end{align}
From Lemma \ref{s1-estimate}, we directly bound
\begin{align}
    \abs{\bbr{\nx\varphi\cdot\a,\ssa}}\ls\tnm{\nx\varphi}\tnm{\ssa}\ls\oot\jnm{\cc}^{\N}+\oot\e^{\N}.
\end{align}
Similar to \eqref{kernel-p02}, based on Lemma \ref{s2-estimate}, Remark \ref{rmk:s2} and Hardy's inequality, we have
\begin{align}
    &\abs{\bbr{\nx\varphi\cdot\a,\ssc}}\leq\abs{\br{\ssc,\nx\varphi\Big|_{\mn=0}}}+\abs{\e\br{\eta\ssc,\frac{1}{\mn}\int_0^{\mn}\p_{\mn}\nx\varphi}}\\
    \ls&\nm{\fb_1+\ssy+\ssz}_{L^2_{\iota_1\iota_2}L^1_{\mn}L^1_{\vv}}\tnms{\nx\varphi}{\gamma}+\e\jnm{\eta\big(\fb_1+\ssy+\ssz\big)}\knm{\frac{1}{\mn}\int_0^{\mn}\p_{\mn}\nx\varphi}\no\\
    \ls&\nm{\fb_1+\ssy+\ssz}_{L^2_{\iota_1\iota_2}L^1_{\mn}L^1_{\vv}}\tnms{\nx\varphi}{\gamma}+\e\jnm{\eta\big(\fb_1+\ssy+\ssz\big)}\knm{\p_{\mn}\nx\varphi}\ls \oot\e^{\frac{2}{\N}}\jnm{\cc}^{\N-1}\ls\oot\jnm{\cc}^{\N}+\oot\e^{2}.\no
\end{align}
Based on Lemma \ref{s3-estimate}, Lemma \ref{s4-estimate}, and Lemma \ref{s5-estimate}, we have
\begin{align}
    \abs{\bbr{\nx\varphi\cdot\a,\ssd+\ssf+\ssg}}\ls&\tnm{\nx\varphi}\tnm{\ssd+\ssf+\ssg}\\
    \ls&\oot\jnm{\cc}^{\N}+\oot\tnm{\re}^{\N}+\oot\e^{\frac{\N}{2}}
    \ls\oot\jnm{\cc}^{\N}+\oot\e^{\frac{\N}{2}}\xnm{\re}^{\N}+\oot\e^{\frac{\N}{2}}.\no
\end{align}
Finally, based on Lemma \ref{s6-estimate}, we have
\begin{align}
    \abs{\bbr{\nx\varphi\cdot\a,\ssh}}\ls\abs{\bbr{\nx\varphi\cdot\a,\Gamma\big[\bre,\bre\big]}}+\abs{\bbr{\nx\varphi\cdot\a,\Gamma\big[\re,\ire\big]}}.
\end{align}
The oddness and orthogonality, combined with the interpolation $\pnm{\bb}{3}\ls\tnm{\bb}^{\frac{1}{2}}\pnm{\bb}{6}^{\frac{1}{2}}\ls\e^{\frac{1}{4}}\xnm{\re}$, imply that
\begin{align}
    \abs{\br{\nx\varphi\cdot\a,\Gamma\big[\bre,\bre\big]}}&\ls\abs{\br{\nx\varphi\cdot\a,\Gamma\left[\mh\left(\vv\cdot\bb\right),\mh\bigg(\frac{\abs{\vv}^2-5}{2}\cc\bigg)\right]}}\\
    &\ls\pnm{\nx\varphi}{\frac{3\N}{2\N-3}}\pnm{\bb}{3}\jnm{\cc}\ls \nm{\varphi}_{W^{2,\frac{\N}{\N-1}}}\pnm{\bb}{3}\jnm{\cc}\ls\e^{\frac{1}{4}}\xnm{\re}\jnm{\cc}^{\N}.\no
\end{align}
In addition, using the interpolation $\pnm{\ire}{3}\ls\tnm{\ire}^{\frac{1}{2}}\pnm{\ire}{6}^{\frac{1}{2}}\ls\e^{\frac{1}{2}}\xnm{\re}$, we have
\begin{align}
    \abs{\bbr{\nx\varphi\cdot\a,\Gamma\Big[\re,\ire\Big]}}&\ls\tnm{\nx\varphi}\pnm{\re}{6}\pnm{\ire}{3}\\
    &\ls\e^{\frac{1}{2}}\jnm{\cc}^{\N-1}\xnm{\re}^2\ls \oo\jnm{\cc}^{\N}+\e^{\frac{\N}{2}}\xnm{\re}^{2\N}.\no
\end{align}
Hence, we know
\begin{align}
    \abs{\bbr{\nx\varphi\cdot\a,\ssh}}\ls\e^{\frac{1}{4}}\xnm{\re}\jnm{\cc}^{\N}+\oo\jnm{\cc}^{\N}+\e^{\frac{\N}{2}}\xnm{\re}^{2\N}.
\end{align}
Collecting the above, we have
\begin{align}\label{kernel-c02}
    &\;\abs{\e^{-1}\br{\varphi\big(\abs{v}^2-5\big)\mh,\ss}}+\abs{\bbr{\nx\varphi\cdot\a,\ss}}\\
    \ls&\;\e^{\frac{1}{4}}\xnm{\re}\jnm{\cc}^{\N}+\big(\oo+\oot\big)\jnm{\cc}^{\N}+\oot\e^{\frac{\N}{2}}\xnm{\re}^{\N}+\e^{\frac{\N}{2}}\xnm{\re}^{2\N}+\oot\left(\e^{\frac{\N}{2}}+\e^2\right).\no
\end{align}
Inserting \eqref{kernel-c02} into \eqref{kernel-c01}, we have
\begin{align}
    \jnm{\cc}^{\N}\ls\e^{\frac{1}{4}}\xnm{\re}\jnm{\cc}^{\N}+\jnms{\m^{\frac{1}{4}}\re}{\gamma_+}^{\N}+\jnm{\ire}^{\N}+\oot\e^{\frac{\N}{2}}\xnm{\re}^{\N}+\e^{\frac{\N}{2}}\xnm{\re}^{2\N}+\oot\left(\e^{\frac{\N}{2}}+\e^2\right).
\end{align}
Hence, \eqref{eq:c-bound} follows.
\end{proof}


\subsubsection{Estimate of $\bb$}

\begin{proposition}\label{prop:b-bound}
Under the assumption \eqref{assumption:stationary}, we have  
\begin{align}\label{est:b-bound}
    \e^{-\frac{1}{2}}\tnm{\bb}+\pnm{\bb}{6}\ls\oot\xnm{\re}+\xnm{\re}^{2}+\oot.
\end{align}
\end{proposition}
\begin{proof}
It suffices to justify for $2\leq\N\leq6$
\begin{align}\label{eq:b-bound}
    \jnm{\bb}
    \ls\e^{\frac{1}{4\N}}\xnm{\re}^{\frac{1}{\N}}\jnm{\bb}+\jnms{\m^{\frac{1}{4}}\re}{\gamma_+}+\jnm{\ire}+\oot\e^{\frac{1}{2}}\xnm{\re}+\e^{\frac{1}{2}}\xnm{\re}^{2}+\oot\big(\e^{\frac{1}{2}}+\e^{\frac{2}{\N}}\big).
\end{align}

\paragraph{\underline{Weak Formulation}}
Assume $(\psi,q)\in\r^3\times\r$ (where $q$ has zero average) is the unique strong solution to the Stokes problem
\begin{align}\label{f:b-test}
\left\{
    \begin{array}{rcll}
    -\lambda\Delta_x\psi+\nx q&=&\bb\abs{\bb}^{\N-2}&\ \ \text{in}\ \ \Omega,\\\rule{0ex}{1.5em}
    \nx\cdot\psi&=&0&\ \ \text{in}\ \ \Omega,\\\rule{0ex}{1.5em}
    \psi&=&0&\ \ \text{on}\ \ \p\Omega.
    \end{array}
\right.
\end{align}
We have the standard estimate \cite{Cattabriga1961}
\begin{align}
    \nm{\psi}_{W^{2,\frac{\N}{\N-1}}}+\nm{q}_{W^{1,\frac{\N}{\N-1}}}\ls \pnm{\bb\abs{\bb}^{\N-2}}{\frac{\N}{\N-1}}\ls\jnm{\bb}^{\N-1}.
\end{align}
Based on Sobolev embedding and trace estimate, we have for $2\leq\N\leq6$
\begin{align}
    \nm{\psi}_{H^1}+\knms{\nx\psi}{\p\Omega}+\tnm{q}+\knms{q}{\p\Omega}\ls\jnm{\bb}^{\N-1}.
\end{align}
Multiplying $\bb$ on both sides of \eqref{f:b-test} and integrating by parts for $\bbrx{\nx q,\bb}$, we have
\begin{align}
    -\bbrx{\lambda\Delta_x\psi,\bb}-\bbrx{ q,\nx\cdot\bb}+\int_{\p\Omega}q(\bb\cdot n)&=\jnm{\bb}^{\N},
\end{align}
which, by combining \eqref{conservation law 1} and Remark \ref{rmk:boundary}, implies
\begin{align}\label{kernel-b00}
    -\bbrx{\lambda\Delta_x\psi,\bb}-\br{q\mh,\ss}+\bbrb{q\mh,\re}{\gamma_+}-\bbrb{q\mh,\h}{\gamma_-}&=\jnm{\bb}^{\N}.
\end{align}
Inserting \eqref{kernel-b00} into \eqref{conservation law 7} to replace $-\bbrx{\lambda\Delta_x\psi,\bb}$, we obtain
\begin{align}
    \jnm{\bb}^{\N}
    =&-\bbrb{q\mh,\h}{\gamma_-}+\bbrb{q\mh,\re}{\gamma_+}+\bbrb{\nx\psi:\b,h}{\gamma_-}-\bbrb{\nx\psi:\b,\re}{\gamma_+}\\
    &+\br{v\cdot\nx\Big(\nx\psi:\b\Big),\ire}-\br{q\mh,\ss}+\e^{-1}\br{\psi\cdot v\mh,\ss}+\bbr{\nx\psi:\b,\ss}.\no
\end{align}
From Lemma \ref{h-estimate}, we have 
\begin{align}
    \abs{\bbrb{q\mh,\h}{\gamma_-}}+\abs{\bbrb{\nx\psi:\b,\h}{\gamma_-}}
    &\ls\Big(\knms{q}{\p\Omega}+\knms{\nx\psi}{\p\Omega}\Big)\jnms{\h}{\gamma_-}\\
    &\ls \oot\jnm{\bb}^{\N}+\oot\e^3,\no\\
    \abs{\bbrb{q\mh,\re}{\gamma_+}}+\abs{\bbrb{\nx\psi:\b,\re}{\gamma_+}}
    &\ls\Big(\knms{q}{\p\Omega}+\knms{\nx\psi}{\p\Omega}\Big)\jnms{\re}{\gamma_+}\\
    &\ls \oo\jnm{\bb}^{\N}+\jnms{\m^{\frac{1}{4}}\re}{\gamma_+}^{\N},\no
\end{align}
and
\begin{align}
    \abs{\br{v\cdot\nx\Big(\nx\psi:\b\Big),\ire}}
    &\ls \knm{v\cdot\nx\Big(\nx\psi:\b\Big)}\jnm{\ire}\\
    &\ls\oo\jnm{\bb}^{\N}+\jnm{\ire}^{\N}.\no
\end{align}
Collecting the above, we have
\begin{align}\label{kernel-b01}
    \jnm{\bb}^{\N}
    \ls&\jnms{\m^{\frac{1}{4}}\re}{\gamma_+}^{\N}+\jnm{\ire}^{\N}+\oot\e^3\\
    &+\abs{\br{q\mh,\ss}}+\abs{\e^{-1}\br{\psi\cdot v\mh,\ss}}+\abs{\bbr{\nx\psi:\b,\ss}}.\no
\end{align}

\paragraph{\underline{Source Term Estimates}}
Due to orthogonality of $\Gamma$ and Lemma \ref{s1-estimate}, we have
\begin{align}
    \abs{\br{q\mh,\ss}}+\abs{\e^{-1}\br{\psi\cdot v\mh,\ss}}=\abs{\br{q,\mh\ssc}}+\abs{\e^{-1}\br{\psi\cdot v\mh,\ssc}}.
\end{align}
Using Lemma \ref{s2-estimate} and Remark \ref{rmk:s2}, integrating by parts in $\va$ for $\ssx$, we obtain
\begin{align}
    \abs{\br{q\mh,\ssc}}\ls\tnm{q}\nm{\fb_1+\ssy+\ssz}_{L^2_xL^1_v}\ls \oot\e^{\frac{1}{2}}\tnm{q}\ls\oot\jnm{\bb}^{\N}+\oot\e^{\frac{\N}{2}}.
\end{align}
Similar to \eqref{kernel-p02}, we have
\begin{align}
    &\abs{\e^{-1}\br{\psi\cdot v\mh,\ssc}}\ls\e^{-1}\abs{\br{\ssc,\int_0^{\mn}\p_{\mn}\psi}}\\
    \ls&\abs{\br{\eta\ssc,\frac{1}{\mn}\int_0^{\mn}\p_{\mn}\psi}}
    \ls\nm{\eta\big(\fb_1+\ssy+\ssz\big)}_{L^2_xL^1_{\vv}}\tnm{\frac{1}{\mn}\int_0^{\mn}\p_{\mn}\psi}\no\\
    \ls&\nm{\eta\big(\fb_1+\ssy+\ssz\big)}_{L^2_xL^1_{\vv}}\tnm{\p_{\mn}\psi}\ls\oot\jnm{\bb}^{\N}+\oot\e^{\frac{\N}{2}}.\no
\end{align}
From Lemma \ref{s1-estimate}, we directly bound
\begin{align}
    \abs{\bbr{\nx\psi:\b,\ssa}}\ls\tnm{\nx\psi}\tnm{\ssa}\ls\oot\jnm{\bb}^{\N}+\oot\e^{\N}.
\end{align}
Similar to \eqref{kernel-p02}, based on Lemma \ref{s2-estimate}, Remark \ref{rmk:s2} and Hardy's inequality, we have
\begin{align}
    &\abs{\bbr{\nx\psi:\b,\ssc}}\leq\abs{\br{\ssc,\nx\psi\Big|_{\mn=0}}}+\abs{\e\br{\eta\ssc,\frac{1}{\mn}\int_0^{\mn}\p_{\mn}\nx\psi}}\\
    \ls&\nm{\fb_1+\ssy+\ssz}_{L^2_{\iota_1\iota_2}L^1_{\mn}L^1_{\vv}}\tnms{\nx\psi}{\gamma}+\e\jnm{\eta\big(\fb_1+\ssy+\ssz\big)}\knm{\frac{1}{\mn}\int_0^{\mn}\p_{\mn}\nx\psi}\no\\
    \ls&\nm{\fb_1+\ssy+\ssz}_{L^2_{\iota_1\iota_2}L^1_{\mn}L^1_{\vv}}\tnms{\nx\psi}{\gamma}+\e\jnm{\eta\big(\fb_1+\ssy+\ssz\big)}\knm{\p_{\mn}\nx\psi}\ls \oot\e^{\frac{2}{\N}}\jnm{\bb}^{\N-1}\ls\oot\jnm{\bb}^{\N}+\oot\e^{2}.\no
\end{align}
Based on Lemma \ref{s3-estimate}, Lemma \ref{s4-estimate}, and Lemma \ref{s5-estimate}, we have
\begin{align}
    \abs{\bbr{\nx\psi:\b,\ssd+\ssf+\ssg}}
    &\ls\tnm{\nx\psi}\tnm{\ssd+\ssf+\ssg}\\
    &\ls\oot\jnm{\bb}^{\N}+\oot\tnm{\re}^{\N}+\oot\e^{\frac{\N}{2}}
    \ls\oot\jnm{\bb}^{\N}+\oot\e^{\frac{\N}{2}}\xnm{\re}^{\N}+\oot\e^{\frac{\N}{2}}.\no
\end{align}
Finally, based on Lemma \ref{s6-estimate}, we have
\begin{align}
    \abs{\bbr{\nx\psi:\b,\ssh}}\ls&\abs{\bbr{\nx\psi:\b,\Gamma\Big[\bre,\bre\Big]}}+\abs{\bbr{\nx\psi:\b,\Gamma\Big[\re,\ire\Big]}}.
\end{align}
The oddness and orthogonality imply that
\begin{align}
    &\abs{\br{\nx\psi:\b,\Gamma\Big[\bre,\bre\Big]}}\\
    \ls&\abs{\br{\nx\psi:\b,\Gamma\left[\mh\left(\vv\cdot\bb\right),\mh\left(\vv\cdot\bb\right)\right]}}+\abs{\br{\nx\psi:\b,\Gamma\left[\mh\left(\frac{\abs{\vv}^2-5}{2}\cc\right),\mh\left(\frac{\abs{\vv}^2-5}{2}\cc\right)\right]}}.\no
\end{align}
We may directly bound
\begin{align}
    \abs{\br{\nx\psi:\b,\Gamma\left[\mh\left(\vv\cdot\bb\right),\mh\left(\vv\cdot\bb\right)\right]}}\ls&\pnm{\nx\psi}{\frac{3\N}{2\N-3}}\pnm{\bb}{3}\jnm{\bb}\\
    \ls& \nm{\psi}_{W^{2,\frac{\N}{\N-1}}}\pnm{\bb}{3}\jnm{\bb}\ls\e^{\frac{1}{4}}\xnm{\re}\jnm{\bb}^{\N}.\no
\end{align}
Due to oddness and $\b_{ii}=\li\left[\left(\abs{v_i}^2-\dfrac{1}{3}\abs{\vv}^2\right)\mh\right]$, noting that $\Gamma\left[\mh\left(\dfrac{\abs{\vv}^2-5}{2}\cc\right),\mh\left(\dfrac{\abs{\vv}^2-5}{2}\cc\right)\right]$ only depends on $\abs{\vv}^2$, we have
\begin{align}
    &\abs{\br{\nx\psi:\b,\Gamma\left[\mh\left(\frac{\abs{\vv}^2-5}{2}\cc\right),\mh\left(\frac{\abs{\vv}^2-5}{2}\cc\right)\right]}}\\
    =&\abs{\br{\p_1\psi_1\b_{11}+\p_2\psi_2\b_{22}+\p_3\psi_3\b_{33},\Gamma\left[\mh\left(\frac{\abs{\vv}^2-5}{2}\cc\right),\mh\left(\frac{\abs{\vv}^2-5}{2}\cc\right)\right]}}\no\\
    =&\abs{\br{\big(\nx\cdot\psi\big)\b_{ii},\Gamma\left[\mh\left(\frac{\abs{\vv}^2-5}{2}\cc\right),\mh\left(\frac{\abs{\vv}^2-5}{2}\cc\right)\right]}}=0.\no
\end{align}
In addition, we have
\begin{align}
    \abs{\bbr{\nx\psi:\b,\Gamma\Big[\re,\ire\Big]}}\ls&\tnm{\nx\psi}\pnm{\re}{6}\pnm{\ire}{3}\\
    \ls&\e^{\frac{1}{2}}\jnm{\bb}^{\N-1}\xnm{\re}^2\ls \oo\jnm{\bb}^{\N}+\e^{\frac{\N}{2}}\xnm{\re}^{2\N}.\no
\end{align}
Hence, we know
\begin{align}
    \abs{\bbr{\nx\psi:\b,\ssh}}\ls\e^{\frac{1}{4}}\xnm{\re}\jnm{\bb}^{\N}+\oo\jnm{\bb}^{\N}+\e^{\frac{\N}{2}}\xnm{\re}^{2\N}.
\end{align}
Collecting the above, we have
\begin{align}\label{kernel-b02}
    &\abs{\br{q\mh,\ss}}+\abs{\e^{-1}\br{\psi\cdot v\mh,\ss}}+\abs{\bbr{\nx\psi:\b,\ss}}\\
    \ls&\e^{\frac{1}{4}}\xnm{\re}\jnm{\bb}^{\N}+\big(\oo+\oot\big)\jnm{\bb}^{\N}+\oot\e^{\frac{\N}{2}}\xnm{\re}^{\N}+\e^{\frac{\N}{2}}\xnm{\re}^{2\N}+\oot\left(\e^{\frac{\N}{2}}+\e^2\right).\no
\end{align}
Inserting \eqref{kernel-b02} into \eqref{kernel-b01}, we have
\begin{align}
    \jnm{\bb}^{\N}
    \ls&\e^{\frac{1}{4}}\xnm{\re}\jnm{\bb}^{\N}+\jnms{\m^{\frac{1}{4}}\re}{\gamma_+}^{\N}+\jnm{\ire}^{\N}+\oot\e^{\frac{\N}{2}}\xnm{\re}^{\N}+\e^{\frac{\N}{2}}\xnm{\re}^{2\N}+\oot\left(\e^{\frac{\N}{2}}+\e^2\right).
\end{align}
Hence, \eqref{eq:b-bound} follows.
\end{proof}


\subsubsection{Synthesis of Kernel Estimates}

\begin{proposition}\label{prop:kernel}
   Under the assumption \eqref{assumption:stationary}, we have  
    \begin{align}
        \e^{-\frac{1}{2}}\tnm{\bre}+\pnm{\bre}{6}\ls\oot\xnm{\re}+\xnm{\re}^{2}+\oot.
    \end{align}
\end{proposition}
\begin{proof}
    Collecting Proposition \ref{prop:p-bound}, Proposition \ref{prop:c-bound} and Proposition \ref{prop:c-bound}, we obtain the desired result.
\end{proof}


\subsection{\texorpdfstring{$L^{\infty}$}{} Estimate}

We define weight functions scaled with parameters $0\leq\vrh<\frac{1}{2}$ and $\vth\geq0$,
\begin{align}\label{ltt 11}
\vh(\vv):=\bv,
\end{align}
\begin{align}\label{ltt 12}
\tvh(\vv):=\frac{1}{\m^{\frac{1}{2}}(\vv)\vh(\vv)}
=(2\pi)^{\frac{3}{4}}\br{v}^{-\vth}\ue^{\left(\frac{1}{4}-\frac{\varrho}{2}\right)\abs{\vv}^2}
\end{align}

\begin{lemma}\label{lem:kernel-operator}
We have
\begin{align}
    \abs{k(\vuu,\vv)}\ls\left(\abs{\vuu-\vv}+\abs{\vuu-\vv}^{-1}\right)\ue^{-\frac{1}{4}\abs{\vuu-\vv}^2-\frac{1}{4}\frac{\abs{\abs{\vuu}^2-\abs{\vv}^2}^2}{\abs{\vuu-\vv}^2}}.
\end{align}
Let $0\leq\varrho< \frac{1}{2}$ and $\vth\geq0$. Then for $\d>0$ sufficiently small and any $\vv\in\r^3$,
\begin{align}
\int_{\R^3}\ue^{\d\abs{\vuu-\vv}^2}\abs{k(\vuu,\vv)}
\frac{\br{\vv}^{\vth}\ue^{\vrh\frac{\abs{\vv}^2}{2}}}{\br{\vuu}^{\vth}\ue^{\vrh\frac{\abs{\vuu}^2}{2}}}\ud{\vuu}
\ls \nu^{-1}.
\end{align}
\end{lemma}
\begin{proof}
This is a rescaled version of  \cite[Lemma 3]{Guo2010} and \cite[Lemma 2.3]{Guo.Jang.Jiang2010}.
\end{proof}

\begin{proposition}\label{prop:infty}
Under the assumption \eqref{assumption:stationary}, we have
\begin{align}
\e^{\frac{1}{2}}\lnmm{\re}+\e^{\frac{1}{2}}\lnmms{\re}{\gamma_+}
\ls \oot\xnm{\re}+\xnm{\re}^{2}+\oot.
\end{align}
\end{proposition}
\begin{proof} 
We will use the well-known $L^2-L^6-L^{\infty}$ framework.\\

\paragraph{\underline{Step 1: Mild Formulation}}
Denote the weighted solution
\begin{align}
\reg(\vx,\vv):=\vh(\vv)\re(\vx,\vv),
\end{align}
and the weighted non-local operator
\begin{align}
K_{\vh(\vv)}[\reg](\vv):=&\vh(\vv)K\left[\frac{\reg}{\vh}\right](\vv)=\int_{\r^3}k_{\vh(\vv)}(\vv,\vuu)\reg(\vuu)\ud{\vuu},
\end{align}
where
\begin{align}
k_{\vh(\vv)}(\vv,\vuu):=k(\vv,\vuu)\frac{\vh(\vv)}{\vh(\vuu)}.
\end{align}
Multiplying $\e\vh$ on both sides of \eqref{remainder}, we have
\begin{align}\label{ltt 01}
\left\{
\begin{array}{l}
\e\vv\cdot\nx\reg+\nu\reg=K_{\vh}[\reg](\vx,\vv)+\e\vh(\vv) \ss(\vx,\vv)\ \ \text{in}\ \ \Omega\times\r^3,\\\rule{0ex}{2em}
\reg(\vx_0,\vv)=\vh\h(\vx_0,\vv)\ \ \text{for}\ \ \vx_0\in\p\Omega\ \
\text{and}\ \ \vv\cdot\vn<0,
\end{array}
\right.
\end{align}
We can rewrite the solution of the equation \eqref{ltt 01} along the characteristics by Duhamel's principle as
\begin{align}
\reg(\vx,\vv)=& \vh(\vv)h(\vx_b,\vv)\ue^{-\nu(\vv)
t_b}+\int_{0}^{t_b}\vh(\vv)\e\ss\Big(\vx-\e(t_b-s)\vv,\vv\Big)\ue^{-\nu(\vv)
(t_b-s)}\ud{s}\\
&+\int_{0}^{t_b}\int_{\R^3}k_{\vh(\vv)}(\vv,\vuu)\reg\Big(\vx-\e(t_b-s)\vuu,\vuu\Big)\ue^{-\nu(\vv)
(t_b-s)}\ud\vuu\ud{s},\no
\end{align}
where
\begin{align}
t_b(\vx,\vv):=&\inf\big\{t>0:\vx-\e t\vv\notin\Omega\big\},
\end{align}
and
\begin{align}
\vx_b(\vx,\vv):=\vx-\e t_b(\vx,\vv)\vv\notin\Omega.
\end{align}
We further rewrite the non-local term along the characteristics as
\begin{align}\label{ltt 00}
&\reg(\vx,\vv)\\
=& \vh(\vv)h(\vx_b,\vv)\ue^{-\nu(\vv)
t_b}+\int_{0}^{t_b}\vh(\vv)\e\ss\Big(\vx-\e(t_b-s)\vv,\vv\Big)\ue^{-\nu(\vv)
(t_b-s)}\ud{s}\no\\
&+\int_{0}^{t_b}\int_{\R^3}k_{\vh(\vv)}(\vv,\vuu)\vh(\vuu)\h(x_b',\vv)\ue^{-\nu(\vuu)t_b'}\ue^{-\nu(\vv)
(t_b-s)}\ud\vuu\ud{s}\no\\
&+\int_{0}^{t_b}\int_{\R^3}k_{\vh(\vv)}(\vv,\vuu)\int_0^{t_b'}\e\ss\Big(\vx-\e(t_b-s)\vuu-\e(t_b'-r)\vuu,\vuu\Big)\ue^{-\nu(\vuu)(t_b'-r)}\ue^{-\nu(\vv)
(t_b-s)}\ud r\ud\vuu\ud{s}\no\\
&+\int_{0}^{t_b}\int_{\R^3}k_{\vh(\vv)}(\vv,\vuu)\int_0^{t_b'}\int_{\R^3}k_{\vh(\vuu)}(\vuu,\vuu')\reg\Big(\vx-\e(t_b-s)\vuu-\e(t_b'-r)\vuu',\vuu'\Big)\ue^{-\nu(\vuu)(t_b'-r)}\ue^{-\nu(\vv)
(t_b-s)}\ud\vuu'\ud r\ud\vuu\ud{s},\no
\end{align}
where
\begin{align}
t_b'(\vx,\vv;s,\vuu):=&\inf\big\{t>0:\vx-\e(t_b-s)-\e t\vuu\notin\Omega\big\},
\end{align}
and
\begin{align}
\vx_b'(\vx,\vv;s,\vuu):=\vx-\e(t_b-s)-\e t_b'(\vx,\vv;s,\vuu)\vuu\notin\Omega.
\end{align}

\paragraph{\underline{Step 2: Estimates of Source Terms and Boundary Terms}}
Based on Lemma \ref{h-estimate} -- Lemma \ref{s6-estimate}, we have
\begin{align}
    &\abs{\vh(\vv)h(\vx_b,\vv)\ue^{-\nu(\vv)t_b}}+\abs{\int_{0}^{t_b}\int_{\R^3}k_{\vh(\vv)}(\vv,\vuu)\vh(\vuu)\h(x_b',\vv)\ue^{-\nu(\vuu)t_b'}\ue^{-\nu(\vv)(t_b-s)}\ud\vuu\ud{s}}\\
    \ls&\lnmms{\h}{\gamma_-}\ls\oot,\no
\end{align}
and
\begin{align}\label{ltt 13}
    &\abs{\int_{0}^{t_b}\vh(v)\e\ss\Big(\vx-\e(t_b-s)\vv,\vv\Big)\ue^{-\nu(\vv)(t_b-s)}\ud{s}}\\
    &+\abs{\int_{0}^{t_b}\int_{\R^3}k_{\vh(\vv)}(\vv,\vuu)\int_0^{t_b'}\e\ss\Big(\vx-\e(t_b-s)\vuu-\e(t_b'-r)\vuu,\vuu\Big)\ue^{-\nu(\vuu)(t_b'-r)}\ue^{-\nu(\vv)(t_b-s)}\ud r\ud\vuu\ud{s}}\no\\
    \ls&\e\lnmm{\nu^{-1}\ss}
    \ls\oot+\oot\e\lnmm{\re}+\e\lnmm{\re}^2\ls\oot\e^{\frac{1}{2}}\xnm{\re}+\xnm{\re}^2+\oot.\no
\end{align}

\paragraph{\underline{Step 3: Estimates of Non-Local Terms}}
The only remaining term in \eqref{ltt 00} is the non-local term
\begin{align}
\\
I:=&\int_{0}^{t_b}\int_{\R^3}k_{\vh(\vv)}(\vv,\vuu)\int_0^{t_b'}\int_{\R^3}k_{\vh(\vuu)}(\vuu,\vuu')\reg\Big(\vx-\e(t_b-s)\vuu-\e(t_b'-r)\vuu',\vuu'\Big)\ue^{-\nu(\vuu)(t_b'-r)}\ue^{-\nu(\vv)
(t_b-s)}\ud\vuu'\ud r\ud\vuu\ud{s},\no
\end{align}
which will estimated in five cases:
\begin{align}
I:=I_1+I_2+I_3+I_4+I_5.
\end{align}

\subparagraph{\underline{Case I: $I_1:$ $\abs{\vv}\geq N$}}
Based on Lemma \ref{lem:kernel-operator}, we have
\begin{align}
\abs{\int_{\r^3}\int_{\r^3}k_{\vh(\vv)}(\vv,\vuu)k_{\vh(\vuu)}(\vuu,\vuu')\ud{\vuu}\ud{\vuu'}}\ls\frac{1}{1+\abs{\vv}}\ls\frac{1}{N}.
\end{align}
Hence, we get
\begin{align}\label{ltt 02}
\abs{I_1}\ls\frac{1}{N}\lnm{\reg}.
\end{align}

\subparagraph{\underline{Case II: $I_2:$ $\abs{\vv}\leq N$, $\abs{\vuu}\geq2N$, or $\abs{\vuu}\leq
2N$, $\abs{\vuu'}\geq3N$}}
Notice this implies either $\abs{\vuu-\vv}\geq N$ or
$\abs{\vuu-\vuu'}\geq N$. Hence, either of the following is valid
correspondingly:
\begin{align}
\abs{k_{\vh(\vv)}(\vv,\vuu)}&\leq C\ue^{-\d N^2}\abs{k_{\vh(\vv)}(\vv,\vuu)}\ue^{\d\abs{\vv-\vuu}^2},\\
\abs{k_{\vh(\vuu)}(\vuu,\vuu')}&\leq C\ue^{-\d N^2}\abs{k_{\vh(\vuu)}(\vuu,\vuu')}\ue^{\d\abs{\vuu-\vuu'}^2}.
\end{align}
Based on Lemma \ref{lem:kernel-operator}, we know
\begin{align}
\int_{\r^3}\abs{k_{\vh(\vv)}(\vv,\vuu)}\ue^{\d\abs{\vv-\vuu}^2}\ud{\vuu}&<\infty,\\
\int_{\r^3}\abs{k_{\vh(\vuu}(\vuu,\vuu')}\ue^{\d\abs{\vuu-\vuu'}^2}\ud{\vuu'}&<\infty.
\end{align}
Hence, we have
\begin{align}\label{ltt 03}
\abs{I_2}\ls \ue^{-\d N^2}\lnm{\reg}.
\end{align}

\subparagraph{\underline{Case III: $I_3:$ $t_b'-r\leq\d$ and $\abs{\vv}\leq N$, $\abs{\vuu}\leq 2N$, $\abs{\vuu'}\leq 3N$}}
In this case, since the integral with respect to $r$ is restricted in a very short interval, there is a small contribution as
\begin{align}\label{ltt 04}
\abs{I_3}\ls\abs{\int_{t_b'-\d}^{t_b'}\ue^{-(t_b'-r)}\ud{r}}\lnm{\reg}\ls \d\lnm{\reg}.
\end{align}

\subparagraph{\underline{Case IV: $I_4:$ $t_b'-r\geq \abs{\ln(\d)}$ and $\abs{\vv}\leq N$, $\abs{\vuu}\leq 2N$, $\abs{\vuu'}\leq 3N$}}
In this case, $t_b'-r$ is significantly large, so $\ue^{-(t_b'-r)}\leq\d$ is very small. Hence, the contribution is small 
\begin{align}\label{ltt 05}
\abs{I_4}\ls\abs{\int_{\abs{\ln(\d)}}^{\infty}\ue^{-(t_b'-r)}\ud{r}}\lnm{\reg}\ls \d\lnm{\reg}.
\end{align}

\subparagraph{\underline{Case V: $I_5:$ $\d\leq t_b'-r\leq \abs{\ln(\d)}$ and $\abs{\vv}\leq N$, $\abs{\vuu}\leq 2N$, $\abs{\vuu'}\leq 3N$}}
This is the most complicated case. Since $k_{\vh(\vv)}(\vv,\vuu)$ has
possible integrable singularity of $\abs{\vv-\vuu}^{-1}$, we can
introduce the truncated kernel $k_N(\vv,\vuu)$ which is smooth and has compactly supported range such that
\begin{align}\label{ltt 06}
\sup_{\abs{\vv}\leq 3N}\int_{\abs{\vuu}\leq
3N}\abs{k_N(\vv,\vuu)-k_{\vh(\vv)}(\vv,\vuu)}\ud{\vuu}\leq\frac{1}{N}.
\end{align}
Then we can split
\begin{align}
k_{\vh(\vv)}(\vv,\vuu)k_{\vh(\vuu)}(\vuu,\vuu')=&k_N(\vv,\vuu)k_N(\vuu,\vuu)
+\bigg(k_{\vh(\vv)}(\vv,\vuu)-k_N(\vv,\vuu)\bigg)k_{\vh(\vuu)}(\vuu,\vuu')\\
&+\bigg(k_{\vh(\vuu)}(\vuu,\vuu')-k_N(\vuu,\vuu')\bigg)k_N(\vv,\vuu).\no
\end{align}
This means that we further split $I_5$ into
\begin{align}
I_5:=I_{5,1}+I_{5,2}+I_{5,3}.
\end{align}
Based on \eqref{ltt 06}, we have
\begin{align}\label{ltt 07}
\abs{I_{5,2}}\ls&\frac{1}{N}\lnm{\reg},\quad \abs{I_{5,3}}\ls\frac{1}{N}\lnm{\reg}.
\end{align}
Therefore, the only remaining term is $I_{5,1}$. Note that we always have $\vx-\e(t_b-s)\vv-\e(t_b'-r)\vuu\in\Omega$. Hence, we define the change of variable $\vuu\rt y$ as
$y=(y_1,y_2,y_3)=\vx-\e(t_b-s)\vv-\e(t_b'-r)\vuu$. Then the Jacobian
\begin{align}
\abs{\frac{\ud{y}}{\ud{\vuu}}}=\abs{\left\vert\begin{array}{ccc}
\e(t_b'-r)&0&0\\
0&\e(t_b'-r)&0\\
0&0&\e(t_b'-r)
\end{array}\right\vert}=\e^3(t_b'-r)^3\geq \e^3\d^3.
\end{align}
Considering $\abs{\vv},\abs{\vuu},\abs{\vuu'}\leq 3N$, we know $\abs{\reg}\simeq\abs{\re}$. Also, since $k_N$ is bounded, we estimate
\begin{align}\label{ltt 08}
\\
\abs{I_{5,1}}\ls&
\int_{\abs{\vuu}\leq2N}\int_{\abs{\vuu'}\leq3N}\int_{0}^{t_b'}
\id_{\{\vx-\e(t_b-s)\vv-\e(t_b'-r)\vuu\in\Omega\}}\abs{\re\Big(\vx-\e(t_b-s)\vv-\e(t_b'-r)\vuu,\vuu'\Big)}\ue^{-\nu(\vuu)
(t_b'-r)}\ud{r}\ud{\vuu}\ud{\vuu'}.\no
\end{align}
Using H\"older's inequality, we estimate
\begin{align}\label{ltt 09}
\\
&\int_{\abs{\vuu}\leq2N}\int_{\abs{\vuu'}\leq3N}\int_{0}^{t_b'}
\id_{\{\vx-\e(t_b-s)\vv-\e(t_b'-r)\vuu\in\Omega\}}\abs{\re\Big(\vx-\e(t_b-s)\vv-\e(t_b'-r)\vuu,\vuu'\Big)}\ue^{-\nu(\vuu)
(t_b'-r)}\ud{r}\ud{\vuu}\ud{\vuu'}\no\\
\leq&\bigg(\int_{\abs{\vuu}\leq2N}\int_{\abs{\vuu'}\leq3N}\int_{0}^{t_b'}
\id_{\{\vx-\e(t_b-s)\vv-\e(t_b'-r)\vuu\in\Omega\}}\ue^{-\nu(\vuu)
(t_b'-r)}\ud{r}\ud{\vuu}\ud{\vuu'}\bigg)^{\frac{5}{6}}\no\\
&\times\bigg(\int_{\abs{\vuu}\leq2N}\int_{\abs{\vuu'}\leq3N}\int_{0}^{t_b'}
\id_{\{\vx-\e(t_b-s)\vv-\e(t_b'-r)\vuu\in\Omega\}}\abs{\re\Big(\vx-\e(t_b-s)\vv-\e(t_b'-r)\vuu,\vuu'\Big)}^{6}\ue^{-\nu(\vuu)
(t_b'-r)}\ud{r}\ud{\vuu}\ud{\vuu'}\bigg)^{\frac{1}{6}}\no\\
\ls&\abs{\int_{0}^{t_b'}\frac{1}{\e^3\d^3}\int_{\abs{\vuu'}\leq3N}
\int_{\Omega}\id_{\{y\in\Omega\}}\babs{\re(y,\vuu')}^6\ue^{-(t_b'-r)}\ud{y}\ud{\vuu'}\ud{r}}^{\frac{1}{6}}
\ls \frac{1}{\e^{\frac{1}{2}}\d^{\frac{1}{2}}}\pnm{\re}{6}.\no
\end{align}
Inserting \eqref{ltt 09} into \eqref{ltt 08}, we obtain
\begin{align}
\abs{I_{5,1}}\ls \frac{1}{\e^{\frac{1}{2}}\d^{\frac{1}{2}}}\pnm{\re}{6}.
\end{align}
Combined with \eqref{ltt 07}, we know
\begin{align}\label{ltt 10}
I_5\ls \frac{1}{N}\lnm{\reg}+\frac{1}{\e^{\frac{1}{2}}\d^{\frac{1}{2}}}\pnm{\re}{6}.
\end{align}
Summarizing all five cases in \eqref{ltt 02}\eqref{ltt 03}\eqref{ltt 04}\eqref{ltt 05}\eqref{ltt 10}, we obtain
\begin{align}
\abs{I}\ls \bigg(\frac{1}{N}+\ue^{-\d N^2}+\d\bigg)\lnm{\reg}+\frac{1}{\e^{\frac{1}{2}}\d^{\frac{1}{2}}}\pnm{\re}{6}.
\end{align}
Choosing $\d$ sufficiently small and then taking $N$ sufficiently large, we have
\begin{align}
\abs{I}\ls \d\lnm{\reg}+\frac{1}{\e^{\frac{1}{2}}\d^{\frac{1}{2}}}\pnm{\re}{6}.
\end{align}

\paragraph{\underline{Step 4: Synthesis}}
Summarizing all above, we obtain for any $(\vx,\vv)\in\overline{\Omega}\times\R^3$,
\begin{align}
\babs{\reg(\vx,\vv)}\ls\d\lnm{\reg}+\frac{1}{\e^{\frac{1}{2}}\d^{\frac{1}{2}}}\pnm{\re}{6}+\oot\e^{\frac{1}{2}}\xnm{\re}+\xnm{\re}^2+\oot.
\end{align}
Hence, when $\d\ll1$, we obtain
\begin{align}
\babs{\reg(\vx,\vv)}\ls\e^{-\frac{1}{2}}\pnm{\re}{6}+\oot\e^{\frac{1}{2}}\xnm{\re}+\xnm{\re}^2+\oot,
\end{align}
and thus the desired result follows from Proposition \ref{prop:kernel}.
\end{proof}


\subsection{Remainder Estimate}

\begin{theorem}\label{thm:priori}
    Under the assumption \eqref{assumption:stationary}, we have 
    \begin{align}
        \xnm{\re}\ls\oot.
    \end{align}
\end{theorem}
\begin{proof}
    Based on Proposition \ref{prop:energy} and Corollary \ref{lem:energy-extension}, we have
    \begin{align}
        \e^{-\frac{1}{2}}\tnms{\re}{\gamma_+}+\pnms{\m^{\frac{1}{4}}\re}{4}{\gamma_+}+\e^{-1}\tnm{\ire}+\pnm{\ire}{6}\ls \oot\xnm{\re}+\xnm{\re}^2+\oot.
    \end{align}
    Based on Proposition \ref{prop:kernel}, we have
    \begin{align}
        \e^{-\frac{1}{2}}\tnm{\bre}+\pnm{\bre}{6}\ls\oot\xnm{\re}+\xnm{\re}^2+\oot.
    \end{align}
    Combining both of them, we have
    \begin{align}\label{priori 01}
        \e^{-\frac{1}{2}}\tnms{\re}{\gamma_+}+\pnms{\m^{\frac{1}{4}}\re}{4}{\gamma_+}+\e^{-\frac{1}{2}}\tnm{\bre}+\e^{-1}\tnm{\ire}+\pnm{\re}{6}\ls \oot\xnm{\re}+\xnm{\re}^2+\oot.
    \end{align}
    Based on Proposition \ref{prop:infty}, we have
    \begin{align}\label{priori 02}
    \e^{\frac{1}{2}}\lnmm{\re}+\e^{\frac{1}{2}}\lnmms{\re}{\gamma_+}
    \ls \oot\xnm{\re}+\xnm{\re}^2+\oot.
    \end{align}
    Collecting \eqref{priori 01}\eqref{priori 02}, we have
    \begin{align}
    \xnm{\re}\ls\oot\xnm{\re}+\xnm{\re}^2+\oot.
    \end{align}
    Hence, we have
    \begin{align}
        \xnm{\re}\ls\xnm{\re}^2+\oot.
    \end{align}
    By a standard iteration/fixed-point argument, our desired result follows.
\end{proof}

\begin{proof}[Proof of Theorem \ref{main theorem}]
    The estimate \eqref{main'} follows from Theorem \ref{thm:priori}. The construction and positivity of $\fs$ based on the expansion \eqref{expand} is standard and we refer to \cite{AA023, Esposito.Guo.Kim.Marra2015}, so we will focus on the proof of \eqref{main}. From Theorem \ref{thm:priori}, we have
    \begin{align}
        \e^{-\frac{1}{2}}\tnm{\bre}+\e^{-1}\tnm{\ire}\ls \oot,
    \end{align}
    which yields
    \begin{align}
        \tnm{\re}\ls\oot\e^{\frac{1}{2}}.
    \end{align}
    From \eqref{expand}, we know
    \begin{align}
        \tnm{\mhh\fs-\mh-\e\f_1-\e^2\f_2-\e\fb_1}=\tnm{\e\re}\ls\oot\e^{\frac{3}{2}}.
    \end{align}
    From Theorem \ref{thm:approximate-solution} and the rescaling $\eta=\e^{-1}\mn$, we have
    \begin{align}
        \tnm{\e^2\f_2}\ls\oot\e^2,\quad \tnm{\e\fb_1}\ls\oot\e^{\frac{3}{2}}.
    \end{align}
    Hence, we have
    \begin{align}
        \tnm{\mhh\fs-\mh-\e\f_1}\ls\oot\e^{\frac{3}{2}}.
    \end{align}
    Therefore, \eqref{main} follows.
\end{proof}


\bigskip
\section{Evolutionary Problem}

\subsection{Asymptotic Analysis}\label{sec:expansion-e}

\subsubsection{Interior Solution}

The derivation of the interior solution is classical. We refer to \cite{Sone2002, Sone2007, Golse2005} and the references therein. By inserting \eqref{expand 1=} into \eqref{large system=} and comparing the order of $\e$, we require that
\begin{align}
    0&=2\mhh\qq\big[\m,\mh\f_1\big],\\
    \vv\cdot\nx\f_1&=2\mhh\qq\big[\m,\mh\f_2\big]+\mhh\qq\big[\mh\f_1,\mh\f_1\big],
\end{align}
which are equivalent to
\begin{align}
    \lc\left[\f_1\right]&=0,\\
    \vv\cdot\nx\f_1+\lc\left[\f_2\right]&=\Gamma\left[\f_1,\f_1\right].
\end{align}
Considering the further expansion, we additionally require
\begin{align}
    \dt\f_1+\vv\cdot\nx\f_2\perp\nk.
\end{align}
Hence, we conclude
\begin{align}
    \f_1(t,\vx,\vv)=\mh(\vv)\left(\rq_1(t,x)+\vv\cdot\uq_1(t,x)+\frac{\abs{\vv}^2-3}{2}\tq_1(t,x)\right),
\end{align}
where $(\rq_1,\uq_1,\tq_1)$ satisfies the incompressible Navier-Stokes-Fourier system \eqref{fluid=}.

Also, we have 
\begin{align}
\f_2(t,\vx,\vv)=&\;\mh(\vv)\left(\rq_2(t,x) +\vv\cdot\uq_2(t,x)+\frac{\abs{\vv}^2-3}{2}\tq_2(t,x)\right)\\
&\;+\mh(\vv)\left(\rq_1(\vv\cdot\uq_1)+\left(\rq_1\tq_1+\frac{\abs{\vv}^2-3}{2}\abs{\uq_1}^2\right)\right)+\li\Big[-\vv\cdot\nx\f_1+\Gamma[\f_1,\f_1]\Big]\no
\end{align}
where $(\rq_2,\uq_2,\tq_2)$ satisfies the fluid system
\begin{align} \label{fluid'=}
\\
\left\{
\begin{array}{l}
\nx P_2=0,\\\rule{0ex}{1.5em}
\dt\uq_2+\uq_1\cdot\nx\uq_2+(\rq_1\uq_1+\uq_2)\cdot\nx\uq_1-\gamma_1\dx\uq_2 +\nx\mathfrak{p}_2 =-\gamma_2\nx\cdot\dx\tq_1-\gamma_4\nx\cdot\Big(\tq_1\big(\nx\uq_1+(\nx\uq_1)^T\big)\Big),\\\rule{0ex}{1.5em}
\nx\cdot\uq_2 =-\uq_1\cdot\nx\rq_1,\\\rule{0ex}{1.5em}
\dt\tq_2+\uq_1\cdot\nx\tq_2+(\rq_1\uq_1+\uq_2)\cdot\nx\tq_1-\uq_1\cdot\nx\mathfrak{p}_1=\gamma_1\Big(\nx\uq_1+(\nx\uq_1)^T\Big)^2+\dx\big(\gamma_2\tq_2+\gamma_5\tq_1^2\big),\no
\end{array}
\right.
\end{align}
for $P_2:=\rq_2 +\tq_2 +\rq_1 \tq_1$ and constants $\gamma_3,\gamma_4,\gamma_5$.

\subsubsection{Boundary Layer}

We define a cutoff boundary layer $\fb_1$ as \eqref{boundary layer}. Denote
\begin{align}\label{boundary layer=}
    \fb_1(t,\eta,\vvv)=\ch\left(\e^{-1}\va\right)\chi(\e\eta)\blff (t,\eta,\vvv).
\end{align}
We may verify that $\fb_1$ satisfies
\begin{align}
    \va\dfrac{\p\fb_1 }{\p\eta}+\lc\left[\fb_1\right]=&\va\ch(\e^{-1}\va)\frac{\p\chi(\e\eta)}{\p\eta}\blff+\chi(\e\eta)\bigg(\ch(\e^{-1}\va)K\Big[\blff\Big]-K\Big[\ch(\e^{-1}\va)\blff\Big]\bigg),
\end{align}
with
\begin{align}
    \fb_1(t,0,\vvv)=\ch\left(\e^{-1}\va\right)\Big(\fss_b(\vvv)-\blfi(\vvv)\Big)\ \ \text{for}\ \ \va>0.
\end{align}

\subsubsection{Matching Procedure}

The construction of boundary layer and the boundary condition of the interior solution is exactly the same as Section \ref{sec:matching}, so we only discuss the initial condition of the interior solution.

Using \eqref{assumption:evolutionary2}, we require the matching condition for $t=0$:
\begin{align}
    \rq_1\Big|_{t=0}=\rq^I,\quad\uq_1\Big|_{t=0}=\uq^I,\quad\tq_1\Big|_{t=0}=\tq^I. 
\end{align} 
By standard fluid theory \cite{Boyer.Fabrie2013, Cattabriga1961} for the unsteady Navier-Stokes equations \eqref{fluid=}, we have for any $\NN\in[2,\infty)$
\begin{align}
    \nm{\rq_1}_{W^{1,\infty}W^{3,\NN}}+\nm{\uq_1}_{W^{1,\infty}W^{3,\NN}}+\nm{\tq_1}_{W^{1,\infty}W^{3,\NN}}\ls\oot.
\end{align}
Also, for $\f_2$, since there is no initial layer, we may simply take 
\begin{align}
    \rq_2\Big|_{t=0}=0,\quad\uq_2\Big|_{t=0}=\od,\quad\tq_2\Big|_{t=0}=0. 
\end{align} 
By standard fluid theory \cite{Boyer.Fabrie2013, Cattabriga1961} for the linear unsteady Navier-Stokes equations \eqref{fluid'=}, we have for any $\NN\in[2,\infty)$
\begin{align}
    \nm{\rq_2}_{W^{1,\infty}W^{2,\NN}}+\nm{\uq_2}_{W^{1,\infty}W^{2,\NN}}+\nm{\tq_2}_{W^{1,\infty}W^{2,\NN}}\ls\oot.
\end{align}

\begin{theorem}\label{thm:approximate-solution=}
    Under the assumption \eqref{assumption:evolutionary1}\eqref{assumption:evolutionary2}\eqref{assumption:evolutionary3}, there exists a unique solution $(\rq_1,\uq_1,\tq_1$ to the unsteady Navier-Stokes equations \eqref{fluid=} satisfying for any $\NN\in[2,\infty)$
    \begin{align}
    \nm{\rq_1}_{W^{1,\infty}W^{3,\NN}}+\nm{\uq_1}_{W^{1,\infty}W^{2,\NN}}+\nm{\tq_1}_{W^{1,\infty}W^{3,\NN}}\ls\oot.
    \end{align}
    Also, we can construct $\f_1$, $\f_2$ and $\fb_1$ such that 
    \begin{align}
    \nm{f_1}_{W^{1,\infty}W^{3,\NN}L^{\infty}_{\varrho,\vartheta}}+\abs{f_1}_{W^{1,\infty}W^{3-\frac{1}{\NN},\NN}L^{\infty}_{\varrho,\vartheta}}&\ls\oot,\\
    \nm{f_2}_{W^{1,\infty}W^{2,\NN}L^{\infty}_{\varrho,\vartheta}}+\abs{f_2}_{W^{1,\infty}W^{2-\frac{1}{\NN},\NN}L^{\infty}_{\varrho,\vartheta}}&\ls\oot,
    \end{align}
and for some $K_0>0$ and any $0<\N\leq 3$
\begin{align}
    \nm{\ue^{K_0\eta}\fb_1}_{W^{1,\infty}L^{\infty}_{\vrh,\vth}}+\nm{\ue^{K_0\eta}\frac{\p^\N\fb_1}{\p\iota_1^\N}}_{W^{1,\infty}L^{\infty}_{\vrh,\vth}}+\nm{\ue^{K_0\eta}\frac{\p^\N\fb_1}{\p\iota_2^\N}}_{W^{1,\infty}L^{\infty}_{\vrh,\vth}}&\ls\oot.
\end{align}
\end{theorem}


\subsection{Remainder Equation}

Inserting \eqref{expand=} into \eqref{large system=}, we have
\begin{align}
    &\e\dt\left(\m+\f+\fb+\e\mh\re\right)+\vv\cdot\nx\left(\m+\f+\fb+\e\mh\re\right)\\
    =&\e^{-1}\qq\left[\m+\f+\fb+\e\mh\re,\m+\f+\fb+\e\mh\re\right]\no
\end{align}
or equivalently
\begin{align}
    &\e\dt\re+\vv\cdot\nx\re-2\e^{-1}\mhh\qq\big[\m,\mh\re\big]\\
    =&-\mhh\dt\left(\f+\fb\right)-\e^{-1}\mhh\Big(\vv\cdot\nx\left(\f+\fb\right)\Big)+\mhh\qq\left[\mh\re,\mh\re\right]\\
    &+2\e^{-1}\mhh\qq\left[\f+\fb,\mh\re\right]+\e^{-2}\mhh\qq\left[\m+\f+\fb,\m+\f+\fb\right].\no
\end{align}
Also, we have the initial and boundary conditions
\begin{align}
    \left(\m+\f+\e\mh\re\right)\Big|_{t=0}=\m+\e\mh\fss_i,\qquad \left(\m+\f+\fb+\e\mh\re\right)\Big|_{\ga_-}=\m+\e\mh\fss_b,
\end{align}
which are equivalent to 
\begin{align}
    \re\big|_{t=0}=\fss_i-\e^{-1}\mhh\f,\qquad \re\big|_{\ga_-}=\fss_b-\e^{-1}\mhh\big(\f+\fb\big).
\end{align}
Note that due to the compatibility condition \eqref{assumption:evolutionary3}, the boundary layer has no influence on the initial data. 

Therefore, we need to consider the remainder equation \eqref{remainder=}.
Here the initial data is given by
\begin{align}\label{d:z=}
    \z:=-\e\f_2,
\end{align}
the boundary data is given by
\begin{align}\label{d:h=}
    \h=-\e\f_2+\chi\big(\e^{-1}\va\big)\blff,
\end{align}
and
\begin{align}
    \ss:=\ssa+\ssc+\ssd+\ssf+\ssg+\ssh,
\end{align}
where
\begin{align}
    \ssa:=&-\e\dt\f_1-\e^2\dt\f_2-\e\vv\cdot\nx\f_2,\label{d:ssa=}\\
    \ssc:=&-\e\dt\fb_1+\dfrac{1}{R_1-\e\eta}\bigg(\vb^2\dfrac{\p\fb_1}{\p\va}-\va\vb\dfrac{\p\fb_1}{\p\vb}\bigg)+\dfrac{1}{R_2-\e\eta}\bigg(\vc^2\dfrac{\p \fb_1}{\p\va}-\va\vc\dfrac{\p\fb_1}{\p\vc}\bigg)\label{d:ssc=}\\
    &-\dfrac{1}{\pl_1\pl_2}\left(\dfrac{R_1\p_{\iota_1\iota_1}\vr\cdot\p_{\iota_2}\vr}{\pl_1(R_1-\e\eta)}\vb\vc
    +\dfrac{R_2\p_{\iota_1\iota_2}\vr\cdot\p_{\iota_2}\vr}{\pl_2(R_2-\e\eta)}\vc^2\right)\dfrac{\p\fb_1}{\p\vb}\no\\
    &-\dfrac{1}{\pl_1\pl_2}\left(\dfrac{R_2\p_{\iota_2\iota_2}\vr\cdot\p_{\iota_1}\vr}{\pl_2(R_2-\e\eta)}\vb\vc
    +\dfrac{R_1\p_{\iota_1\iota_2}\vr\cdot\p_{\iota_1}\vr}{\pl_1(R_1-\e\eta)}\vb^2\right)\dfrac{\p\fb_1}{\p\vc}\no\\
    &-\left(\dfrac{R_1\vb}{\pl_1(R_1-\e\eta)}\dfrac{\p\fb_1}{\p\iota_1}+\dfrac{R_2\vc}{\pl_2(R_2-\e\eta)}\dfrac{\p\fb_1}{\p\iota_2}\right)\no\\
    &+\e^{-1}\va\ch(\e^{-1}\va)\frac{\p\chi(\e\eta)}{\p\eta}\blff-\e^{-1}\Big(K\left[\blff\right]\chi(\e^{-1}\va)\chi(\e\eta)-K\Big[\blff\chi(\e^{-1}\va)\chi(\e\eta)\Big]\Big),\no\\
    \ssd:=&2\mhh\qq\left[\mh\f_1+\e\mh\f_2,\mh\re\right]=2\Gamma[\f_1+\e\f_2,\re],\label{d:ssd=}\\
    \ssf:=&2\mhh\qq\left[\mh\fb_1,\mh\re\right]=2\Gamma\left[\fb_1,\re\right],\label{d:ssf=}\\
    \ssg:=&\e\mhh\qq\left[\mh\f_2,\mh\big(2\f_1+\e\f_2\big)\right]\label{d:ssg=}+2\mhh\qq\left[\mh\big(2\f_1+2\e\f_2+\fb_1\big),\mh\fb_1\right]\\
    =&\e\Gamma\left[\f_2,2\f_1+\e\f_2\right]+\Gamma\left[2\f_1+2\e\f_2+\fb_1,\fb_1\right],\no\\
    \ssh:=&\mhh\qq\left[\mh\re,\mh\re\right]=\Gamma[\re,\re].\label{d:ssh=}
\end{align}
In particular, we may further split $\ssc$:
\begin{align}
    \ssx:=&\dfrac{1}{R_1-\e\eta}\bigg(\vb^2\dfrac{\p\fb_1}{\p\va}\bigg)
    +\dfrac{1}{R_2-\e\eta}\bigg(\vc^2\dfrac{\p \fb_1}{\p\va}\bigg),\\
    \ssy:=&-\e\dt\fb_1-\dfrac{1}{R_1-\e\eta}\bigg(\va\vb\dfrac{\p\fb_1}{\p\vb}\bigg)
    -\dfrac{1}{R_2-\e\eta}\bigg(\va\vc\dfrac{\p\fb_1}{\p\vc}\bigg)\\
    &-\dfrac{1}{\pl_1\pl_2}\left(\dfrac{R_1\p_{\iota_1\iota_1}\vr\cdot\p_{\iota_2}\vr}{\pl_1(R_1-\e\eta)}\vb\vc
    +\dfrac{R_2\p_{\iota_1\iota_2}\vr\cdot\p_{\iota_2}\vr}{\pl_2(R_2-\e\eta)}\vc^2\right)\dfrac{\p\fb_1}{\p\vb}\no\\
    &-\dfrac{1}{\pl_1\pl_2}\left(\dfrac{R_2\p_{\iota_2\iota_2}\vr\cdot\p_{\iota_1}\vr}{\pl_2(R_2-\e\eta)}\vb\vc
    +\dfrac{R_1\p_{\iota_1\iota_2}\vr\cdot\p_{\iota_1}\vr}{\pl_1(R_1-\e\eta)}\vb^2\right)\dfrac{\p\fb_1}{\p\vc}\no\\
    &-\left(\dfrac{R_1\vb}{\pl_1(R_1-\e\eta)}\dfrac{\p\fb_1}{\p\iota_1}+\dfrac{R_2\vc}{\pl_2(R_2-\e\eta)}\dfrac{\p\fb_1}{\p\iota_2}\right)+\e^{-1}\va\ch(\e^{-1}\va)\frac{\p\chi(\e\eta)}{\p\eta}\blff,\no\\
    \ssz:=&-\e^{-1}\Big(K\left[\blff\right]\chi(\e^{-1}\va)\chi(\e\eta)-K\Big[\blff\chi(\e^{-1}\va)\chi(\e\eta)\Big]\Big).
\end{align}

We also consider the time derivative of remainder equation
\begin{align}\label{remainder==}
\left\{
\begin{array}{l}\displaystyle
\e\dt\big(\dt\re\big)+\vv\cdot\nabla_x \big(\dt\re\big)+\e^{-1}\lc[\dt\re]=\dt\ss\ \ \text{in}\ \ \rp\times\Omega\times\R^3,\\\rule{0ex}{1.5em}
\dt\re(0,\vx,\vv)=\dt\z(\vx,\vv)\ \ \text{in}\ \ \Omega\times\R^3,\\\rule{0ex}{1.5em}
\dt\re(t,\vx_0,\vv)=\dt\h(t,\vx_0,\vv)\ \ \text{for}\
\ \vv\cdot\vn<0\ \ \text{and}\ \ \vx_0\in\p\Omega.
\end{array}
\right.
\end{align}
Here the initial data $\dt\z$ is solved from \eqref{expand=} and Remark \ref{rmk:initial}:
\begin{align}\label{d:z==}
    \dt\z:=\dt\re\big|_{t=0}=\Big(\e^{-1}\mhh\dt\fs-\dt\f_1-\e\dt\f_2\Big)\Big|_{t=0}.
\end{align}
For any fixed $t\in\rp$, we may also rewrite \eqref{remainder=} as a stationary remainder equation
\begin{align}\label{remainder=s}
\left\{
\begin{array}{l}\displaystyle
\vv\cdot\nabla_x \re(t)+\e^{-1}\lc[\re(t)]=\ss(t)-\e\dt\re(t)\ \ \text{in}\ \ \Omega\times\R^3,\\\rule{0ex}{1.5em}
\re(t,\vx_0,\vv)=\h(t,\vx_0,\vv)\ \ \text{for}\
\ \vv\cdot\vn<0\ \ \text{and}\ \ \vx_0\in\p\Omega.
\end{array}
\right.
\end{align}

\begin{lemma}[Green's Identity, Lemma 2.2 of \cite{Esposito.Guo.Kim.Marra2013}]\label{lem:green-identity=}
Assume $f(t,\vx,\vv),\ g(t,\vx,\vv)\in L^{\infty}([0,T]; L^2_{\nu}(\Omega\times\r^3))$ and
$\dt f+\vv\cdot\nx f,\ \dt g+\vv\cdot\nx g\in L^2([0,T]\times\Omega\times\r^3)$ with $f,\
g\in L^2_{\ga}$. Then for almost all $t,s\in[0,T]$
\begin{align}
&\int_s^t\iint_{\Omega\times\r^3}\big(\dt f+\vv\cdot\nx f\big)g+\int_s^t\iint_{\Omega\times\r^3}\big(\dt g+\vv\cdot\nx
g\big)f\\
=&\iint_{\Omega\times\r^3}f(t)g(t)-\iint_{\Omega\times\r^3}f(s)g(s)+\int_s^t\int_{\gamma}fg(v\cdot n).\no
\end{align}
\end{lemma}

Using Lemma \ref{lem:green-identity=}, we can derive the weak formulation of \eqref{remainder=}. For any test function $\test(t,x,v)\in L^{\infty}([0,T]; L^2_{\nu}(\Omega\times\r^3))$ with $\dt\test+\vv\cdot\nx\test\in L^2([0,T]\times\Omega\times\r^3)$ with $\test\in L^2_{\ga}$, we have
\begin{align}\label{weak formulation=}
    \e\bbr{\re(t),\test(t)}-\e\bbr{\z,\test(0)}-\e\dbbr{\re,\dt\test}
    +\int_{\ga}\re\test(v\cdot n)-\dbbr{v\cdot\nx \test,\re}+\e^{-1}\dbbr{\lc[\re],\test}&=\dbbr{\ss,\test}.
\end{align}

\subsubsection{Estimates of Initial, Boundary and Source Terms}\label{sec:source=}

The estimates below follow from analogous argument as \cite[Section 3]{AA023} with $\alpha=1$ and Section \ref{sec:source}, so we omit the details and only highlight the key differences. In particular, for $\ssa$--$\ssh$ estimates, we need both the accumulative $L^p_t$ and instantaneous $L^{\infty}_t$ versions. For simplicity, we only write down the instantaneous estimates (which is similar to those in Section \ref{sec:source}) and the accumulative ones are similar.

\paragraph{\underline{Estimates of $\z$}}

\begin{lemma}\label{z-estimate=}
Under the assumption \eqref{assumption:evolutionary1}\eqref{assumption:evolutionary2}\eqref{assumption:evolutionary3}, for $\z$ defined in \eqref{d:z=}, we have
    \begin{align}
    \tnm{\z}\ls\oot\e,\qquad \lnmm{\z}\ls \oot\e.
    \end{align}
In addition,  for $\dt\z$ defined in \eqref{d:z==}, we have
    \begin{align}
    \tnm{\dt\z}\ls\oot\e,\qquad \lnmm{\dt\z}\ls \oot\e.
    \end{align}
\end{lemma}
\begin{proof}
    The estimates follow from Remark \ref{rmk:initial} and Theorem \ref{thm:approximate-solution=}.
\end{proof}

\paragraph{\underline{Estimates of $\h$}}

\begin{lemma}\label{h-estimate=}
Under the assumption \eqref{assumption:evolutionary1}\eqref{assumption:evolutionary2}\eqref{assumption:evolutionary3}, for $h$ defined in \eqref{d:h=}, we have
\begin{align}
    \tnms{h}{\gamma_-}&\ls\oot\e,\quad
    \jnms{h}{\gamma_-}\ls\oot\e^{\frac{3}{\N}},\quad
    \lnmms{h}{\gamma_-}\ls\oot,\quad \sup_{\iota_1,\iota_2}\int_{v\cdot n<0}\abs{h}\abs{v\cdot n}\ud v\ls\oot\e.
\end{align}
In addition, $\dt\h$ satisfies exactly the same estimates as above.
\end{lemma}

\paragraph{\underline{Estimates of $\ssa$}}

\begin{lemma}\label{s1-estimate=}
Under the assumption
\eqref{assumption:evolutionary1}\eqref{assumption:evolutionary2}\eqref{assumption:evolutionary3}, for $\ssa$ defined in \eqref{d:ssa=}, we have
\begin{align}
    \btnm{\!\br{v}^2\!\ssa}\ls \oot\e,\qquad
    \pnm{\ssa}{\N}\ls\oot\e,\qquad
    \lnmm{\ssa}\ls\oot\e.
\end{align}
Also, we have the property
\begin{align}
\\
    \brv{\mh,\ssa}=\e^2\brv{\mh,\dt\f_2},\quad\brv{\mh\vv,\ssa}=\e^2\brv{\mh\vv,\dt\f_2},\quad \brv{\mh\abs{\vv}^2,\ssa}=\e^2\brv{\mh\abs{\vv}^2,\dt\f_2}.\no
\end{align}
In addition, $\dt\ssa$ satisfies exactly the same estimates as above.
\end{lemma}

\paragraph{\underline{Estimates of $\ssc$}}

\begin{lemma}\label{s2-estimate=}
Under the assumption
\eqref{assumption:evolutionary1}\eqref{assumption:evolutionary2}\eqref{assumption:evolutionary3}, for $\ssc$ defined in \eqref{d:ssc=}, we have
\begin{align}
    \pnm{\ssc}{1}+\pnm{\eta\left(\ssy+\ssz\right)}{1}+\pnm{\eta^2\left(\ssy+\ssz\right)}{1}&\ls \oot \e,\label{ss3-estimate1=}\\
    \tnm{\br{v}^2\ssc}+\tnm{\eta\left(\ssy+\ssz\right)}+\tnm{\eta^2\left(\ssy+\ssz\right)}&\ls \oot,\label{ss3-estimate2=}\\
    \pnm{\ssc}{\N}+\pnm{\eta\left(\ssy+\ssz\right)}{\N}+\pnm{\eta^2\left(\ssy+\ssz\right)}{\N}&\ls\oot\e^{\frac{2}{\N}-1},\label{ss3-estimate3=}\\
    \nm{\ssc}_{L^{\N}_{\iota_1\iota_2}L^1_{\mn}L^1_v}+\nm{\eta\left(\ssy+\ssz\right)}_{L^{\N}_{\iota_1\iota_2}L^1_{\mn}L^1_v}&\ls\oot\e,\label{ss3-estimate4=},
\end{align}
and
\begin{align}
    \nm{\ssy+\ssz}_{L^{\N}_{x}L^1_v}+\nm{\eta\left(\ssy+\ssz\right)}_{L^{\N}_{x}L^1_v}&\ls\oot\e^{\frac{1}{\N}},\label{ss3-estimate5=}\\
    \abs{\br{\ssx,g}}+\abs{\br{\eta\ssx,g}}+\abs{\br{\eta^2\ssx,g}}&\ls\knm{\br{v}^2\fb_1}\jnm{\nabla_v g}\ls\oot\e^{1-\frac{1}{\N}}\jnm{\nabla_v g}.\label{ss3-estimate6=}
\end{align}
Also, we have
\begin{align}
    \lnmm{\ssc}&\ls\oot\e^{-1}.
\end{align}
In addition, $\dt\ssc$ satisfies exactly the same estimates as above.
\end{lemma}

\begin{remark}\label{rmk:s2=}
    Notice that the BV estimate in Theorem \ref{boundary regularity} does not contain exponential decay in $\eta$, and thus we cannot directly bound $\eta\ssx$ and $\eta^2\ssx$. Instead, we should first integrate by parts with respect to $\va$ as in \eqref{ss3-estimate6=} to study $\fb_1$ instead:
    \begin{align}
        \pnm{\fb_1}{\N}+\pnm{\eta\fb_1}{\N}+\pnm{\eta^2\fb_1}{\N}&\ls\oot\e^{\frac{2}{\N}-1},\label{ss3-estimate7=}\\
        \nm{\fb_1}_{L^{\N}_{\iota_1\iota_2}L^1_{\mn}L^1_v}+\nm{\eta\fb_1}_{L^{\N}_{\iota_1\iota_2}L^1_{\mn}L^1_v}&\ls\oot\e,\label{ss3-estimate8=},\\
        \nm{\fb_1}_{L^{\N}_{x}L^1_v}+\nm{\eta\fb_1}_{L^{\N}_{x}L^1_v}&\ls\oot\e^{\frac{1}{\N}}.
    \end{align}
\end{remark}

\paragraph{\underline{Estimates of $\ssd$}}

\begin{lemma}\label{s3-estimate=}
Under the assumption
\eqref{assumption:evolutionary1}\eqref{assumption:evolutionary2}\eqref{assumption:evolutionary3},  for $\ssd$ defined in \eqref{d:ssd=}, we have
\begin{align}
    \abs{\brv{\ssd,g}}\ls \oot\e\left(\int_{\r^3}\nu\abs{g}^2\right)^{\frac{1}{2}}\left(\int_{\r^3}\nu\abs{\re}^2\right)^{\frac{1}{2}},
\end{align}
and thus
\begin{align}
    \abs{\br{\ssd,g}}\ls \oot\e\um{g}\um{\re}\ls\oot\e\um{g}\left(\tnm{\bre}+\um{\ire}\right).
\end{align}
Also, we have
\begin{align}
    \tnm{\ssd}\ls\oot\e\um{\re},\qquad
    \lnmm{\nu^{-1}\ssd}\ls\oot\e\lnmm{\re}.
\end{align}
In addition, $\dt\ssd$ satisfies the similar estimates as above with $\re$ replaced by $\dt\re$.
\end{lemma}

\paragraph{\underline{Estimates of $\ssf$}}

\begin{lemma}\label{s4-estimate=}
Under the assumption
\eqref{assumption:evolutionary1}\eqref{assumption:evolutionary2}\eqref{assumption:evolutionary3}, for $\ssf$ defined in \eqref{d:ssf=}, we have
\begin{align}
    \abs{\brv{\ssf,g}}\ls \left(\int_{\r^3}\nu\abs{g}^2\right)^{\frac{1}{2}}\left(\int_{\r^3}\nu\abs{\fb_1}^2\right)^{\frac{1}{2}}\left(\int_{\r^3}\nu\abs{\re}^2\right)^{\frac{1}{2}},
\end{align}
and thus
\begin{align}
    \abs{\br{\ssf,g}}\ls& \oot\um{g}\um{\re}\ls\oot\um{g}\left(\tnm{\bre}+\um{\ire}\right),\\
    \abs{\br{\ssf,g}}\ls& \oot\um{\fb_1}\lnmm{g}\um{\re}\ls\oot\e^{\frac{1}{2}}\lnmm{g}\left(\tnm{\bre}+\um{\ire}\right).
\end{align}
Also, we have
\begin{align}
    \tnm{\ssf}\ls\oot\um{\re},\qquad
    \lnmm{\nu^{-1}\ssf}\ls\oot\lnmm{\re}.
\end{align}
In addition, $\dt\ssf$ satisfies the similar estimates as above with $\re$ replaced by $\dt\re$.
\end{lemma}

\paragraph{\underline{Estimates of $\ssg$}}

\begin{lemma}\label{s5-estimate=}
Under the assumption
\eqref{assumption:evolutionary1}\eqref{assumption:evolutionary2}\eqref{assumption:evolutionary3}, for $\ssg$ defined in \eqref{d:ssg=}, we have
\begin{align}
    \abs{\brv{\ssg,g}}\ls \oot\left(\int_{\r^3}\nu\abs{g}^2\right)^{\frac{1}{2}},
\end{align}
and thus
\begin{align}
    \abs{\br{\ssg,g}}\ls\oot\e^{\frac{1}{2}}\um{g},\qquad
    \abs{\br{\ssg,g}}\ls\oot\e\lnmm{g}.
\end{align}
Also, we have
\begin{align}
    \tnm{\ssg}&\ls\oot\e^{\frac{1}{2}},\qquad
    \lnmm{\nu^{-1}\ssg}\ls\oot.
\end{align}
In addition, $\dt\ssg$ satisfies exactly the same estimates as above.
\end{lemma}

\paragraph{\underline{Estimates of $\ssh$}}

Note that $\dt\Gamma[\re,\re]=\Gamma[\re,\dt\re]$. Then the proof follows from that of Lemma \ref{s6-estimate}.
\begin{lemma}\label{s6-estimate=}
Under the assumption
\eqref{assumption:evolutionary1}\eqref{assumption:evolutionary2}\eqref{assumption:evolutionary3}, for $\ssh$ defined in \eqref{d:ssh=}, we have
\begin{align}
    \abs{\brv{\ssh,g}}\ls \left(\int_{\r^3}\nu\abs{g}^2\right)^{\frac{1}{2}}\left(\int_{\r^3}\nu\abs{\re}^2\right),
\end{align}
and thus
\begin{align}
    \abs{\br{\ssh,g}}\ls& \um{g}\um{\re}\lnmm{\re}.
\end{align}
Also, we have
\begin{align}
    \tnm{\ssh}&\ls\um{\re}\lnmm{\re},\\
    \lnmm{\nu^{-1}\ssh}&\ls\lnmm{\re}^2.
\end{align}
In addition, we have
\begin{align}
    \abs{\brv{\dt\ssh,g}}\ls \left(\int_{\r^3}\nu\abs{g}^2\right)^{\frac{1}{2}}\left(\int_{\r^3}\nu\abs{\re}\abs{\dt\re}\right),
\end{align}
and thus
\begin{align}
    \abs{\br{\dt\ssh,g}}\ls& \um{g}\um{\dt\re}\lnmm{\re}.
\end{align}
Also, we have
\begin{align}
    \tnm{\dt\ssh}&\ls\um{\dt\re}\lnmm{\re},\\
    \lnmm{\nu^{-1}\dt\ssh}&\ls\lnmm{\dt\re}\lnmm{\re}.
\end{align}
\end{lemma}

\subsubsection{Conservation Laws}

    \paragraph{\underline{Classical Conservation Laws}}

\begin{lemma}\label{lem:conservation=}
    Under the assumption
    \eqref{assumption:evolutionary1}\eqref{assumption:evolutionary2}\eqref{assumption:evolutionary3}, we have the conservation laws
    \begin{align}
    \e\dt\big(\P-\cc\big)+\nx\cdot\bb=&\brv{\mh,\ssa+\ssc},\label{conservation law 1=}\\
    \e\dt\bb+\nx\P+\nx\cdot\varpi=&\brv{v\mh,\ssa+\ssc},\label{conservation law 2=}\\
    \e\dt\big(3\P\big)+5\nx\cdot\bb+\nx\cdot\varsigma=&\brv{\abs{v}^2\mh,\ssa+\ssc},\label{conservation law 3=}
    \end{align}
    where $\varpi$ and $\varsigma$ are defined in Lemma \ref{lem:conservation}.
\end{lemma}
\begin{proof}
    We multiply test functions $\mh,v\mh,\abs{v}^2\mh$ on both sides of \eqref{remainder=} and integrate over $v\in\r^3$. Using the orthogonality of $\lc$ and noticing
    \begin{align}
        \int_{\r^3}\mh\re=\P-\cc,\qquad\int_{\r^3}\vv\mh\re=\bb,\qquad \int_{\r^3}\abs{\vv}^2\mh\re=3\P.
    \end{align}
    the results follow. 
\end{proof}

\paragraph{\underline{Conservation Law with Test Function $\nx\varphi\cdot\a$}}

\begin{lemma}
    Under the assumption
    \eqref{assumption:evolutionary1}\eqref{assumption:evolutionary2}\eqref{assumption:evolutionary3}, for smooth test function $\varphi(t,x)$, we have
    \begin{align}\label{conservation law 4=}
    &\e\bbr{\re(t),\nx\varphi(t)\cdot\a}-\e\bbr{\z,\nx\varphi(0)\cdot\a}-\e\dbbr{\re,\dt\nx\varphi\cdot\a}-\k\dbbrx{\dx\varphi,c}+\e^{-1}\dbbr{\nx\varphi,\varsigma}\\
    =&\dbbrb{\nx\varphi\cdot\a,h}{\ga_-}-\dbbrb{\nx\varphi\cdot\a,\re}{\ga_+}+\dbr{v\cdot\nx\Big(\nx\varphi\cdot\a\Big),\ire}+\dbbr{\nx\varphi\cdot\a,\ss}.\no
    \end{align}
\end{lemma}
\begin{proof}
Taking test function $\test=\nx\varphi\cdot\a$ in \eqref{weak formulation=}, we obtain
\begin{align}
    \e\bbr{\re(t),\nx\varphi(t)\cdot\a}-\e\bbr{\z,\nx\varphi(0)\cdot\a}-\e\dbbr{\re,\dt\nx\varphi\cdot\a}&\\
    +\int_{\ga}\Big(\nx\varphi\cdot\a\Big)\re\big(v\cdot n\big)-\dbbr{v\cdot\nx\Big(\nx\varphi\cdot\a\Big), \re}+\e^{-1}\dbbr{\lc[\re],\nx\varphi\cdot\a}&=\dbbr{\nx\varphi\cdot\a,\ss}.\no
\end{align}
Then following the similar argument as the proof of Lemma \ref{lem:conservation 1}, we have \eqref{conservation law 4=}.
\end{proof}

\paragraph{\underline{Conservation Law with Test Function $\nx\psi:\b$}}

\begin{lemma}
    Under the assumption
    \eqref{assumption:evolutionary1}\eqref{assumption:evolutionary2}\eqref{assumption:evolutionary3}, for smooth test function $\psi(t,x)$ satisfying $\nx\cdot\psi=0$, we have
    \begin{align}\label{conservation law 5=}
    &\e\bbr{\re(t),\nx\psi(t):\b}-\e\bbr{\z,\nx\psi(0):\b}-\e\dbbr{\re,\dt\nx\psi:\b}-\lambda\dbbrx{\dx\psi,\bb}+\e^{-1}\dbbr{\nx\psi,\varpi}\\
    =&\dbbrb{\nx\psi\cdot\b,h}{\ga_-}-\dbbrb{\nx\psi\cdot\b,\re}{\ga_+}+\dbr{v\cdot\nx\Big(\nx\psi:\b\Big),\ire}+\dbbr{\nx\psi:\b,\ss}.\no
    \end{align}
    \end{lemma}
\begin{proof}
Taking test function $\test=\nx\psi:\b$ in \eqref{weak formulation=}, we obtain
\begin{align}
    \e\bbr{\re(t),\nx\psi(t):\b}-\e\bbr{\z,\nx\psi(0):\b}-\e\dbbr{\re,\dt\nx\psi:\b}\\
    +\int_{\ga}\Big(\nx\psi:\b\Big)\re\big(v\cdot n\big)-\dbbr{v\cdot\nx\Big(\nx\psi:\b\Big), \re}+\e^{-1}\dbbr{\lc[\re],\nx\psi:\b}&=\dbbr{\nx\psi:\b,\ss}.\no
\end{align}
Then following the similar argument as the proof of Lemma \ref{lem:conservation 2}, we have \eqref{conservation law 5=}.
\end{proof}

\paragraph{\underline{Conservation Law with Test Function $\nx\varphi\cdot\a+\e^{-1}\varphi\big(\abs{v}^2-5\big)\mh$}}

\begin{lemma}
    Under the assumption
    \eqref{assumption:evolutionary1}\eqref{assumption:evolutionary2}\eqref{assumption:evolutionary3}, for smooth test function $\varphi(t,x)$ satisfying $\varphi\big|_{\p\Omega}=0$, we have
    \begin{align}\label{conservation law 6=}
    &\bbrx{5\cc(t)-2\P(t),\varphi(t)}-\bbrx{5\cc(0)-2\P(0),\varphi(0)}-\dbbrx{5\cc-2\P,\dt\varphi}\\
    &
    +\e\bbr{\re(t),\nx\varphi(t)\cdot\a}-\e\bbr{\z,\nx\varphi(0)\cdot\a}-\e\dbbr{\re,\dt\nx\varphi\cdot\a}-\k\dbbrx{\dx\varphi,c}\no\\
    =&\dbbrb{\nx\varphi\cdot\a,h}{\ga_-}-\dbbrb{\nx\varphi\cdot\a,\re}{\ga_+}+\dbr{v\cdot\nx\Big(\nx\varphi\cdot\a\Big),\ire}\no\\
    &+\e^{-1}\dbr{\varphi\left(\abs{v}^2-5\right)\mh,\ss}+\dbbr{\nx\varphi\cdot\a,\ss}.\no
    \end{align}
\end{lemma}
\begin{proof}
From \eqref{conservation law 1=} and \eqref{conservation law 3=}, we have
\begin{align}\label{cc 06=}
    \e\dt\big(5\cc-2\P\big)+\nx\cdot\varsigma=&\brv{\left(\abs{v}^2-5\right)\mh,\ss}.
\end{align}
Multiplying $\varphi(t,x)\in\r$ on both sides of \eqref{cc 06=} and integrating over $(t,x)\in[0,t]\times\Omega$, we obtain
\begin{align}\label{cc 01=}
\e\bbrx{5\cc(t)-2\P(t),\varphi(t)}-\e\bbrx{5\cc(0)-2\P(0),\varphi(0)}-\e\dbbrx{5\cc-2\P,\dt\varphi}\\-\dbbrx{\nx\varphi,\varsigma}+\int_0^t\int_{\p\Omega}\varphi\varsigma\cdot n=&\dbr{\varphi\left(\abs{v}^2-5\right)\mh,\ss}.\no
\end{align}
Hence, adding $\e^{-1}\times$\eqref{cc 01=} and \eqref{conservation law 4=} to eliminate $\e^{-1}\dbbrx{\nx\varphi,\varsigma}$ yields
\begin{align}\label{cc 02=}
    &\bbrx{5\cc(t)-2\P(t),\varphi(t)}-\bbrx{5\cc(0)-2\P(0),\varphi(0)}-\dbbrx{5\cc-2\P,\dt\varphi}\\
    &
    +\e\bbr{\re(t),\nx\varphi(t)\cdot\a}-\e\bbr{\z,\nx\varphi(0)\cdot\a}-\e\dbbr{\re,\dt\nx\varphi\cdot\a}-\k\dbbrx{\dx\varphi,c}+\e^{-1}\int_0^t\int_{\p\Omega}\varphi\varsigma\cdot n\no\\
    =&\dbbrb{\nx\varphi\cdot\a,h}{\ga_-}-\dbbrb{\nx\varphi\cdot\a,\re}{\ga_+}+\dbr{v\cdot\nx\Big(\nx\varphi\cdot\a\Big),\ire}\no\\
    &+\e^{-1}\dbr{\varphi\left(\abs{v}^2-5\right)\mh,\ss}+\dbbr{\nx\varphi\cdot\a,\ss}.\no
    \end{align}
The assumption $\varphi\big|_{\p\Omega}=0$
completely eliminates the boundary term $\ds\e^{-1}\int_0^t\int_{\p\Omega}\varphi\varsigma\cdot n$ in \eqref{cc 02=}. Hence, we have \eqref{conservation law 6=}.

\end{proof}

\paragraph{\underline{Conservation Law with Test Function $\nx\psi:\b+\e^{-1}\psi\cdot v\mh$}}

\begin{lemma}
    Under the assumption
    \eqref{assumption:evolutionary1}\eqref{assumption:evolutionary2}\eqref{assumption:evolutionary3}, for smooth test function  $\psi(t,x)$ satisfying $\nx\cdot\psi=0$, $\psi\big|_{\p\Omega}=0$, we have
    \begin{align}\label{conservation law 7=}
    &\bbrx{\bb(t),\psi(t)}-\bbrx{\bb(0),\psi(0)}-\dbbrx{\bb,\dt\psi}\\
    &+\e\bbr{\re(t),\nx\psi(t):\b}-\e\bbr{\z,\nx\psi(0):\b}-\e\dbbr{\re,\dt\nx\psi:\b}-\lambda\dbbrx{\dx\psi,\bb}\no\\
    =&\dbbrb{\nx\psi:\b,h}{\ga_-}-\dbbrb{\nx\psi:\b,\re}{\ga_+}+\dbr{v\cdot\nx\Big(\nx\psi:\b\Big),\ire}\no\\
    &+\e^{-1}\dbr{\psi\cdot v\mh,\ss}+\dbbr{\nx\psi:\b,\ss}.\no
    \end{align}
\end{lemma}
\begin{proof}
Multiplying $\psi(t, x)\in\r^3$ on both sides of \eqref{conservation law 2=} and integrating over $(t,x)\in[0,t]\times\Omega$, we obtain
\begin{align}\label{cc 03=}
\e\bbrx{\bb(t),\psi(t)}-\e\bbrx{\bb(0),\psi(0)}-\e\dbbrx{\bb,\dt\psi}\\
-\dbbrx{\nx\cdot\psi, \P}-\dbbrx{\nx\psi,\varpi}+\int_0^t\int_{\p\Omega}\Big(\P\psi+\psi\cdot\varpi\Big)\cdot n &=\dbr{\psi\cdot v\mh,\ss}.\no
\end{align}
Hence, adding $\e^{-1}\times$\eqref{cc 03=} and \eqref{conservation law 5=} to eliminate $\e^{-1}\dbbrx{\nx\psi,\varpi}$ yields
\begin{align}\label{cc 04=}
&\bbrx{\bb(t),\psi(t)}-\bbrx{\bb(0),\psi(0)}-\dbbrx{\bb,\dt\psi}\\
&+\e\bbr{\re(t),\nx\psi(t):\b}-\e\bbr{\z,\nx\psi(0):\b}-\e\dbbr{\re,\dt\nx\psi:\b}\no\\
&-\lambda\dbbrx{\dx\psi,\bb}-\e^{-1}\dbbrx{\nx\cdot\psi, \P}+\e^{-1}\int_0^t\int_{\p\Omega}\Big(\P\psi+\psi\cdot\varpi\Big)\cdot n\no\\
=&\dbbrb{\nx\psi:\b,h}{\ga_-}-\dbbrb{\nx\psi:\b,\re}{\ga_+}+\dbr{v\cdot\nx\Big(\nx\psi:\b\Big),\ire}\no\\
&+\e^{-1}\dbr{\psi\cdot v\mh,\ss}+\dbbr{\nx\psi:\b,\ss}.\no
\end{align}
The assumptions $\nx\cdot\psi=0$ and $\psi\big|_{\p\Omega}=0$ 
eliminates $\e^{-1}\dbbrx{\nx\cdot\psi, \P}$ and $\ds\e^{-1}\int_0^t\int_{\p\Omega}\Big(\P\psi+\psi\cdot\varpi\Big)\cdot n$ in \eqref{cc 04=}. 
Hence, we have \eqref{conservation law 7=}.
\end{proof}


\subsection{Energy Estimate -- Accumulative}

\begin{proposition}\label{prop:energy=}
    Under the assumption
    \eqref{assumption:evolutionary1}\eqref{assumption:evolutionary2}\eqref{assumption:evolutionary3}, we have 
    \begin{align}\label{est:energy=}
    \tnm{\re(t)}+\e^{-\frac{1}{2}}\tnnms{\re}{\ga_+}+\e^{-1}\tnnm{\ire}\ls\oot\xnnm{\re}+\xnnm{\re}^2+\oot.
    \end{align}
\end{proposition}

\begin{proof}
It suffices to justify 
\begin{align}\label{eq:energy=}
    \tnm{\re(t)}+\e^{-\frac{1}{2}}\tnnms{\re}{\ga_+}+\e^{-1}\tnnm{\ire}\ls \oot\e^{-\frac{1}{2}}\tnnm{\bre}+\oot\xnnm{\re}+\xnnm{\re}^2+\oot.
\end{align}

\paragraph{\underline{Weak Formulation}}
Taking test function $\test=\e^{-1}\re$ in \eqref{weak formulation=}, we obtain
\begin{align}
    \frac{1}{2}\tnm{\re(t)}^2-\frac{1}{2}\tnm{\z}^2+\frac{\e^{-1}}{2}\int_{\ga}\re^2(\vv\cdot\vn)+\e^{-2}\dbbr{\lc[\re],\re}=\e^{-1}\dbbr{\ss,\re}.
\end{align}
Notice that
\begin{align}
   \int_{\ga}\re^2(\vv\cdot\vn)=\tnnms{\re}{\ga_+}^2-\tnnms{\re}{\ga_-}^2=\tnnms{\re}{\ga_+}^2-\tnnms{\h}{\ga}^2,
\end{align}
and
\begin{align}
    \dbbr{\lc[\re],\re}\gs\unm{\ire}^2.
\end{align}
Then we know
\begin{align}
    \tnm{\re(t)}^2+\e^{-1}\tnnms{\re}{\ga_+}^2+\e^{-2}\unm{\ire}^2\ls \abs{\e^{-1}\dbbr{\ss,\re}}+\e^{-1}\tnnms{\h}{\ga}^2+\tnm{\z}^2.
\end{align}
Using Lemma \ref{z-estimate=} and Lemma \ref{h-estimate=}, we have 
\begin{align}\label{energy 01=}
    \tnm{\re(t)}^2+\e^{-1}\tnnms{\re}{\ga_+}^2+\e^{-2}\unm{\ire}^2\ls \abs{\e^{-1}\dbbr{\ss,\re}}+\oot\e.
\end{align}

\paragraph{\underline{Source Term Estimates}}
We split 
\begin{align}
    \e^{-1}\dbbr{\ss,\re}=\e^{-1}\dbbr{\ss,\bre}+\e^{-1}\dbbr{\ss,\ire}.
\end{align}
We may directly bound using Lemma \ref{s1-estimate=} -- Lemma \ref{s6-estimate=}
\begin{align}
    \abs{\e^{-1}\dbbr{\ss,\ire}}\ls& \e^{-1}\tnnm{\ss}\tnnm{\ire}\\
    \ls& \big(\oo+\oot\big)\e^{-2}\tnnm{\ire}^2+\oot\xnnm{\re}^2+\xnnm{\re}^4+\oot.\no
\end{align}
Using orthogonality of $\Gamma$, we have
\begin{align}
    \e^{-1}\dbbr{\ss,\bre}=\e^{-1}\dbbr{\ssa+\ssc,\bre}.
\end{align}
From Lemma \ref{s1-estimate=}, we know
\begin{align}
    \abs{\e^{-1}\dbbr{\ssa,\bre}}=\e\abs{\dbbr{\dt\f_2,\bre}}\ls \oot\e\tnnm{\bre}\ls\oot\tnnm{\bre}^2+\oot\e^2.
\end{align}
Also, from Lemma \ref{s2-estimate=} and Remark \ref{rmk:s2=}, after integrating by parts with respect to $\va$ in $\ssx$ term, we obtain
\begin{align}
    \abs{\e^{-1}\dbbr{\ssc,\bre}}&\ls\e^{-1}\nnm{\fb_1+\ssy+\ssz}_{L^2_{tx}L^1_v}\nnm{\bre}_{L^2_{tx}L^{\infty}_v}\\
    &\ls\oot\e^{-\frac{1}{2}}\tnnm{\bre}\ls\oot\e^{-1}\tnnm{\bre}^2+\oot.\no
\end{align}
In total, we have
\begin{align}\label{energy 02=}
    \abs{\e^{-1}\dbbr{\ss,\re}}\ls \oot\e^{-1}\tnnm{\bre}^2+\big(\oo+\oot\big)\e^{-2}\tnnm{\ire}^2+\oot\xnnm{\re}^2+\xnnm{\re}^4+\oot.
\end{align}

\paragraph{\underline{Synthesis}}
Inserting \eqref{energy 02=} into \eqref{energy 01=}, we have 
\begin{align}
    \tnm{\re(t)}^2+\e^{-1}\tnnms{\re}{\ga_+}^2+\e^{-2}\unm{\ire}^2\ls \oot\e^{-1}\tnnm{\bre}^2+\oot\xnnm{\re}^2+\xnnm{\re}^4+\oot.
    \end{align}
Then we have \eqref{eq:energy=}.
\end{proof}


\begin{proposition}\label{prop:energy=t}
    Under the assumption
    \eqref{assumption:evolutionary1}\eqref{assumption:evolutionary2}\eqref{assumption:evolutionary3}, we have 
    \begin{align}\label{energy=t}
    \tnm{\dt\re(t)}
    +\e^{-\frac{1}{2}}\tnnms{\dt\re}{\ga_+}+\e^{-1}\tnnm{(\ik-\pk)[\dt\re]}
    \ls\oot\xnnm{\re}+\xnnm{\re}^2+\oot.
    \end{align}
\end{proposition}
\begin{proof}
Applying the similar argument as in the proof of Proposition \ref{prop:energy=} to the equation \eqref{remainder==}, using Lemma \ref{z-estimate=} -- Lemma \ref{s6-estimate=}, we obtain the desired result. In particular, we should use $\dt\z$ estimates in Lemma \ref{z-estimate=}.
\end{proof}


\subsection{Kernel Estimate -- Accumulative}

\subsubsection{Estimate of $\P$}

\begin{proposition}\label{prop:p-bound=}
    Under the assumption
    \eqref{assumption:evolutionary1}\eqref{assumption:evolutionary2}\eqref{assumption:evolutionary3}, we have
    \begin{align}\label{est:p-bound=}
        \e^{-\frac{1}{2}}\tnnm{\P}\ls\e^{-\frac{1}{2}}\tnnm{\bb}+\oot\xnnm{\re}+\xnnm{\re}^2+\oot.
    \end{align}
\end{proposition}

\begin{proof}
It suffices to show
\begin{align}\label{eq:p-bound=}
    \tnnm{\P}\ls\e^{\frac{1}{2}}\tnm{\re(t)}+\tnnms{\re}{\ga_+}+\tnnm{\bb}+\tnnm{\ire}+\oot\e.
\end{align}

\paragraph{\underline{Weak Formulation}}
Denote
\begin{align}
\psi(t,x,v):=\mh(\vv)\Big(\vv\cdot\nx\varphi(t,x)\Big),
\end{align}
where $\varphi(t,x)$ is defined via solving the elliptic problem
\begin{align}
\left\{
\begin{array}{l}
-\dx\varphi=\ds \P\ \ \text{in}\ \ \Omega,\\\rule{0ex}{1.5em}
\varphi=0\ \ \text{on}\ \ \p\Omega.
\end{array}
\right.
\end{align}
Based on standard elliptic estimates \cite{Krylov2008} and trace theorem, there exists a solution $\varphi$ satisfying
\begin{align}
\tnms{\psi(t)}{\gamma}+\nm{\psi(t)}_{H^{1}L^{\infty}_{\vrh,\vth}}\ls\nm{\varphi(t)}_{H^{2}}\ls\tnm{\P(t)}.
\end{align}
Taking test function $\test=\psi$ in \eqref{weak formulation=}, we obtain
\begin{align}
    \e\bbr{\re(t),\psi(t)}-\e\bbr{\z,\psi(0)}-\e\dbbr{\re,\dt\psi}+\int_{\ga}\re\psi(\vv\cdot\vn)-\dbbr{\re,\vv\cdot\nx\psi}
    &=\dbbr{\ss,\psi}.
\end{align}
From Lemma \ref{z-estimate=}, we know
\begin{align}
    \abs{\e\bbr{\re(t),\psi(t)}}&\ls\e\tnm{\re(t)}\tnm{\psi(t)}\ls \e\tnm{\re(t)}\tnm{\P(t)}\ls \e\tnm{\re(t)}^2,\\
    \abs{\e\bbr{\re(0),\psi(0)}}&\ls\e\tnm{\re(0)}\tnm{\psi(0)}\ls  \e\tnm{\z}^2\ls\oot\e^3,
\end{align}
and from oddness and orthogonality
\begin{align}
    \abs{\e\dbbr{\re,\dt\psi}}=\abs{\e\bbr{\mh\big(\vv\cdot\bb\big),\dt\psi}}\ls \tnnm{\bb}^2+\e^2\tnnm{\dt\nx\varphi}^2.
\end{align}
Based on Lemma \ref{h-estimate=}, we know
\begin{align}
    \abs{\int_{\ga}\re\psi(\vv\cdot\vn)}
    &\ls\tnnms{\re}{\ga_+}\tnnms{\psi}{\ga_+}+\tnnms{\h}{\ga_-}\tnnms{\psi}{\ga_-}\\
    &\ls\oo\tnnms{\psi}{\ga}^2+ \tnnms{\re}{\ga_+}^2+\tnnms{\h}{\ga_-}^{2}
   \ls\oo\tnnm{\P}^{2}+\tnnms{\re}{\ga_+}^{2}+\oot\e^2.\no
\end{align}
Due to oddness and orthogonality, we have
\begin{align}
\dbr{\mh\big(\vv\cdot\bb\big),\vv\cdot\nx\psi}=\dbbr{\ire,\vv\cdot\nx\psi}=0.
\end{align}
Due to orthogonality of $\ab$, we know
\begin{align}
&\dbr{\mh\frac{\abs{\vv}^2-5}{2}\cc,\vv\cdot\nx\psi}
=\dbbr{c,\mh\ab\cdot\nx\psi}=0.
\end{align}
Also, we have
\begin{align}
-\dbr{\mh\P ,\vv\cdot\nx\psi}
=&-\dbr{\P\m,\vv\cdot\nx\Big(\vv\cdot\nx\varphi\Big)}
=-\frac{1}{3}\int_0^t\int_{\Omega}p\big(\dx\varphi \big)\int_{\r^3}\m\abs{\vv}^2
=\tnnm{\P}^{2}.
\end{align}
Collecting the above, we have
\begin{align}\label{kernel-p01=}
    \tnnm{\P}^2\ls\e\tnm{\re(t)}^2+\tnnms{\re}{\ga_+}^2+\tnnm{\bb}^2+\e^2\tnnm{\dt\nx\varphi}^2+\oot\e^2+\abs{\dbbr{\ss,\psi}}.
\end{align}

\paragraph{\underline{Source Term Estimates}}
Due to the orthogonality of $\Gamma$ and Lemma \ref{s1-estimate}, we know
\begin{align}
    \dbbr{\ss,\psi}=\dbbr{\ssa+\ssc,\psi}.
\end{align}
Using Lemma \ref{s1-estimate=}, we have
\begin{align}
    \abs{\dbbr{\ssa,\psi}}=\e^2\abs{\dbbr{\dt\f_2,\psi}}\ls\oot\tnnm{\P}^2+\oot\e^4.
\end{align}
Using Hardy's inequality and integrating by parts with respect to $\va$ in $\ssx$, based on Lemma \ref{s2-estimate=} and Remark \ref{rmk:s2=}, we have
\begin{align}\label{kernel-p00=}
    &\abs{\dbbr{\ssc,\psi}}\leq\abs{\dbr{\ssc,\psi\Big|_{\mn=0}}}+\abs{\dbr{\ssc,\int_0^{\mn}\p_{\mn}\psi}}
    =\abs{\dbr{\ssc,\psi\Big|_{\mn=0}}}+\abs{\e\dbr{\eta\ssc,\frac{1}{\mn}\int_0^{\mn}\p_{\mn}\psi}}\\
    \ls&\nnm{\fb_1+\ssy+\ssz}_{L^2_tL^2_{\iota_1\iota_2}L^1_{\mn}L^1_{\vv}}\tnnms{\psi}{\ga}+\e\tnnm{\eta\big(\fb_1+\ssy+\ssz\big)}\tnnm{\frac{1}{\mn}\int_0^{\mn}\p_{\mn}\psi}\no\\
    \ls&\nnm{\fb_1+\ssy+\ssz}_{L^2_tL^2_{\iota_1\iota_2}L^1_{\mn}L^1_{\vv}}\tnnms{\psi}{\ga}+\e\tnnm{\eta\big(\fb_1+\ssy+\ssz\big)}\tnnm{\p_{\mn}\psi}\no\\
    \ls&\;\oot\e\tnnms{\psi}{\ga}+\oot\e\tnnm{\p_{\mn}\psi}\ls \oot\e\tnnm{\P}\ls\oot\tnnm{\P}^{2}+\oot\e^{2}.\no
\end{align}
Collecting the above, we have shown
\begin{align}\label{kernel-p02=}
    \abs{\dbr{\ss,\psi}}\ls\oot\tnnm{\P}^{2}+\oot\e^{2}
\end{align}
Inserting \eqref{kernel-p02=} into \eqref{kernel-p01=}, we have
\begin{align}\label{kernel-p03=}
    \tnnm{\P}^2\ls\e\tnm{\re(t)}^2+\tnnms{\re}{\ga_+}^2+\tnnm{\bb}^2+\e^2\tnnm{\dt\nx\varphi}^2+\oot\e^2.
\end{align}

\paragraph{\underline{Estimate of $\tnnm{\dt\nx\varphi}$}}
Denote $\Phi=\dt\varphi$. Taking $\test=\e\Phi\abs{\vv}^2\mh$ in \eqref{weak formulation=}, due to orthogonality and $\Phi\big|_{\p\Omega}=0$, we obtain
\begin{align}
    \e^2\dbbr{\dt\re,\Phi\abs{\vv}^2\mh}-\e\dbbr{\re,\vv\cdot\nx\left(\Phi\abs{\vv}^2\mh\right)}
    &=\e\dbbr{\ss,\Phi\abs{\vv}^2\mh}.
\end{align}
Notice that 
\begin{align}
    \e^2\dbbr{\dt\re,\Phi}=3\e^2\dbbr{\dt\P,\Phi}=-3\e^2\dbbr{\dx\Phi,\Phi}=3\e^2\dbbr{\nx\Phi,\nx\Phi}=3\e^2\tnnm{\dt\nx\varphi}^2.
\end{align}
Also, we know
\begin{align}
    \abs{\e\dbbr{\re,\vv\cdot\nx\left(\Phi\abs{\vv}^2\mh\right)}}
    \leq&\abs{\e\dbbr{\mh\big(\vv\cdot\bb\big),\vv\abs{\vv}^2\mh\cdot\nx\Phi}}+\abs{\e\dbbr{\ire,\vv\abs{\vv}^2\mh\cdot\nx\Phi}}\\
    \ls& \tnnm{\bb}^2+\tnnm{\ire}^2+\oo\e^2\tnnm{\dt\nx\varphi}^2.\no
\end{align}
Then, by a similar argument as the above estimates for $\dbbr{\ss,\psi}$, we have
\begin{align}
    \abs{\e\dbbr{\ss,\Phi\abs{\vv}^2\mh}}=\abs{\e\bbr{\ssa+\ssc,\Phi\abs{\vv}^2\mh}}\ls \oot\e^2\tnnm{\dt\nx\varphi}^2+\oot\e^2.
\end{align}
Collecting the above, we have
\begin{align}\label{kernel-p04=}
    \e^2\tnnm{\dt\nx\varphi}^2\ls \tnnm{\bb}^2+\tnnm{\ire}^2+\oot\e^2.
\end{align}
Inserting \eqref{kernel-p04=} into \eqref{kernel-p03=}, we have
\begin{align}
    \tnnm{\P}^2\ls\e\tnm{\re(t)}^2+\tnnms{\re}{\ga_+}^2+\tnnm{\bb}^2+\tnnm{\ire}^2+\oot\e^2.
\end{align}
Then \eqref{eq:p-bound=} follows.
\end{proof}

\begin{proposition}\label{prop:p-bound=t}
    Under the assumption
    \eqref{assumption:evolutionary1}\eqref{assumption:evolutionary2}\eqref{assumption:evolutionary3}, we have
    \begin{align}\label{est:p-bound=t}
    \e^{-\frac{1}{2}}\tnnm{\dt\P}\ls\e^{-\frac{1}{2}}\tnnm{\dt\bb}+\oot\xnnm{\re}+\xnnm{\re}^2+\oot.
    \end{align}
\end{proposition}
\begin{proof}
Applying the similar argument as in the proof of Proposition \ref{prop:p-bound=} to the equation \eqref{remainder==}, we obtain the desired result. 
\end{proof}


\subsubsection{Estimate of $\cc$}

\begin{proposition}\label{prop:c-bound=}
Under the assumption
    \eqref{assumption:evolutionary1}\eqref{assumption:evolutionary2}\eqref{assumption:evolutionary3}, we have
    \begin{align}\label{est:c-bound=}
    \e^{-\frac{1}{2}}\tnnm{\cc}\ls\e^{-\frac{1}{2}}\tnnm{\P}+\oot\xnnm{\re}+\xnnm{\re}^{2}+\oot.
    \end{align}
\end{proposition}

\begin{proof}
It suffices to justify
\begin{align}\label{eq:c-bound=}
    \tnnm{\cc}\ls&\;\e^{\frac{1}{12}}\xnnm{\re}^{\frac{1}{2}}\tnnm{\cc}+\e^{\frac{1}{2}}\tnm{\re(t)}+\tnnms{\re}{\ga_+}+\tnnm{\P}+\tnnm{\ire}\\
    &\;+\oot\e^{\frac{1}{2}}\xnnm{\re}+\e^{\frac{1}{2}}\xnnm{\re}^{2}+\oot\e^{\frac{1}{2}}.\no
\end{align}

\paragraph{\underline{Weak Formulation}}
We consider the conservation law \eqref{conservation law 6=}
where the smooth test function $\varphi(t,x)$ satisfies
\begin{align}
\left\{
\begin{array}{l}
-\dx\varphi=5\cc-2\P\ \ \text{in}\ \ \Omega,\\\rule{0ex}{1.5em}
\varphi=0\ \ \text{on}\ \ \p\Omega.
\end{array}
\right.
\end{align}
Based on the standard elliptic estimates \cite{Krylov2008} and trace theorem, there exists a solution $\varphi$ satisfying
\begin{align}
    \tnms{\nx\varphi(t)}{\p\Omega}+\nm{\varphi(t)}_{H^2}\ls\tnm{\cc(t)}+\tnm{\P(t)}.
\end{align}
From \eqref{conservation law 6=}, we have
\begin{align}
\\
    &\dbbrx{\dt\big(5\cc-2\P\big),\varphi}
    +\e\bbr{\re(t),\nx\varphi(t)\cdot\a}-\e\bbr{\z,\nx\varphi(0)\cdot\a}-\e\dbbr{\re,\dt\nx\varphi\cdot\a}-\k\dbbrx{\dx\varphi,\cc}\no\\
    =&\dbbrb{\nx\varphi\cdot\a,h}{\ga_-}-\dbbrb{\nx\varphi\cdot\a,\re}{\ga_+}+\dbr{v\cdot\nx\Big(\nx\varphi\cdot\a\Big),\ire}\no\\
    &+\e^{-1}\dbr{\varphi\left(\abs{v}^2-5\right)\mh,\ss}+\dbbr{\nx\varphi\cdot\a,\ss}.\no
\end{align}
Direct computation reveals that
\begin{align}
    \dbbrx{\dt\big(5\cc-2\P\big),\varphi}=&-\dbbrx{\dt\dx\varphi,\varphi}=\dbbrx{\dt\nx\varphi,\nx\varphi}=\frac{1}{2}\tnm{\nx\varphi(t)}^2-\frac{1}{2}\tnm{\nx\varphi(0)}^2,
\end{align}
and from Lemma \ref{z-estimate=} 
\begin{align}
    \tnm{\nx\varphi(0)}^2\ls\tnm{\P(0)}^2+\tnm{\cc(0)}^2\ls\tnm{\z}^2\ls\oot\e^2.
\end{align}
Also, we have
\begin{align}
   -\k\dbbrx{\dx\varphi,\cc}=\k\dbbrx{5\cc-\P,\cc}=5\k\tnnm{\cc}^2-\k\dbbrx{\P,\cc}
\end{align}
with
\begin{align}
    \abs{\k\dbbrx{\P,\cc}}\ls\oo\tnnm{\cc}^2+\tnnm{\P}^2.
\end{align}
Using Lemma \ref{z-estimate=}, we have
\begin{align}
   \abs{\e\bbr{\re(t),\nx\varphi(t)\cdot\a}}\ls&\e\tnm{\re(t)}\nm{\varphi(t)}_{H^1}\ls\e\tnm{\re(t)}\Big(\tnm{\cc(t)}+\tnm{\P(t)}\Big)\ls\e\tnm{\re(t)}^2, \\
   \abs{\e\bbr{\re(0),\nx\varphi(0)\cdot\a}}\ls&\e\tnm{\re(0)}\nm{\varphi(0)}_{H^1}\ls\e\tnm{\z}^2\ls\oot\e^3.
\end{align}
Due to the orthogonality of $\a$, we know
\begin{align}
    \abs{\e\dbbr{\re,\dt\nx\varphi\cdot\a}}=\abs{\e\dbbr{\ire,\dt\nx\varphi\cdot\a}}\ls \tnnm{\ire}^2+\oo\e^2\tnnm{\dt\nx\varphi}^2.
\end{align}
Using Lemma \ref{h-estimate=}, we have
\begin{align}
    \abs{\dbbrb{\nx\varphi\cdot\a,\h}{\ga_-}}\ls&\tnnms{\nx\varphi\cdot\a}{\ga_-}\tnnms{\h}{\ga_-}\ls\oot\tnnm{\cc}^2+\oot\tnnm{\P}^2+\oot\e^2,\\
    \abs{\dbbrb{\nx\varphi\cdot\a,\re}{\ga_+}}\ls&\tnnms{\nx\varphi\cdot\a}{\ga_+}\tnnms{\re}{\ga_+}\ls\oo\tnnm{\cc}^2+\oo\tnnm{\P}^2+\tnnms{\re}{\ga_+}^2,
\end{align}
and
\begin{align}
    \abs{\dbr{v\cdot\nx\Big(\nx\varphi\cdot\a\Big),\ire}}&\ls\tnnm{v\cdot\nx\Big(\nx\varphi\cdot\a\Big)}\tnnm{\ire}\\
    &\ls\oo\tnnm{\cc}^2+\oo\tnnm{\P}^2+\tnnm{\ire}^2.\no
\end{align}
Collecting the above, we have shown that
\begin{align}\label{kernel-c01=}
\\
    \tnm{\nx\varphi(t)}^2+\tnnm{\cc}^2\ls&\;\e\tnm{\re(t)}^2+\tnnms{\re}{\ga_+}^2+\tnnm{\P}^2+\tnnm{\ire}^2+\oo\e^2\tnnm{\dt\nx\varphi}^2+\oot\e^2\no\\
    &\;+\abs{\e^{-1}\dbr{\varphi\big(\abs{v}^2-5\big)\mh,\ss}}+\abs{\dbbr{\nx\varphi\cdot\a,\ss}}.\no
\end{align}

\paragraph{\underline{Source Term Estimates}}
Due to the orthogonality of $\Gamma$, we have
\begin{align}
    \e^{-1}\dbr{\varphi\big(\abs{v}^2-5\big)\mh,\ss}=\e^{-1}\dbr{\varphi\big(\abs{v}^2-5\big)\mh,\ssa+\ssc}.
\end{align}
Based on Lemma \ref{s1-estimate=}, we have
\begin{align}
    \abs{\e^{-1}\dbr{\varphi\big(\abs{v}^2-5\big)\mh,\ssa}}=\abs{\e\dbr{\varphi\big(\abs{v}^2-5\big)\mh,\dt\f_2}}\ls \oot\e\tnnm{\cc}\ls\oot\tnnm{\cc}^2+\oot\e^2.
\end{align}
Similar to \eqref{kernel-p00=}, based on Lemma \ref{s2-estimate=}, Remark \ref{rmk:s2=} and Hardy's inequality, we have
\begin{align}
    &\abs{\e^{-1}\dbr{\varphi\big(\abs{v}^2-5\big)\mh,\ssc}}\ls\e^{-1}\abs{\dbr{\ssc,\int_0^{\mn}\p_{\mn}\varphi}}\\
    \ls&\abs{\dbr{\eta\ssc,\frac{1}{\mn}\int_0^{\mn}\p_{\mn}\varphi}}
    \ls\nm{\eta\big(\fb_1+\ssy+\ssz\big)}_{L^2_tL^2_xL^1_{\vv}}\tnnm{\frac{1}{\mn}\int_0^{\mn}\p_{\mn}\varphi}\no\\
    \ls&\nm{\eta\big(\fb_1+\ssy+\ssz\big)}_{L^2_tL^2_xL^1_{\vv}}\tnnm{\p_{\mn}\varphi}\ls\oot\tnnm{\cc}^{2}+\oot\e.\no
\end{align}
From Lemma \ref{s1-estimate=}, we directly bound
\begin{align}
    \abs{\dbbr{\nx\varphi\cdot\a,\ssa}}\ls\tnnm{\nx\varphi}\tnnm{\ssa}\ls\oot\tnm{\cc}^{2}+\oot\e^2.
\end{align}
Similar to \eqref{kernel-p00=}, based on Lemma \ref{s2-estimate=}, Remark \ref{rmk:s2=} and Hardy's inequality, we have
\begin{align}
    &\abs{\dbbr{\nx\varphi\cdot\a,\ssc}}\leq\abs{\dbr{\ssc,\nx\varphi\Big|_{\mn=0}}}+\abs{\e\dbr{\eta\ssc,\frac{1}{\mn}\int_0^{\mn}\p_{\mn}\nx\varphi}}\\
    \ls&\nm{\fb_1+\ssy+\ssz}_{L^2_tL^2_{\iota_1\iota_2}L^1_{\mn}L^1_{\vv}}\tnnms{\nx\varphi}{\ga}+\e\tnnm{\eta\big(\fb_1+\ssy+\ssz\big)}\tnnm{\frac{1}{\mn}\int_0^{\mn}\p_{\mn}\nx\varphi}\no\\
    \ls&\nm{\fb_1+\ssy+\ssz}_{L^2_tL^2_{\iota_1\iota_2}L^1_{\mn}L^1_{\vv}}\tnnms{\nx\varphi}{\ga}+\e\tnnm{\eta\big(\fb_1+\ssy+\ssz\big)}\tnnm{\p_{\mn}\nx\varphi}\ls \oot\e\tnnm{\cc}\ls\oot\tnnm{\cc}^{2}+\oot\e^{2}.\no
\end{align}
Based on Lemma \ref{s3-estimate=}, Lemma \ref{s4-estimate=}, and Lemma \ref{s5-estimate=}, we have
\begin{align}
    \abs{\dbbr{\nx\varphi\cdot\a,\ssd+\ssf+\ssg}}&\ls\tnnm{\nx\varphi}\tnnm{\ssd+\ssf+\ssg}\\
    &\ls\oot\tnnm{\cc}^{2}+\oot\tnnm{\re}^{2}+\oot\e
    \ls\oot\tnnm{\cc}^{2}+\oot\e\xnnm{\re}^{2}+\oot\e.\no
\end{align}
Finally, based on Lemma \ref{s6-estimate=}, we have
\begin{align}
    \abs{\dbbr{\nx\varphi\cdot\a,\ssh}}\ls\abs{\dbbr{\nx\varphi\cdot\a,\Gamma\Big[\bre,\bre\Big]}}+\abs{\dbbr{\nx\varphi\cdot\a,\Gamma\Big[\re,\ire\Big]}}.
\end{align}
The oddness and orthogonality, with the help of interpolation $\pnnm{\bb}{3}\ls\tnnm{\bb}^{\frac{2}{3}}\lnnmm{\bb}^{\frac{1}{3}}\ls\e^{\frac{1}{6}}\xnnm{\re}$, imply that
\begin{align}\label{kernel-c05=}
    \abs{\dbr{\nx\varphi\cdot\a,\Gamma\Big[\bre,\bre\Big]}}&\ls\abs{\dbr{\nx\varphi\cdot\a,\Gamma\left[\mh\left(\vv\cdot\bb\right),\mh\left(\frac{\abs{\vv}^2-5}{2}\cc\right)\right]}}\\
    &\ls\pnnm{\nx\varphi}{6}\pnnm{\bb}{3}\tnnm{\cc}\ls \nnm{\varphi}_{L^2_tH^2_x}\pnnm{\bb}{3}\tnnm{\cc}\ls\e^{\frac{1}{6}}\xnnm{\re}\tnnm{\cc}^{2}.\no
\end{align}
In addition, with the help of $\pnnm{\ire}{3}\ls\tnnm{\ire}^{\frac{2}{3}}\lnnmm{\ire}^{\frac{1}{3}}\ls\e^{\frac{1}{2}}\xnnm{\re}$, we have
\begin{align}\label{kernel-c06=}
    \abs{\dbbr{\nx\varphi\cdot\a,\Gamma\Big[\re,\ire\Big]}}\ls&\pnnm{\nx\varphi}{6}\tnnm{\re}\pnnm{\ire}{3}\\
    \ls&\e\tnnm{\cc}\xnnm{\re}^2\ls \oo\tnnm{\cc}^{2}+\e^2\xnnm{\re}^{4}.\no
\end{align}
Hence, we know
\begin{align}
    \abs{\dbbr{\nx\varphi\cdot\a,\ssh}}\ls\e^{\frac{1}{6}}\xnnm{\re}\tnnm{\cc}^{2}+\oo\tnnm{\cc}^{2}+\e^2\xnnm{\re}^{4}.
\end{align}
Collecting the above, we have
\begin{align}\label{kernel-c02=}
    &\abs{\e^{-1}\dbr{\varphi\big(\abs{v}^2-5\big)\mh,\ss}}+\abs{\dbbr{\nx\varphi\cdot\a,\ss}}\\
    \ls&\e^{\frac{1}{6}}\xnnm{\re}\tnnm{\cc}^{2}+\big(\oo+\oot\big)\tnnm{\cc}^{2}+\oot\e\xnnm{\re}^{2}+\e\xnnm{\re}^{4}+\oot\e.\no
\end{align}
Inserting \eqref{kernel-c02=} into \eqref{kernel-c01=}, we have
\begin{align}\label{kernel-c03=}
\\
    \tnm{\nx\varphi(t)}^2+\tnnm{\cc}^2\ls&\e^{\frac{1}{6}}\xnnm{\re}\tnnm{\cc}^{2}+\e\tnm{\re(t)}^2+\tnnms{\re}{\ga_+}^2+\tnnm{\P}^2+\tnnm{\ire}^2+\e^2\tnnm{\dt\nx\varphi}^2\no\\
    &+\oot\e\xnnm{\re}^{2}+\e\xnnm{\re}^{4}+\oot\e.\no
\end{align}

\paragraph{\underline{Estimate of $\tnnm{\dt\nx\varphi}$}}
Denote $\Phi=\dt\varphi$. Taking $\test=\e\Phi\big(\abs{\vv}^2-5\big)\mh$ in \eqref{weak formulation=}, due to orthogonality and $\Phi\big|_{\p\Omega}=0$, we obtain
\begin{align}
    \e^2\dbbr{\dt\re,\Phi\left(\abs{\vv}^2-5\right)\mh}-\e\dbbr{\re,\vv\cdot\nx\left(\Phi\left(\abs{\vv}^2-5\right)\mh\right)}
    &=\e\dbbr{\ss,\Phi\left(\abs{\vv}^2-5\right)\mh}.
\end{align}
Notice that 
\begin{align}
    \e^2\dbbr{\dt\re,\Phi\left(\abs{\vv}^2-5\right)\mh}=2\e^2\dbbr{\dt\big(5\cc-\P\big),\Phi}=-2\e^2\dbbr{\dx\Phi,\Phi}=2\e^2\tnnm{\dt\nx\varphi}^2.
\end{align}
Based on orthogonality, we have
\begin{align}
    \abs{\e\dbbr{\re,\vv\cdot\nx\left(\Phi\left(\abs{\vv}^2-5\right)\mh\right)}}=&\abs{\e\dbbr{\ire,\vv\cdot\nx\left(\Phi\left(\abs{\vv}^2-5\right)\mh\right)}}\\
    \ls& \tnnm{\ire}^2+\oo\e^{2}\tnnm{\dt\nx\varphi}^2.\no
\end{align}
Then, by a similar argument as the above estimates for $\e^{-1}\dbr{\varphi\big(\abs{v}^2-5\big)\mh,\ss}$, we have
\begin{align}
    \abs{\e\dbbr{\ss,\Phi\left(\abs{\vv}^2-5\right)\mh}}=\e\abs{\dbbr{\ssa+\ssc,\Phi\left(\abs{\vv}^2-5\right)\mh}}\ls \oo\e^2\tnnm{\dt\nx\varphi}^2+\oot\e.
\end{align}
Collecting the above, we have
\begin{align}\label{kernel-c04=}
    \e^2\tnnm{\dt\nx\varphi}^2\ls \tnm{\ire}^2+\oot\e.
\end{align}
Inserting \eqref{kernel-c04=} into \eqref{kernel-c03=}, we have
\begin{align}
    \tnm{\nx\varphi(t)}^2+\tnnm{\cc}^2\ls&\e^{\frac{1}{6}}\xnnm{\re}\tnnm{\cc}^{2}+\e\tnm{\re(t)}^2+\tnnms{\re}{\ga_+}^2+\tnnm{\P}^2+\tnnm{\ire}^2\\
    &+\oot\e\xnnm{\re}^{2}+\e\xnnm{\re}^{4}+\oot\e.\no
\end{align}
Hence, \eqref{eq:c-bound=} follows.
\end{proof}

\begin{proposition}\label{prop:c-bound=t}
    Under the assumption
    \eqref{assumption:evolutionary1}\eqref{assumption:evolutionary2}\eqref{assumption:evolutionary3}, we have
    \begin{align}\label{est:c-bound=t}
    \e^{-\frac{1}{2}}\tnnm{\dt\cc}\ls\e^{-\frac{1}{2}}\tnm{\dt\P}+\oot\xnnm{\re}+\xnnm{\re}^2+\oot.
    \end{align}
\end{proposition}
\begin{proof}
Applying the similar argument as in the proof of Proposition \ref{prop:c-bound=} to the equation \eqref{remainder==}, we obtain the desired result. Notice that in the bounds \eqref{kernel-c05=} and \eqref{kernel-c06=}, we should always assign $L^2$ norm to the time-derivative terms and $L^3$ to the no-derivative terms.
\end{proof}


\subsubsection{Estimate of $\bb$}

\begin{proposition}\label{prop:b-bound=}
    Under the assumption
    \eqref{assumption:evolutionary1}\eqref{assumption:evolutionary2}\eqref{assumption:evolutionary3}, we have
    \begin{align}\label{est:b-bound=}
    \e^{-\frac{1}{2}}\tnnm{\bb}
    \ls\oo\e^{-\frac{1}{2}}\tnnm{\P}+\oot\xnnm{\re}+\xnnm{\re}^{2}+\oot.
    \end{align}
\end{proposition}

\begin{proof}
It suffices to justify
\begin{align}\label{eq:b-bound=}
    \tnnm{\bb}
    \ls&\e^{\frac{1}{12}}\xnnm{\re}^{\frac{1}{2}}\tnnm{\bb}+\e^{\frac{1}{2}}\tnm{\re(t)}+\tnnms{\re}{\ga_+}+\oo\tnnm{\P}+\tnnm{\ire}\\
    &+\oot\e^{\frac{1}{2}}\xnnm{\re}+\e^{\frac{1}{2}}\xnnm{\re}^{2}+\oot\e^{\frac{1}{2}}.\no
\end{align}

\paragraph{\underline{Weak Formulation}}
Assume $(\psi,q)\in\r^3\times\r$ (where $q$ has zero average) is the unique strong solution to the Stokes problem
\begin{align}\label{f:b-test=}
\left\{
    \begin{array}{rll}
    -\beta\Delta_x\psi+\nx q&=\bb&\ \ \text{in}\ \ \Omega,\\\rule{0ex}{1.5em}
    \nx\cdot\psi&=0&\ \ \text{in}\ \ \Omega,\\\rule{0ex}{1.5em}
    \psi&=0&\ \ \text{on}\ \ \p\Omega.
    \end{array}
\right.
\end{align}
Based on the standard fluid estimates \cite{Cattabriga1961} and trace theorem, we have
\begin{align}
    \nm{\psi(t)}_{H^2}+\abs{\nx\psi(t)}_{L^2}+\nm{q(t)}_{H^1}+\abs{q(t)}_{L^2}\ls\tnm{\bb(t)}.
\end{align}
Multiplying $\bb$ on both sides of \eqref{f:b-test=} and integrating by parts for $\dbbrx{\nx q,\bb}$, we have
\begin{align}
    -\dbbrx{\lambda\Delta_x\psi,\bb}-\dbbrx{ q,\nx\cdot\bb}+\int_0^t\int_{\p\Omega}q(\bb\cdot n)&=\tnnm{\bb}^2,
\end{align}
which, by combining \eqref{conservation law 1=} and Remark \ref{rmk:boundary}, implies
\begin{align}\label{kernel-b00=}
    -\dbbrx{\lambda\Delta_x\psi,\bb}-\dbr{q\mh,\ss}+\dbbrb{q\mh,\re}{\ga_+}-\dbbrb{q\mh,\h}{\ga_-}&=\tnnm{\bb}^{2}.
\end{align}
Inserting \eqref{kernel-b00=} into \eqref{conservation law 7=} to replace $-\dbbrx{\lambda\Delta_x\psi,\bb}$, we obtain
\begin{align}
    &\dbbrx{\dt\bb,\psi}+\e\bbr{\re(t),\nx\psi(t):\b}-\e\bbr{\z,\nx\psi(0):\b}-\e\dbbr{\re,\dt\nx\psi:\b}+\tnnm{\bb}^{2}\\
    =&-\dbbrb{q\mh,\h}{\ga_-}+\dbbrb{q\mh,\re}{\ga_+}+\dbbrb{\nx\psi:\b,h}{\ga_-}-\dbbrb{\nx\psi:\b,\re}{\ga_+}\no\\
    &+\dbr{v\cdot\nx\Big(\nx\psi:\b\Big),\ire}-\dbr{q\mh,\ss}+\e^{-1}\dbr{\psi\cdot v\mh,\ss}+\dbbr{\nx\psi:\b,\ss}.\no
\end{align}
Using the divergence-free of $\psi$ and $\psi\big|_{\p\Omega}=0$, we have
\begin{align}
    &\dbbrx{\dt\bb,\psi}=\dbbrx{-\lambda\dt\dx\psi+\dt\nx q,\psi}=\lambda\dbbrx{\dt\nx\psi,\nx\psi}
    =\frac{\lambda}{2}\tnm{\nx\psi(t)}^2-\frac{\lambda}{2}\tnm{\nx\psi(0)}^2,
\end{align}
and from Lemma \ref{z-estimate=} 
\begin{align}
    \frac{\lambda}{2}\tnm{\nx\psi(0)}^2\ls\tnm{\bb(0)}^2\ls\tnm{\z}^2\ls\oot\e^2.
\end{align}
Also, based on Lemma \ref{z-estimate=}, we know
\begin{align}
    \abs{\e\bbr{\re(t),\nx\psi(t):\b}}\ls&\e\tnm{\re(t)}\tnm{\nx\psi(t)}\ls\e\tnm{\re(t)}\tnm{\bb(t)}\ls \e\tnm{\re(t)}^2,\\
    \abs{\e\bbr{\z,\nx\psi(0):\b}}\ls&\e\tnm{\z}\tnm{\nx\psi(0)}\ls \e\tnm{\z}^2\ls\oot\e^3,
\end{align}
and
\begin{align}
    \abs{\e\dbbr{\re,\dt\nx\psi:\b}}=\abs{\e\dbbr{\ire,\dt\nx\psi:\b}}\ls&\tnnm{\ire}^2+\oo\e^2\tnnm{\dt\nx\psi}^2.
\end{align}
In addition, based on Lemma \ref{h-estimate=}, we have
\begin{align}
    \abs{\dbbrb{q\mh,\h}{\ga_-}}+\abs{\dbbrb{\nx\psi:\b,\h}{\ga_-}}\ls&\oo\tnnms{q\mh}{\ga_-}^2+\oo\tnnms{\nx\psi:\b}{\ga_-}^2+\tnnms{\h}{\ga_-}^2\\
    \ls& \oo\tnnm{\bb}^2+\oot\e^2,\no\\
    \abs{\dbbrb{q\mh,\re}{\ga_+}}+\abs{\dbbrb{\nx\psi:\b,\re}{\ga_+}}\ls&\oo\tnnms{q\mh}{\ga_+}^2+\oo\tnnms{\nx\psi:\b}{\ga_+}^2+\tnnms{\re}{\ga_+}^2\\
    \ls& \oo\tnnm{\bb}^2+\tnnms{\re}{\ga_+}^2,\no
\end{align}
and
\begin{align}
    \abs{\dbr{v\cdot\nx\Big(\nx\psi:\b\Big),\ire}}\ls& \oo\tnnm{\nx^2\psi}^2+\tnnm{\ire}^2\\
    \ls&\oo\tnnm{\bb}^2+\tnnm{\ire}^2.\no
\end{align}
Collecting the above, we have
\begin{align}\label{kernel-b01=}
    \tnm{\nx\psi(t)}^2+\tnnm{\bb}^2
    \ls&\e\tnm{\re(t)}^2+\tnnms{\re}{\ga_+}^2+\tnnm{\ire}^2+\oo\e^2\tnnm{\dt\nx\psi}^2+\oot\e^2\\
    &+\abs{\dbr{q\mh,\ss}}+\abs{\e^{-1}\dbr{\psi\cdot v\mh,\ss}}+\abs{\dbbr{\nx\psi:\b,\ss}}.\no
\end{align}

\paragraph{\underline{Source Term Estimates}}
Due to orthogonality of $\Gamma$, we have
\begin{align}
    \abs{\dbr{q\mh,\ss}}+\abs{\e^{-1}\dbr{\psi\cdot v\mh,\ss}}=\abs{\dbr{q\mh,\ssa+\ssc}}+\abs{\e^{-1}\dbr{\psi\cdot v\mh,\ssa+\ssc}}.
\end{align}
Using Lemma \ref{s1-estimate=}, we have
\begin{align}
    \abs{\dbr{q\mh,\ssa}}+\abs{\e^{-1}\dbr{\psi\cdot v\mh,\ssa}}=&\e^2\abs{\dbr{q\mh,\dt\f_2}}+\e\abs{\dbr{\psi\cdot v\mh,\dt\f_2}}\\
    \ls&\e\Big(\tnnm{q}+\tnnm{\nx\psi}\Big)\tnnm{\dt\f_2}\ls \oot\tnnm{\bb}^2+\oot\e^2.\no
\end{align}
Using Lemma \ref{s2-estimate=} and Remark \ref{rmk:s2=}, integrating by parts in $\va$ for $\ssx$, we obtain
\begin{align}
    \abs{\dbr{q\mh,\ssc}}\ls&\tnnm{q}\nnm{\fb_1+\ssy+\ssz}_{L^2_tL^2_xL^1_v}\ls \oot\e^{\frac{1}{2}}\tnnm{q}\ls\oot\tnnm{\bb}^{2}+\oot\e.
\end{align}
Similar to \eqref{kernel-p00=}, we have
\begin{align}
    &\abs{\e^{-1}\dbr{\psi\cdot v\mh,\ssc}}\ls\e^{-1}\abs{\dbr{\ssc,\int_0^{\mn}\p_{\mn}\psi}}\\
    \ls&\abs{\dbr{\eta\ssc,\frac{1}{\mn}\int_0^{\mn}\p_{\mn}\psi}}
    \ls\nnm{\eta\big(\fb_1+\ssy+\ssz\big)}_{L^2_tL^2_xL^1_{\vv}}\tnnm{\frac{1}{\mn}\int_0^{\mn}\p_{\mn}\psi}\no\\
    \ls&\nnm{\eta\big(\fb_1+\ssy+\ssz\big)}_{L^2_tL^2_xL^1_{\vv}}\tnnm{\p_{\mn}\psi}\ls\oot\tnnm{\bb}^{2}+\oot\e.\no
\end{align}
From Lemma \ref{s1-estimate=}, we directly bound
\begin{align}
    \abs{\dbbr{\nx\psi:\b,\ssa}}\ls\tnnm{\nx\psi}\tnnm{\ssa}\ls\oot\tnnm{\bb}^{2}+\oot\e^{2}.
\end{align}
Similar to \eqref{kernel-p00=}, based on Lemma \ref{s2-estimate=}, Remark \ref{rmk:s2=} and Hardy's inequality, we have
\begin{align}
    &\abs{\dbbr{\nx\psi:\b,\ssc}}\leq\abs{\dbr{\ssc,\nx\psi\Big|_{\mn=0}}}+\abs{\e\dbr{\eta\ssc,\frac{1}{\mn}\int_0^{\mn}\p_{\mn}\nx\psi}}\\
    \ls&\nnm{\fb_1+\ssy+\ssz}_{L^2_tL^2_{\iota_1\iota_2}L^1_{\mn}L^1_{\vv}}\tnnms{\nx\psi}{\ga}+\e\tnnm{\eta\big(\fb_1+\ssy+\ssz\big)}\tnnm{\frac{1}{\mn}\int_0^{\mn}\p_{\mn}\nx\psi}\no\\
    \ls&\nnm{\fb_1+\ssy+\ssz}_{L^2_tL^2_{\iota_1\iota_2}L^1_{\mn}L^1_{\vv}}\tnnms{\nx\psi}{\ga}+\e\tnnm{\eta\big(\fb_1+\ssy+\ssz\big)}\tnnm{\p_{\mn}\nx\psi}\ls \oot\e\tnnm{\bb}\ls\oot\tnnm{\bb}^{2}+\oot\e^{2}.\no
\end{align}
Based on Lemma \ref{s3-estimate=}, Lemma \ref{s4-estimate=}, and Lemma \ref{s5-estimate=}, we have
\begin{align}
    \abs{\dbbr{\nx\psi:\b,\ssd+\ssf+\ssg}}\ls&\tnnm{\nx\psi}\tnnm{\ssd+\ssf+\ssg}\\
    \ls&\oot\tnnm{\bb}^{2}+\oot\tnnm{\re}^{2}+\oot\e
    \ls\oot\tnnm{\bb}^{2}+\oot\e\xnnm{\re}^{2}+\oot\e.\no
\end{align}
Finally, based on Lemma \ref{s6-estimate=}, we have
\begin{align}
    \abs{\dbbr{\nx\psi:\b,\ssh}}\ls&\abs{\dbbr{\nx\psi:\b,\Gamma\Big[\bre,\bre\Big]}}+\abs{\dbbr{\nx\psi:\b,\Gamma\Big[\re,\ire\Big]}}.
\end{align}
The oddness and orthogonality implies that
\begin{align}
    &\abs{\dbr{\nx\psi:\b,\Gamma\Big[\bre,\bre\Big]}}\\
    \ls&\abs{\dbr{\nx\psi:\b,\Gamma\left[\mh\left(\vv\cdot\bb\right),\mh\left(\vv\cdot\bb\right)\right]}}+\abs{\dbr{\nx\psi:\b,\Gamma\left[\mh\left(\frac{\abs{\vv}^2-5}{2}\cc\right),\mh\left(\frac{\abs{\vv}^2-5}{2}\cc\right)\right]}}.\no
\end{align}
Similar to \eqref{kernel-c05=}, we may directly bound
\begin{align}\label{kernel-b05=}
    \abs{\dbr{\nx\psi:\b,\Gamma\left[\mh\left(\vv\cdot\bb\right),\mh\left(\vv\cdot\bb\right)\right]}}\ls&\pnnm{\nx\psi}{6}\pnnm{\bb}{3}\tnnm{\bb}\\
    \ls& \nnm{\psi}_{L^2_tH^2_x}\pnnm{\bb}{3}\tnnm{\bb}\ls\e^{\frac{1}{6}}\xnnm{\re}\tnnm{\bb}^{2}.\no
\end{align}
Due to oddness and $\b_{ii}=\li\left[\left(\abs{v_i}^2-\dfrac{1}{3}\abs{\vv}^2\right)\mh\right]$, noting that $\Gamma\left[\mh\left(\dfrac{\abs{\vv}^2-5}{2}\cc\right),\mh\left(\dfrac{\abs{\vv}^2-5}{2}\cc\right)\right]$ only depends on $\abs{\vv}^2$, we have
\begin{align}
    &\abs{\dbr{\nx\psi:\b,\Gamma\left[\mh\left(\frac{\abs{\vv}^2-5}{2}\cc\right),\mh\left(\frac{\abs{\vv}^2-5}{2}\cc\right)\right]}}\\
    =&\abs{\dbr{\p_1\psi_1\b_{11}+\p_2\psi_2\b_{22}+\p_3\psi_3\b_{33},\Gamma\left[\mh\left(\frac{\abs{\vv}^2-5}{2}\cc\right),\mh\left(\frac{\abs{\vv}^2-5}{2}\cc\right)\right]}}\no\\
    =&\abs{\dbr{\big(\nx\cdot\psi\big)\b_{ii},\Gamma\left[\mh\left(\frac{\abs{\vv}^2-5}{2}\cc\right),\mh\left(\frac{\abs{\vv}^2-5}{2}\cc\right)\right]}}=0.\no
\end{align}
In addition, similar to \eqref{kernel-c06=}, we have
\begin{align}\label{kernel-b06=}
    \abs{\dbbr{\nx\psi:\b,\Gamma\Big[\re,\ire\Big]}}\ls&\pnnm{\nx\psi}{6}\tnnm{\re}\pnnm{\ire}{3}\\
    \ls&\e\tnnm{\bb}\xnnm{\re}^2\ls \oo\tnnm{\bb}^{2}+\e^{2}\xnnm{\re}^{4}.\no
\end{align}
Hence, we know
\begin{align}
    \abs{\dbbr{\nx\psi:\b,\ssh}}\ls\e^{\frac{1}{6}}\xnnm{\re}\tnnm{\bb}^{2}+\oo\tnnm{\bb}^{2}+\e^2\xnnm{\re}^{4}.
\end{align}
Collecting the above, we have
\begin{align}\label{kernel-b02=}
    &\abs{\dbr{q\mh,\ss}}+\abs{\e^{-1}\dbr{\psi\cdot v\mh,\ss}}+\abs{\dbbr{\nx\psi:\b,\ss}}\\
    \ls&\e^{\frac{1}{6}}\xnnm{\re}\tnnm{\bb}^{2}+\big(\oo+\oot\big)\tnnm{\bb}^{2}+\oot\e\xnnm{\re}^{2}+\e\xnnm{\re}^{4}+\oot\e.\no
\end{align}
Inserting \eqref{kernel-b02=} into \eqref{kernel-b01=}, we have
\begin{align}\label{kernel-b03=}
\\
    \tnm{\nx\psi(t)}^2+\tnnm{\bb}^2
    \ls&\e^{\frac{1}{6}}\xnnm{\re}\tnnm{\bb}^{2}+\e\tnm{\re(t)}^2+\tnnms{\re}{\ga_+}^2+\tnnm{\ire}^2+\oo\e^2\tnnm{\dt\nx\psi}^2\no\\
    &+\oot\e\xnnm{\re}^{2}+\e\xnnm{\re}^{4}+\oot\e.\no
\end{align}

\paragraph{\underline{Estimate of $\tnnm{\dt\nx\psi}$}}
Denote $\Psi=\dt\psi$. Taking $\test=\e\Psi\cdot v\mh$ in \eqref{weak formulation=}, due to orthogonality and $\Psi\big|_{\p\Omega}=0$, we obtain
\begin{align}
    \e^2\dbbr{\dt\re,\Psi\cdot v\mh}-\e\dbbr{\re,\vv\cdot\nx\left(\Psi\cdot\vv\mh\right)}
    &=\e\dbbr{\ss,\Psi\cdot\vv\mh}.
\end{align}
Notice that due to the divergence-free of $\Psi$ and $\Psi\big|_{\p\Omega}=0$
\begin{align}
    \e^2\dbbr{\dt\re,\Psi\cdot v\mh}=&\e^2\dbbr{\dt\bb,\Psi}=\e^2\dbbr{-\lambda\dx\Psi+\dt\nx q,\Psi}\\
    =&\e^2\dbbr{-\lambda\dx\Psi,\Psi}=\lambda\e^2\tnm{\dt\nx\psi}^2.\no
\end{align}
Also, using orthogonality, we have
\begin{align}
    \abs{\e\dbbr{\re,\vv\cdot\nx\left(\Psi\cdot\vv\mh\right)}}=&\e\abs{\dbr{\mh\P+\mh\left(\frac{\abs{\vv}^2-5}{2}\right)\cc,\vv\cdot\nx\left(\Psi\cdot\vv\mh\right)}}\\=&\e\abs{\dbr{\mh\P,\vv\cdot\nx\left(\Psi\cdot\vv\mh\right)}}
    \ls\tnnm{\P}^2+\oo\e^2\tnnm{\dt\nx\psi}^2,\no
\end{align}
Then, by a similar argument as the above estimates for $\e^{-1}\dbr{\psi\cdot v\mh,\ss}$, we have
\begin{align}
    \abs{\e\dbbr{\ss,\Psi\cdot\vv\mh}}=\abs{\e\dbbr{\ssa+\ssc,\Psi\cdot\vv\mh}}\ls \oo\e^2\tnnm{\dt\nx\varphi}^2+\oot\e.
\end{align}
Collecting the above, we have
\begin{align}\label{kernel-b04=}
    \e^2\tnnm{\dt\nx\varphi}^2\ls \tnnm{\P}^2+\oot\e.
\end{align}
Inserting \eqref{kernel-b04=} into \eqref{kernel-b03=}, we have
\begin{align}
    \tnm{\nx\psi(t)}^2+\tnnm{\bb}^2
    \ls&\e^{\frac{1}{6}}\xnnm{\re}\tnnm{\bb}^{2}+\e\tnm{\re(t)}^2+\tnnms{\re}{\ga_+}^2+\oo\tnnm{\P}^2+\tnnm{\ire}^2\\
    &+\oot\e\xnnm{\re}^{2}+\e\xnnm{\re}^{4}+\oot\e.\no
\end{align}
Hence, \eqref{eq:b-bound=} follows.
\end{proof}

\begin{proposition}\label{prop:b-bound=t}
    Under the assumption
    \eqref{assumption:evolutionary1}\eqref{assumption:evolutionary2}\eqref{assumption:evolutionary3}, we have
    \begin{align}\label{eq:b-bound=t}
        \e^{-\frac{1}{2}}\tnnm{\dt\bb}
        \ls\oo\e^{-\frac{1}{2}}\tnnm{\dt\P}+\oot\xnnm{\re}+\xnnm{\re}^{2}+\oot.
    \end{align}
\end{proposition}
\begin{proof}
Applying the similar argument as in the proof of Proposition \ref{prop:b-bound=} to the equation \eqref{remainder==}, we obtain the desired result. Notice that in the bounds \eqref{kernel-b05=} and \eqref{kernel-b06=}, we should always assign $L^2$ norm to the time-derivative terms and $L^3$ to the no-derivative terms.
\end{proof}


\subsubsection{Summary of Kernel Estimates}

\begin{proposition}\label{prop:kernel=}
    Under the assumption
    \eqref{assumption:evolutionary1}\eqref{assumption:evolutionary2}\eqref{assumption:evolutionary3}, we have
    \begin{align}
        \e^{-\frac{1}{2}}\tnnm{\bre}\ls\oot\xnnm{\re}+\xnnm{\re}^{2}+\oot.
    \end{align}
\end{proposition}
\begin{proof}
    Summarizing Proposition \ref{prop:p-bound=}, Proposition \ref{prop:c-bound=} and Proposition \ref{prop:b-bound=} leads to the desired result.
\end{proof}

\begin{proposition}\label{prop:kernel=t}
    Under the assumption
    \eqref{assumption:evolutionary1}\eqref{assumption:evolutionary2}\eqref{assumption:evolutionary3}, we have
    \begin{align}
        \e^{-\frac{1}{2}}\tnnm{\pk[\dt\re]}\ls\oot\xnnm{\re}+\xnnm{\re}^{2}+\oot.
    \end{align}
\end{proposition}
\begin{proof}
    Summarizing Proposition \ref{prop:p-bound=t}, Proposition \ref{prop:c-bound=t} and Proposition \ref{prop:b-bound=t} leads to the desired result.
\end{proof}


\subsection{Energy Estimate -- Instantaneous}

\begin{proposition}\label{prop:energy=s}
    Under the assumption
    \eqref{assumption:evolutionary1}\eqref{assumption:evolutionary2}\eqref{assumption:evolutionary3}, we have 
    \begin{align}\label{est:energy=s}
    \e^{-\frac{1}{2}}\tnms{\re(t)}{\gamma_+}+\e^{-1}\tnm{\ire(t)}\ls\oot\xnnm{\re}+\xnnm{\re}^2+\oot.
    \end{align}
\end{proposition}

\begin{proof}
    For fixed $t\in\rp$, we apply the similar argument as the proof of Proposition \ref{prop:energy} to \eqref{remainder=s}, and obtain
    \begin{align}
        \e^{-1}\tnms{\re(t)}{\gamma_+}^2+\e^{-2}\tnm{\ire(t)}^2\ls\oot\xnnm{\re}^2+\xnnm{\re}^4+\oot+\abs{\e^{-1}\bbr{\e\dt\re(t),\re(t)}}.
    \end{align}
    Using Proposition \ref{prop:energy=} and Proposition \ref{prop:energy=t}, we have
    \begin{align}
        \abs{\e^{-1}\bbr{\e\dt\re(t),\re(t)}}=& \abs{\bbr{\dt\re(t),\re(t)}}\ls\tnm{\re(t)}^2+\tnm{\dt\re(t)}^2
        \ls\oot\xnnm{\re}^2+\xnnm{\re}^4+\oot.
    \end{align}
\end{proof}

\begin{corollary}\label{lem:energy-extension=s} 
    Under the assumption
    \eqref{assumption:evolutionary1}\eqref{assumption:evolutionary2}\eqref{assumption:evolutionary3}, we have 
    \begin{align}\label{energy-extension=s}
        \pnm{\ire(t)}{6}+\pnms{\m^{\frac{1}{4}}\re(t)}{4}{\gamma_+}
        \ls\oot\xnnm{\re}+\xnnm{\re}^2+\oot.
    \end{align}
\end{corollary}
\begin{proof}
    This is similar to the proof of Corollary \ref{lem:energy-extension}.
\end{proof}


\subsection{Kernel Estimate -- Instantaneous}

\begin{proposition}\label{prop:p-bound=s}
    Under the assumption
    \eqref{assumption:evolutionary1}\eqref{assumption:evolutionary2}\eqref{assumption:evolutionary3}, we have
    \begin{align}\label{p-bound=s}
    \pnm{\P(t)}{6}\ls\oot\xnnm{\re}+\xnnm{\re}^2+\oot.
    \end{align}
\end{proposition}
\begin{proof}
    For fixed $t\in\rp$, we apply the similar argument as the proof of Proposition \ref{prop:p-bound} with $\N=6$ to \eqref{remainder=s}. We obtain for $\psi$ defined in \eqref{f:p-test}
    \begin{align}
        \pnm{\P(t)}{6}^{6}\ls\oot\xnnm{\re}^6+\xnnm{\re}^{12}+\oot+\abs{\bbr{\e\dt\re(t),\psi}}.
    \end{align}
    Using Proposition \ref{prop:energy=t}, we have 
    \begin{align}
        \abs{\bbr{\e\dt\re(t),\psi}}&\ls \e\tnm{\dt\re(t)}\tnm{\psi}\ls \e\tnm{\dt\re(t)}\nm{\psi}_{W^{1,\frac{6}{5}}}\\
        &\ls\e\Big(\oot\xnnm{\re}+\xnnm{\re}^2+\oot\Big)\pnm{\P(t)}{6}^5\ls \e^6\pnm{\P(t)}{6}^6+\oot\xnnm{\re}^6+\xnnm{\re}^{12}+\oot.\no
    \end{align}
    Hence, we have
    \begin{align}
        \pnm{\P(t)}{6}^{6}\ls\oot\xnnm{\re}^6+\xnnm{\re}^{12}+\oot,
    \end{align}
    and thus \eqref{p-bound=s} follows.
\end{proof}

\begin{proposition}\label{prop:c-bound=s}
    Under the assumption
    \eqref{assumption:evolutionary1}\eqref{assumption:evolutionary2}\eqref{assumption:evolutionary3}, we have
    \begin{align}\label{c-bound=s}
    \pnm{\cc(t)}{6}\ls\oot\xnnm{\re}+\xnnm{\re}^{2}+\oot.
    \end{align}
\end{proposition}
\begin{proof}
    For fixed $t\in\rp$, we apply the similar argument as the proof of Proposition \ref{prop:c-bound} with $\N=6$ to \eqref{remainder=s}. We obtain for $\varphi$ defined in \eqref{f:c-test}
    \begin{align}
        \pnm{\cc(t)}{6}^{6}\ls\oot\xnnm{\re}^6+\xnnm{\re}^{12}+\oot+\abs{\e^{-1}\br{\varphi\big(\abs{v}^2-5\big)\mh,\e\dt\re(t)}}+\abs{\bbr{\nx\varphi\cdot\a,\e\dt\re(t)}}.
    \end{align}
    Using Proposition \ref{prop:energy=t}, we have 
    \begin{align}
        &\abs{\e^{-1}\br{\varphi\big(\abs{v}^2-5\big)\mh,\e\dt\re(t)}}+\abs{\bbr{\nx\varphi\cdot\a,\e\dt\re(t)}}\\
        \ls& \tnm{\dt\re(t)}\nm{\varphi}_{H^1}\ls \tnm{\dt\re(t)}\nm{\varphi}_{W^{2,\frac{6}{5}}}
        \ls\Big(\oot\xnnm{\re}+\xnnm{\re}^2+\oot\Big)\pnm{\cc(t)}{6}^5\no\\
        \ls& \oo\pnm{\cc(t)}{6}^6+\oot\xnnm{\re}^6+\xnnm{\re}^{12}+\oot.\no
    \end{align}
    Hence, we have
    \begin{align}
        \pnm{\cc(t)}{6}^{6}\ls\oot\xnnm{\re}^6+\xnnm{\re}^{12}+\oot,
    \end{align}
    and thus \eqref{c-bound=s} follows.
\end{proof}

\begin{proposition}\label{prop:b-bound=s}
    Under the assumption
    \eqref{assumption:evolutionary1}\eqref{assumption:evolutionary2}\eqref{assumption:evolutionary3}, we have
    \begin{align}\label{b-bound=s}
    \pnm{\bb(t)}{6}
    \ls\oot\xnnm{\re}+\xnnm{\re}^{2}+\oot.
    \end{align}
\end{proposition}
\begin{proof}
    For fixed $t\in\rp$, we apply the similar argument as the proof of Proposition \ref{prop:b-bound} with $\N=6$ to \eqref{remainder=s}. We obtain for $\psi$ and $q$ defined in \eqref{f:b-test}
    \begin{align}
        \pnm{\bb(t)}{6}^{6}
        \ls&\oot\xnnm{\re}^6+\xnnm{\re}^{12}+\oot\\
        &+\abs{\br{q\mh,\e\dt\re(t)}}+\abs{\e^{-1}\br{\psi\cdot v\mh,\e\dt\re(t)}}+\abs{\bbr{\nx\psi:\b,\e\dt\re(t)}}.\no
    \end{align}
    Using Proposition \ref{prop:energy=t}, we have 
    \begin{align}
        &\abs{\br{q\mh,\e\dt\re(t)}}+\abs{\e^{-1}\br{\psi\cdot v\mh,\e\dt\re(t)}}+\abs{\bbr{\nx\psi:\b,\e\dt\re(t)}}\\
        \ls& \tnm{\dt\re(t)}\Big(\nm{\psi}_{H^1}+\tnm{q}\Big)\ls \e\tnm{\dt\re(t)}\Big(\nm{\psi}_{W^{2,\frac{6}{5}}}+\nm{q}_{W^{1,\frac{6}{5}}}\Big)\no\\
        \ls&\Big(\oot\xnnm{\re}+\xnnm{\re}^2+\oot\Big)\pnm{\bb(t)}{6}^5\ls \oo\pnm{\bb(t)}{6}^6+\oot\xnnm{\re}^6+\xnnm{\re}^{12}+\oot.\no
    \end{align}
    Hence, we have
    \begin{align}
        \pnm{\bb(t)}{6}^{6}\ls&\oot\xnnm{\re}^6+\xnnm{\re}^{12}+\oot,
    \end{align}
    and thus \eqref{b-bound=s} follows.
\end{proof}

\begin{proposition}\label{prop:kernel=s}
    Under the assumption
    \eqref{assumption:evolutionary1}\eqref{assumption:evolutionary2}\eqref{assumption:evolutionary3}, we have
    \begin{align}
        \pnm{\bre(t)}{6}\ls\oot\xnnm{\re}+\xnnm{\re}^{2}+\oot.
    \end{align}
\end{proposition}
\begin{proof}
    Summarizing Proposition \ref{prop:p-bound=s}, Proposition \ref{prop:c-bound=s} and Proposition \ref{prop:b-bound=s} leads to the desired result.
\end{proof}


\subsection{\texorpdfstring{$L^{\infty}$}{} Estimate}

We define a weight functions as \eqref{ltt 11} and \eqref{ltt 12}.

\begin{proposition}\label{prop:infty=}
    Under the assumption
    \eqref{assumption:evolutionary1}\eqref{assumption:evolutionary2}\eqref{assumption:evolutionary3}, we have
    \begin{align}
    \e^{\frac{1}{2}}\lnnmm{\re}+\e^{\frac{1}{2}}\lnnmms{\re}{\ga_+}
    \ls&\oot\xnnm{\re}+\xnnm{\re}^{2}+\oot.
    \end{align}
\end{proposition}
\begin{proof}
We will use the well-known $L^2-L^6-L^{\infty}$ framework.\\

\paragraph{\underline{Step 1: Mild Formulation}}
Denote the weighted solution
\begin{align}
\reg(t,\vx,\vv):=&\vh(\vv)\re(t,\vx,\vv),
\end{align}
and the weighted non-local operator
\begin{align}
K_{\vh(\vv)}[\reg](\vv):=&\vh(\vv)K\left[\frac{\reg}{\vh}\right](\vv)=\int_{\r^3}k_{\vh(\vv)}(\vv,\vuu)\reg(\vuu)\ud{\vuu},
\end{align}
where
\begin{align}
k_{\vh(\vv)}(\vv,\vuu):=k(\vv,\vuu)\frac{\vh(\vv)}{\vh(\vuu)}.
\end{align}
Multiplying $\e\vh$ on both sides of \eqref{remainder}, we have
\begin{align}\label{ltt 00=}
\left\{
\begin{array}{l}
\e^2\dt\reg+\e\vv\cdot\nx\reg+\nu\reg=K_{\vh}[\reg](\vx,\vv)+\e\vh(\vv) \ss(t,\vx,\vv)\ \ \text{in}\ \ \rp\times\Omega\times\r^3,\\\rule{0ex}{2em}
\reg(0,\vx,\vv)=\vh\z(\vx,\vv)\ \ \text{in}\ \ \Omega\times\r^3,\\\rule{0ex}{2em}
\reg(t,\vx_0,\vv)=\vh\h(t,\vx_0,\vv)\ \ \text{for}\ \ \vx_0\in\p\Omega\ \
\text{and}\ \ \vv\cdot\vn<0,
\end{array}
\right.
\end{align}
We can rewrite the solution of the equation \eqref{ltt 00=} along the characteristics by Duhamel's principle as
\begin{align}
\reg(t,\vx,\vv)=& \id_{t_i< t_b}\vh(\vv)\z(\xm,\vv)\ue^{-\nu(\vv)
t_i}+\id_{t_b< t_i}\vh(\vv)\h(\xm,\vv)\ue^{-\nu(\vv)
t_b}\\
&+\int_{0}^{\tm}\vh(\vv)\e\ss\Big(t-\e^2s,\vx-\e(\tm-s)\vv,\vv\Big)\ue^{-\nu(\vv)
(\tm-s)}\ud{s}\no\\
&+\int_{0}^{\tm}\int_{\R^3}k_{\vh(\vv)}(\vv,\vuu)\reg\Big(t-\e^2s, \vx-\e(\tm-s)\vuu,\vuu\Big)\ue^{-\nu(\vv)
(\tm-s)}\ud\vuu\ud{s},\no
\end{align}
where
\begin{align}
t_b(\vx,\vv):=&\inf\big\{t>0:\vx-\e t\vv\notin\Omega\big\},\quad t_i(t)=\e^{-2}t,\quad \tm=\min\big\{t_i, t_b\big\},
\end{align}
and
\begin{align}
\xm(\vx,\vv):=\vx-\e \tm(\vx,\vv)\vv.
\end{align}
We further rewrite the non-local term along the characteristics as
\begin{align}\label{ltt 01=}
&\reg(t,\vx,\vv)\\
=& \id_{t_i< t_b}\vh(\vv)\z(\xm,\vv)\ue^{-\nu(\vv)
t_i}+\id_{t_b< t_i}\vh(\vv)\h(\xm,\vv)\ue^{-\nu(\vv)
t_b}\no\\
&+\int_{0}^{\tm}\vh(\vv)\e\ss\Big(t-\e^2s,\vx-\e(\tm-s)\vv,\vv\Big)\ue^{-\nu(\vv)
(\tm-s)}\ud{s}\no\\
&+\id_{t_i'< t_b'}\int_{0}^{\tm}\int_{\R^3}k_{\vh(\vv)}(\vv,\vuu)\vh(\vuu)\z(\xm',\vv)\ue^{-\nu(\vuu)t_i'}\ue^{-\nu(\vv)
(\tm-s)}\ud\vuu\ud{s}\no\\
&+\id_{t_b'< t_i'}\int_{0}^{\tm}\int_{\R^3}k_{\vh(\vv)}(\vv,\vuu)\vh(\vuu)\h(\xm',\vv)\ue^{-\nu(\vuu)t_b'}\ue^{-\nu(\vv)
(\tm-s)}\ud\vuu\ud{s}\no\\
&+\int_{0}^{\tm}\int_{\R^3}k_{\vh(\vv)}(\vv,\vuu)\int_0^{\tm'}\e\ss\Big(t-\e^2s-\e^2r,\vx-\e(\tm-s)\vuu-\e(\tm'-r)\vuu,\vuu\Big)\ue^{-\nu(\vuu)(\tm'-r)}\ue^{-\nu(\vv)
(\tm-s)}\ud r\ud\vuu\ud{s}\no\\
&+\int_{0}^{\tm}\int_{\R^3}k_{\vh(\vv)}(\vv,\vuu)\int_0^{\tm'}\int_{\R^3}k_{\vh(\vuu)}(\vuu,\vuu')\reg\Big(t-\e^2s-\e^2r,\vx-\e(\tm-s)\vuu-\e(\tm'-r)\vuu',\vuu'\Big)\ue^{-\nu(\vuu)(\tm'-r)}\ue^{-\nu(\vv)
(\tm-s)}\ud\vuu'\ud r\ud\vuu\ud{s},\no
\end{align}
where
\begin{align}
t_b'(\vx,\vv;s,\vuu):=&\inf\big\{t>0:\vx-\e(\tm-s)-\e t\vuu\notin\Omega\big\},\quad t_i'(t;s)=\e^{-2}t-s,\quad\tm'=\min\big\{t_i',t_b'\big\},
\end{align}
and
\begin{align}
\xm'(\vx,\vv;s,\vuu):=\vx-\e(\tm-s)-\e \tm'(\vx,\vv;s,\vuu)\vuu.
\end{align}

\paragraph{\underline{Step 2: Estimates of Source Terms and Boundary Terms}}
Based on Lemma \ref{z-estimate=} -- Lemma \ref{s6-estimate=}, we have
\begin{align}
    &\abs{\id_{t_i< t_b}\vh(\vv)\z(\xm,\vv)\ue^{-\nu(\vv)t_i}}+\abs{\id_{t_i'< t_b'}\int_{0}^{\tm}\int_{\R^3}k_{\vh(\vv)}(\vv,\vuu)\vh(\vuu)\z(\xm',\vv)\ue^{-\nu(\vuu)t_i'}\ue^{-\nu(\vv)(\tm-s)}\ud\vuu\ud{s}}\\
    \ls&\lnmm{\z}\ls\oot\e,\no
\end{align}
\begin{align}
    &\abs{\id_{t_b< t_i}\vh(\vv)\h(\xm,\vv)\ue^{-\nu(\vv)t_b}}+\abs{\id_{t_b'< t_i'}\int_{0}^{\tm}\int_{\R^3}k_{\vh(\vv)}(\vv,\vuu)\vh(\vuu)\h(\xm',\vv)\ue^{-\nu(\vuu)t_b'}\ue^{-\nu(\vv)(\tm-s)}\ud\vuu\ud{s}}\\
    \ls&\lnnmms{\h}{\ga_-}\ls\oot,\no
\end{align}
and
\begin{align}
    &\abs{\int_{0}^{\tm}\vh(\vv)\e\ss\Big(t-\e^2s,\vx-\e(\tm-s)\vv,\vv\Big)\ue^{-\nu(\vv)(\tm-s)}\ud{s}}\\
    &+\abs{\int_{0}^{\tm}\int_{\R^3}k_{\vh(\vv)}(\vv,\vuu)\int_0^{\tm'}\e\ss\Big(t-\e^2s-\e^2r,\vx-\e(\tm-s)\vuu-\e(\tm'-r)\vuu,\vuu\Big)\ue^{-\nu(\vuu)(\tm'-r)}\ue^{-\nu(\vv)(\tm-s)}\ud r\ud\vuu\ud{s}}\no\\
    \ls&\e\lnnmm{\nu^{-1}\ss}
    \ls\oot+\oot\e\lnmm{\re}+\e\lnmm{\re}^2\ls\oot\e^{\frac{1}{2}}\xnm{\re}+\xnm{\re}^2+\oot.\no
\end{align}

\paragraph{\underline{Step 3: Estimates of Non-Local Terms}}
The only remaining term in \eqref{ltt 01=} is the non-local term
\begin{align}
\\
I:=&\int_{0}^{\tm}\int_{\R^3}k_{\vh(\vv)}(\vv,\vuu)\int_0^{\tm'}\int_{\R^3}k_{\vh(\vuu)}(\vuu,\vuu')\reg\Big(t-\e^2s-\e^2r,\vx-\e(\tm-s)\vuu-\e(\tm'-r)\vuu',\vuu'\Big)\ue^{-\nu(\vuu)(\tm'-r)}\ue^{-\nu(\vv)
(\tm-s)}\ud\vuu'\ud r\ud\vuu\ud{s},\no
\end{align}which will estimated in five cases:
\begin{align}
I:=I_1+I_2+I_3+I_4+I_5.
\end{align}

\subparagraph{\underline{Case I: $I_1:$ $\abs{\vv}\geq N$}}
Based on Lemma \ref{lem:kernel-operator}, we have
\begin{align}
\abs{\int_{\r^3}\int_{\r^3}k_{\vh(\vv)}(\vv,\vuu)k_{\vh(\vuu)}(\vuu,\vuu')\ud{\vuu}\ud{\vuu'}}\ls\frac{1}{1+\abs{\vv}}\ls\frac{1}{N}.
\end{align}
Hence, we get
\begin{align}\label{ltt 02=}
\abs{I_1}\ls\frac{1}{N}\lnnm{\reg}.
\end{align}

\subparagraph{\underline{Case II: $I_2:$ $\abs{\vv}\leq N$, $\abs{\vuu}\geq2N$, or $\abs{\vuu}\leq
2N$, $\abs{\vuu'}\geq3N$}}
Notice this implies either $\abs{\vuu-\vv}\geq N$ or
$\abs{\vuu-\vuu'}\geq N$. Hence, either of the following is valid
correspondingly:
\begin{align}
\abs{k_{\vh(\vv)}(\vv,\vuu)}\leq& C\ue^{-\d N^2}\abs{k_{\vh(\vv)}(\vv,\vuu)}\ue^{\d\abs{\vv-\vuu}^2},\\
\abs{k_{\vh(\vuu)}(\vuu,\vuu')}\leq& C\ue^{-\d N^2}\abs{k_{\vh(\vuu)}(\vuu,\vuu')}\ue^{\d\abs{\vuu-\vuu'}^2}.
\end{align}
Based on Lemma \ref{lem:kernel-operator}, we know
\begin{align}
\int_{\r^3}\abs{k_{\vh(\vv)}(\vv,\vuu)}\ue^{\d\abs{\vv-\vuu}^2}\ud{\vuu}<&\infty,\\
\int_{\r^3}\abs{k_{\vh(\vuu}(\vuu,\vuu')}\ue^{\d\abs{\vuu-\vuu'}^2}\ud{\vuu'}<&\infty.
\end{align}
Hence, we have
\begin{align}\label{ltt 03=}
\abs{I_2}\ls \ue^{-\d N^2}\lnnm{\reg}.
\end{align}

\subparagraph{\underline{Case III: $I_3:$ $\tm'-r\leq\d$ and $\abs{\vv}\leq N$, $\abs{\vuu}\leq 2N$, $\abs{\vuu'}\leq 3N$}}
In this case, since the integral with respect to $r$ is restricted in a very short interval, there is a small contribution as
\begin{align}\label{ltt 04=}
\abs{I_3}\ls\abs{\int_{\tm'-\d}^{\tm'}\ue^{-(\tm'-r)}\ud{r}}\lnnm{\reg}\ls \d\lnnm{\reg}.
\end{align}

\subparagraph{\underline{Case IV: $I_4:$ $\tm'-r\geq \abs{\ln(\d)}$ and $\abs{\vv}\leq N$, $\abs{\vuu}\leq 2N$, $\abs{\vuu'}\leq 3N$}}
In this case, $\tm'-r$ is significantly large, so $\ue^{-(\tm'-r)}\leq\d$ is very small. Hence, the contribution is small 
\begin{align}\label{ltt 05=}
\abs{I_4}\ls\abs{\int_{\abs{\ln(\d)}}^{\infty}\ue^{-(\tm'-r)}\ud{r}}\lnnm{\reg}\ls \d\lnnm{\reg}.
\end{align}

\subparagraph{\underline{Case V: $I_5:$ $\d\leq \tm'-r\leq \abs{\ln(\d)}$ and $\abs{\vv}\leq N$, $\abs{\vuu}\leq 2N$, $\abs{\vuu'}\leq 3N$}}
This is the most complicated case. Since $k_{\vh(\vv)}(\vv,\vuu)$ has
possible integrable singularity of $\abs{\vv-\vuu}^{-1}$, we can
introduce the truncated kernel $k_N(\vv,\vuu)$ which is smooth and has compactly supported range such that
\begin{align}\label{ltt 07=}
\sup_{\abs{\vv}\leq 3N}\int_{\abs{\vuu}\leq
3N}\abs{k_N(\vv,\vuu)-k_{\vh(\vv)}(\vv,\vuu)}\ud{\vuu}\leq\frac{1}{N}.
\end{align}
Then we can split
\begin{align}
k_{\vh(\vv)}(\vv,\vuu)k_{\vh(\vuu)}(\vuu,\vuu')=&k_N(\vv,\vuu)k_N(\vuu,\vuu)
+\bigg(k_{\vh(\vv)}(\vv,\vuu)-k_N(\vv,\vuu)\bigg)k_{\vh(\vuu)}(\vuu,\vuu')\\
&+\bigg(k_{\vh(\vuu)}(\vuu,\vuu')-k_N(\vuu,\vuu')\bigg)k_N(\vv,\vuu).\no
\end{align}
This means that we further split $I_5$ into
\begin{align}
I_5:=I_{5,1}+I_{5,2}+I_{5,3}.
\end{align}
Based on \eqref{ltt 07=}, we have
\begin{align}\label{ltt 08=}
\abs{I_{5,2}}\ls&\frac{1}{N}\lnnm{\reg},\quad \abs{I_{5,3}}\ls\frac{1}{N}\lnnm{\reg}.
\end{align}
Therefore, the only remaining term is $I_{5,1}$. Note that we always have $\vx-\e(\tm-s)\vv-\e(\tm'-r)\vuu\in\Omega$. Hence, we define the change of variable $\vuu\rt y$ as
$y=(y_1,y_2,y_3)=\vx-\e(\tm-s)\vv-\e(\tm'-r)\vuu$. Then the Jacobian
\begin{align}
\abs{\frac{\ud{y}}{\ud{\vuu}}}=\abs{\left\vert\begin{array}{ccc}
\e(\tm'-r)&0&0\\
0&\e(\tm'-r)&0\\
0&0&\e(\tm'-r)
\end{array}\right\vert}=\e^3(\tm'-r)^3\geq \e^3\d^3.
\end{align}
Considering $\abs{\vv},\abs{\vuu},\abs{\vuu'}\leq 3N$, we know $\abs{\reg}\simeq\abs{\re}$. Also, since $k_N$ is bounded, we estimate
\begin{align}\label{ltt 09=}
\\
\abs{I_{5,1}}\ls&
\int_{\abs{\vuu}\leq2N}\int_{\abs{\vuu'}\leq3N}\int_{0}^{\tm'}
\id_{\{\vx-\e(\tm-s)\vv-\e(\tm'-r)\vuu\in\Omega\}}\abs{\re\Big(t-\e^2s-\e^2r,\vx-\e(\tm-s)\vv-\e(\tm'-r)\vuu,\vuu'\Big)}\ue^{-\nu(\vuu)
(\tm'-r)}\ud{r}\ud{\vuu}\ud{\vuu'}.\no
\end{align}
Using H\"older's inequality, we estimate
\begin{align}\label{ltt 10=}
\\
&\int_{\abs{\vuu}\leq2N}\int_{\abs{\vuu'}\leq3N}\int_{0}^{\tm'}
\id_{\{\vx-\e(\tm-s)\vv-\e(\tm'-r)\vuu\in\Omega\}}\abs{\re\Big(t-\e^2s-\e^2r,\vx-\e(\tm-s)\vv-\e(\tm'-r)\vuu,\vuu'\Big)}\ue^{-\nu(\vuu)
(\tm'-r)}\ud{r}\ud{\vuu}\ud{\vuu'}\no\\
\leq&\bigg(\int_{\abs{\vuu}\leq2N}\int_{\abs{\vuu'}\leq3N}\int_{0}^{\tm'}
\id_{\{\vx-\e(\tm-s)\vv-\e(\tm'-r)\vuu\in\Omega\}}\ue^{-\nu(\vuu)
(\tm'-r)}\ud{r}\ud{\vuu}\ud{\vuu'}\bigg)^{\frac{5}{6}}\no\\
&\times\bigg(\int_{\abs{\vuu}\leq2N}\int_{\abs{\vuu'}\leq3N}\int_{0}^{\tm'}
\id_{\{\vx-\e(\tm-s)\vv-\e(\tm'-r)\vuu\in\Omega\}}\abs{\re\Big(t-\e^2s-\e^2r,\vx-\e(\tm-s)\vv-\e(\tm'-r)\vuu,\vuu'\Big)}^{6}\ue^{-\nu(\vuu)
(\tm'-r)}\ud{r}\ud{\vuu}\ud{\vuu'}\bigg)^{\frac{1}{6}}\no\\
\ls&\abs{\int_{0}^{\tm'}\frac{1}{\e^3\d^3}\int_{\abs{\vuu'}\leq3N}
\int_{\Omega}\id_{\{y\in\Omega\}}\babs{\re(t-\e^2s-\e^2r,y,\vuu')}^6\ue^{-(\tm'-r)}\ud{y}\ud{\vuu'}\ud{r}}^{\frac{1}{6}}
\ls \frac{1}{\e^{\frac{1}{2}}\d^{\frac{1}{2}}}\nnm{\re}_{L^{\infty}_tL^6_{x,v}}.\no
\end{align}
Inserting \eqref{ltt 10=} into \eqref{ltt 09=}, we obtain
\begin{align}
\abs{I_{5,1}}\ls \frac{1}{\e^{\frac{1}{2}}\d^{\frac{1}{2}}}\nnm{\re}_{L^{\infty}_tL^6_{x,v}}.
\end{align}
Combined with \eqref{ltt 08=}, we know
\begin{align}\label{ltt 06=}
I_5\ls \frac{1}{N}\lnnm{\reg}+\frac{1}{\e^{\frac{1}{2}}\d^{\frac{1}{2}}}\nnm{\re}_{L^{\infty}_tL^6_{x,v}}.
\end{align}
Summarizing all five cases in \eqref{ltt 02=}\eqref{ltt 03=}\eqref{ltt 04=}\eqref{ltt 05=}\eqref{ltt 06=}, we obtain
\begin{align}
\abs{I}\ls \bigg(\frac{1}{N}+\ue^{-\d N^2}+\d\bigg)\lnnm{\reg}+\frac{1}{\e^{\frac{1}{2}}\d^{\frac{1}{2}}}\nnm{\re}_{L^{\infty}_tL^6_{x,v}}.
\end{align}
Choosing $\d$ sufficiently small and then taking $N$ sufficiently large, we have
\begin{align}
\abs{I}\ls \d\lnnm{\reg}+\frac{1}{\e^{\frac{1}{2}}\d^{\frac{1}{2}}}\nnm{\re}_{L^{\infty}_tL^6_{x,v}}.
\end{align}

\paragraph{\underline{Step 4: Synthesis}}
Summarizing all above, we obtain for any $(t,\vx,\vv)\in\rp\times\overline{\Omega}\times\R^3$,
\begin{align}
\babs{\reg(t,\vx,\vv)}\ls\d\lnnm{\reg}+\frac{1}{\e^{\frac{1}{2}}\d^{\frac{1}{2}}}\nnm{\re}_{L^{\infty}_tL^6_{x,v}}+\oot\e^{\frac{1}{2}}\xnm{\re}+\xnm{\re}^2+\oot.
\end{align}
Hence, when $\d\ll1$, we obtain
\begin{align}
\babs{\reg(t,\vx,\vv)}\ls\e^{-\frac{1}{2}}\nnm{\re}_{L^{\infty}_tL^6_{x,v}}+\oot\e^{\frac{1}{2}}\xnm{\re}+\xnm{\re}^2+\oot,
\end{align}
and thus the desired result follows.
\end{proof}


\subsection{Remainder Estimate}

\begin{theorem}\label{thm:priori=}
    Under the assumption
    \eqref{assumption:evolutionary1}\eqref{assumption:evolutionary2}\eqref{assumption:evolutionary3}, we have
    \begin{align}
        \xnnm{\re}\ls\oot.
    \end{align}
\end{theorem}
\begin{proof}
    Based on Proposition \ref{prop:energy=}, we have
    \begin{align}
        \nm{\re}_{L^{\infty}_tL^2_{xv}}+\e^{-\frac{1}{2}}\tnnms{\re}{\ga_+}+\e^{-1}\tnnm{\ire}\ls\oot\xnnm{\re}+\xnnm{\re}^2+\oot.
    \end{align}
    Based on Proposition \ref{prop:kernel=}, we have
    \begin{align}
        \e^{-\frac{1}{2}}\tnnm{\bre}\ls\oot\xnnm{\re}+\xnnm{\re}^{2}+\oot.
    \end{align}
    Combining both of them, we arrive at
    \begin{align}\label{priori 01=}
        \nm{\re}_{L^{\infty}_tL^2_{xv}}+\e^{-\frac{1}{2}}\tnnms{\re}{\ga_+}+\e^{-\frac{1}{2}}\tnnm{\bre}+\e^{-1}\tnnm{\ire}\ls \oot\xnnm{\re}+\xnnm{\re}^2+\oot.
    \end{align}
    Similarly, combining Proposition \ref{prop:energy=t} and Proposition \ref{prop:kernel=t}, we arrive at 
    \begin{align}\label{priori 01=t}
    \\
        \nm{\dt\re}_{L^{\infty}_tL^2_{xv}}+\e^{-\frac{1}{2}}\tnnms{\dt\re}{\ga_+}+\e^{-\frac{1}{2}}\tnnm{\dt\bre}+\e^{-1}\tnnm{\dt\ire}\ls \oot\xnnm{\re}+\xnnm{\re}^2+\oot.\no
    \end{align}
    Based on Proposition \ref{prop:energy=s} and Corollary \ref{lem:energy-extension=s}, we have
    \begin{align}
    \\
        \e^{-\frac{1}{2}}\nm{\re}_{L^{\infty}_tL^2_{\gamma_+}}+\nm{\m^{\frac{1}{4}}\re}_{L^{\infty}_tL^4_{\gamma_+}}+\e^{-1}\nnm{\ire}_{L^{\infty}_tL^2_{\nu}}+\nnm{\ire}_{L^{\infty}_tL^6_{xv}}
        \ls&\oot\xnnm{\re}+\xnnm{\re}^2+\oot.\no
    \end{align}
    Based on Proposition \ref{prop:kernel=s}, we have
    \begin{align}
        \nm{\bre}_{L^{\infty}_tL^6_{xv}}\ls\oot\xnnm{\re}+\xnnm{\re}^{2}+\oot.
    \end{align}
    Combining both of them, we arrive at
    \begin{align}
    \label{priori 02=}\\
        \e^{-\frac{1}{2}}\nm{\re}_{L^{\infty}_tL^2_{\gamma_+}}+\nm{\m^{\frac{1}{4}}\re}_{L^{\infty}_tL^4_{\gamma_+}}+\e^{-1}\nnm{\ire}_{L^{\infty}_tL^2_{\nu}}+\nnm{\re}_{L^{\infty}_tL^6_{xv}}
        \ls&\oot\xnnm{\re}+\xnnm{\re}^2+\oot.\no
    \end{align}
    Based on Proposition \ref{prop:infty=}, we have
    \begin{align}\label{priori 03=}
    \e^{\frac{1}{2}}\nnm{\re}_{L^{\infty}_tL^{\infty}_{\vrh,\vth}}+\e^{\frac{1}{2}}\nm{\re}_{L^{\infty}_tL^{\infty}_{\gamma_+}}+\e^{\frac{1}{2}}\lnnmm{\re}+\e^{\frac{1}{2}}\lnnmms{\re}{\ga_+}
    \ls&\oot\xnnm{\re}+\xnnm{\re}^{2}+\oot.
    \end{align}
    Collecting \eqref{priori 01=}\eqref{priori 01=t}\eqref{priori 02=}\eqref{priori 03=}, we have
    \begin{align}
    \xnnm{\re}\ls\oot\xnnm{\re}+\xnnm{\re}^2+\oot.
    \end{align}
    Hence, we have
    \begin{align}
        \xnnm{\re}\ls\xnnm{\re}^2+\oot.
    \end{align}
    By a standard iteration/fixed-point argument, our desired result follows.
\end{proof}

\begin{proof}[Proof of Theorem \ref{main theorem=}]
    The estimate \eqref{main'=} follows from Theorem \ref{thm:priori=}. The construction and positivity of $\fs$ based on the expansion \eqref{expand=} is standard and we refer to \cite{AA023, Esposito.Guo.Kim.Marra2015}, so we will focus on the proof of \eqref{main=}. From Theorem \ref{thm:priori=}, we have
    \begin{align}
        \e^{-\frac{1}{2}}\tnnm{\bre}+\e^{-1}\tnnm{\ire}\ls \oot,
    \end{align}
    which yields
    \begin{align}
        \tnnm{\re}\ls\oot\e^{\frac{1}{2}}.
    \end{align}
    From \eqref{expand=}, we know
    \begin{align}
        \tnnm{\mhh\fs-\mh-\e\f_1-\e^2\f_2-\e\fb_1}=\tnnm{\e\re}\ls\oot\e^{\frac{3}{2}}.
    \end{align}
    From Theorem \ref{thm:approximate-solution=} and the rescaling $\eta=\e^{-1}\mn$, we have
    \begin{align}
        \tnnm{\e^2\f_2}\ls\oot\e^2,\quad \tnnm{\e\fb_1}\ls\oot\e^{\frac{3}{2}}.
    \end{align}
    Hence, we have
    \begin{align}
        \tnnm{\mhh\fs-\mh-\e\f_1}\ls\oot\e^{\frac{3}{2}}.
    \end{align}
    Therefore, \eqref{main=} follows.
\end{proof}


\bigskip
\appendix

\makeatletter
\renewcommand \theequation {%
A.%
\ifnum\c@subsection>\z@\@arabic\c@subsection.%
\fi
\@arabic\c@equation} \@addtoreset{equation}{section}
\@addtoreset{equation}{subsection} \makeatother

\section{Linearized Boltzmann Operator}

Based on \cite[Chapter 7]{Cercignani.Illner.Pulvirenti1994} and \cite[Chapter 1\&3]{Glassey1996}, define the symmetrized version of $Q$
\begin{align}
Q^{\ast}[F,G]:=\frac{1}{2}\iint_{\r^3\times\R^3}q(\vo,\abs{\vuu-\vv})\Big(F(\vuu_{\ast})G(\vv_{\ast})+F(\vv_{\ast})G(\vuu_{\ast})-F(\vuu)G(\vv)-F(\vv)G(\vuu)\Big)\ud{\vo}\ud{\vuu}.
\end{align}
Clearly, $Q[F,F]=Q^{\ast}[F,F]$.
Denote the linearized Boltzmann operator $\lc$
\begin{align}\label{linearized}
\lc[f]:=&-2\m^{-\frac{1}{2}}Q^{\ast}\left[\m,\m^{\frac{1}{2}}f\right]:=\nu f-K[f],
\end{align}
where
for some kernels $k(\vuu,\vv)$,
\begin{align}
\nu(\vv)=\int_{\r^3}\int_{\S^2}q(\vo,\abs{\vuu-\vv})\m(\vuu)\ud{\vo}\ud{\vuu},\quad
K[f](\vv)=\int_{\r^3}k(\vuu,\vv)f(\vuu)\ud{\vuu}.
\end{align}
$\lc$ is self-adjoint in $L^2_{\nu}(\r^3)$ satisfying the coercivity property 
\begin{align}\label{coercivity}
    \int_{\r^3}f\lc[f]\gs \um{(\ik-\pk)[f]}^2.
\end{align} 
Denote $\li: \nnk\rt\nnk$ the quasi-inverse of $\lc$.
Also, denote the nonlinear Boltzmann operator $\Gamma$
\begin{align}
\Gamma[f,g]:=\mhh Q^{\ast}\left[\mh f,\mh g\right]\in\nnk.
\end{align}

\makeatletter
\renewcommand \theequation {%
B.%
\ifnum\c@subsection>\z@\@arabic\c@subsection.%
\fi
\@arabic\c@equation} \@addtoreset{equation}{section}
\@addtoreset{equation}{subsection} \makeatother

\bigskip
\section{Inner Products and Norms}

Based on the flow direction, we can divide the boundary $\gamma:=\big\{(\vx_0,\vv):\ \vx_0\in\p\Omega,\vv\in\r^3\big\}$ into the incoming boundary $\gamma_-$, the outgoing boundary $\gamma_+$, and the grazing set $\gamma_0$ based on the sign of $\vv\cdot\vn(\vx_0)$.
Similarly, we further divide the boundary $\ga:=\big\{(t,\vx_0,\vv):\ t\in\rp, \vx_0\in\p\Omega,\vv\in\r^3\big\}$ into $\ga_-$, $\ga_+$, and $\ga_0$.

Let $\brv{\ \cdot\ ,\ \cdot\ }$ denote the inner product in $v\in\r^3$, $\brx{\ \cdot\ ,\ \cdot\ }$ the inner product in $x\in\Omega$, $\br{\ \cdot\ ,\ \cdot\ }$ the inner product in $(x,v)\in\Omega\times\r^3$. Also, let $\brb{\ \cdot\ ,\ \cdot\ }{\gamma_{\pm}}$ denote the inner product on $\gamma_{\pm}$ with measure $\ud\gamma:=\abs{v\cdot n}\ud v\ud S_x\ud s$. 

When time integral is involved, we define the corresponding inner products $\dbrv{\ \cdot\ ,\ \cdot\ }$, $\dbrx{\ \cdot\ ,\ \cdot\ }$, $\dbr{\ \cdot\ ,\ \cdot\ }$ and $\dbrb{\ \cdot\ ,\ \cdot\ }{\ga_{\pm}}$.

Denote the bulk and boundary norms
\begin{align}
    \pnm{f}{r}:=\left(\iint_{\Omega\times\r^3}\abs{f(x,v)}^r\ud v\ud x\right)^{\frac{1}{r}},\qquad \pnms{f}{r}{\gamma_{\pm}}:=\left(\int_{\gamma_{\pm}}\abs{f(x,v)}^r\abs{v\cdot n}\ud v\ud x\right)^{\frac{1}{r}}.
\end{align}
Define the weighted $L^{\infty}$ norms for $0\leq\varrho<\dfrac{1}{2}$ and $\vartheta\geq0$
\begin{align}
    \lnmm{f}:=\esssup_{(x,v)\in\Omega\times\r^3}\bigg(\br{v}^{\vartheta}\ue^{\varrho\frac{\abs{v}^2}{2}}\abs{f(x,v)}\bigg),\qquad
    \lnmms{f}{\gamma_{\pm}}:=\esssup_{(x,v)\in\gamma_{\pm}}\bigg(\br{v}^{\vartheta}\ue^{\varrho\frac{\abs{v}^2}{2}}\abs{f(x,v)}\bigg).
\end{align}
Denote the $\nu$-norm
\begin{align}
    \um{f}:=\left(\iint_{\Omega\times\r^3}\nu(v)\abs{f(x,v)}^2\ud v\ud x\right)^{\frac{1}{2}}.
\end{align}
When time variable is involved, we define the corresponding norms: $\pnnm{f}{r}$, $\pnnms{f}{r}{\ga_{\pm}}$, $\lnnmm{f}$, $\lnnmms{f}{\ga_{\pm}}$ and $\unm{f}$.

Let $\nm{\cdot}_{W^{k,p}L^q}$ denote $W^{k,p}$ norm for $x\in\Omega$ and $L^q$ norm for $v\in\r^3$, and $\nnm{\cdot}_{W^{m,s}W^{k,p}L^q}$ denote $W^{m,s}$ norm for $t\in[0,\tz]$ with some constant time $\tz>0$, $W^{k,p}$ norm for $x\in\Omega$ and $L^q$ norm for $v\in\r^3$. The similar notation also applies when we replace $L^q$ by $L^{\infty}_{\varrho,\vartheta}$, $W^{1,\infty}_{\varrho,\vartheta}$ or $L^q_{\gamma}$. We will only write the variables $(t,x,v)$ explicitly when there is possibility of confusion.

\makeatletter
\renewcommand \theequation {%
C.%
\ifnum\c@subsection>\z@\@arabic\c@subsection.%
\fi
\@arabic\c@equation} \@addtoreset{equation}{section}
\@addtoreset{equation}{subsection} \makeatother

\bigskip
\section{Symbols and Constants}\label{sec:notation}

Define the quantities
\begin{align}\label{final 22}
    &\ab:=v\cdot\left(\abs{v}^2-5T\right)\mh\in\r^3,\quad \a:=\li\left[\ab\right]\in\r^3,\\
    &\bbb=\bigg(v\otimes v-\frac{\abs{v}^2}{3}\mathbf{1}\bigg)\mh\in\r^{3\times3},\quad \b=\lc^{-1}\left[\bbb\right]\in\r^{3\times3},
\end{align}
with
\begin{align}\label{final 23}
    &\k\id:=\int_{\r^3}\left(\a\otimes\ab\right)\ud v,\quad \sigma\id:=\int_{\r^3}\left(\abs{v}^2-5T\right)\left(\a\otimes\ab\right)\ud v,\\
    &\lambda:=\int_{\r^3}\b_{ij}\bbb_{ij},\quad\alpha:=\int_{\r^3}\b_{ii}\bbb_{ii},\quad \gamma:=\int_{\r^3}\b_{ii}\bbb_{jj}\ \ \text{for}\ \ i\neq j.
\end{align}

Throughout this paper, $C>0$ denotes a constant that only depends on
the domain $\Omega$, but does not depend on the data or $\e$. It is
referred as universal and can change from one inequality to another.
When we write $C(z)$, it means a certain positive constant depending
on the quantity $z$. We write $a\ls b$ to denote $a\leq Cb$ and $a\gs b$ to denote $a\geq Cb$. Also, we write $a\simeq b$ if $a\ls b$ and $a\gs b$.

In this paper, we will use $\oo$ to denote a sufficiently small constant independent of the data. Also, let $\oot$ be a small constant. For stationary problem, $\oot$ depends on $\fss_b$ only satisfying
\begin{align}\label{def:oot-s}
    \oot=\oo\rt0\ \ \text{as}\ \ \ass\rt0.
\end{align}
For evolutionary problem, $\oot$ depends on $\fss_i$ and $\fss_b$ satisfying
\begin{align}\label{def:oot}
    \oot=\oo\rt0\ \ \text{as}\ \ \ase\rt0.
\end{align}
In principle, while $\oot$ is determined by data a priori, we are free to choose $\oo$ depending on the estimate.

\bigskip

\bibliographystyle{siam}
\bibliography{Reference}

\end{document}